\documentclass[a4paper, 10pt]{article}
\usepackage{amsmath}
\numberwithin{equation}{section}
\usepackage{amssymb,esint,color,hyperref}
\usepackage{amscd}
\usepackage{xspace}
\usepackage{fancyhdr}
\usepackage{color}
\usepackage{verbatim}
\usepackage{graphicx}
\usepackage{cite}
\usepackage{stmaryrd}
\usepackage{bbm}
\usepackage{xcolor}
\usepackage{mathrsfs}
\usepackage{srcltx} 

\usepackage{marginnote}
\let\oldmarginpar\marginpar
\renewcommand\marginpar[1]{\-\oldmarginpar[\raggedleft\footnotesize #1]%
	{\raggedright\footnotesize\color{red} #1}} 
\newtheorem{theorem}{Theorem}[section]

\newtheorem{lemma}[theorem]{Lemma}

\newtheorem{remark}[theorem]{Remark}

\def\da{\mathsf{Data}}
\def\B{\mathsf{B}}
\def\A{\mathsf{A}}
\def\M{\mathsf{M}}

\def\G{\mathsf{G}}
\def\H{\mathsf{H}}

\def\K{\mathsf{K}}
\def\L{\mathsf{L}}

\def\T{\mathsf{T}}

\def\b{\mathbf{b}}

\def\HH{\dot{\mathbf{H}}}
\newcommand{\eps}{\varepsilon}
\newenvironment{proof}[1][Proof]{\textbf{#1.} }{\hfill\rule{0.5em}{0.5em}}
{\catcode`\@=11\global\let\AddToReset=\@addtoreset
	\AddToReset{equation}{section}
	
	\AddToReset{theorem}{section}

	\title{Schauder-type Estimates and Log-Critical Well-posedness for the Two-Phase Muskat Problem with Surface Tension}
	\author{
		{{\bf Ke Chen,\thanks{E-mail address: k1chen@polyu.edu.hk, Department of Applied Mathematics, The Hong Kong Polytechnic University, Kowloon, Hong Kong, PR China.}
				~~Ruilin Hu\thanks{E-mail address: huruilin16@mails.ucas.ac.cn, Academy of Mathematics and Systems Science, Chinese Academy of Sciences, Beijing, 100190, China.},
				~~Quoc-Hung Nguyen\thanks{E-mail address: qhnguyen@amss.ac.cn, Academy of Mathematics and Systems Science, Chinese Academy of Sciences, Beijing, 100190, China.}}}}
	
\begin{document}
		\maketitle
        \begin{abstract}
        We prove short-time well-posedness for the Muskat problem with surface tension in the full two-phase setting, allowing different viscosities, arbitrary density contrast, and rigid boundaries. In particular, no Rayleigh--Taylor sign condition on the density contrast is imposed. The interface is assumed to be a graph, uniformly separated from the fixed boundaries, and the initial data may be large in the log-critical class $\dot C^{1,\log^\varkappa}\cap H^1$, with $\varkappa>1$. Thus the result reaches the natural Lipschitz threshold up to a logarithmic correction.
        
The main difficulty is that, in the presence of viscosity jump and boundaries, the interface equation is not given by a closed explicit contour dynamics law. Instead, the normal velocity is recovered through an elliptic transmission problem in moving domains, and the resulting evolution is a genuinely nonlocal quasilinear equation. We derive sharp Schauder-type estimates, adapted to the log-critical scale, for the transmission operators generated by the bulk Darcy flow. These estimates identify the third-order parabolic mechanism produced by surface tension and control the nonlinear coupling between the interface geometry and the elliptic transmission structure.

The proof builds on the Schauder framework developed in Part~I of this series \cite{CHN1}, but requires a new analysis of the Muskat transmission problem in moving domains. Combining this elliptic theory with the contour formulation and time-weighted Hölder estimates, we obtain existence, uniqueness, smoothing, and stability for large interfaces in arbitrary dimension.
\end{abstract}

		\section{Introduction}\label{secintro}
In this paper, we study the Muskat problem with surface tension in the general two-phase setting, allowing both viscosity jump and rigid boundaries. The interface is represented as the graph of a function $\eta(t,x)$, and we address the Cauchy problem for large initial data near the critical Lipschitz threshold. Our main result establishes short-time well-posedness for initial interfaces in the logarithmically refined critical space
\[
\dot C^{1,\log^\varkappa}\cap H^1,\qquad \varkappa>1,
\]
under the natural geometric assumption that the interface remains uniformly separated from the fixed boundaries.

		The Muskat problem \cite{Mus34} describes the evolution of the interface between two immiscible fluids in a porous medium. This model has significant applications in various fields including petroleum engineering, groundwater hydrology, and environmental sciences. When the interface is represented as a graph of a time-dependent function $\eta(t,x)$, the system can be formulated as a free boundary problem governed by Darcy's law with appropriate boundary conditions at the interface and rigid boundaries.
		
		In the physical configuration, consider two fluids with densities $\rho^{\pm}$ and viscosities $\mu^{\pm}$ occupying domains $\Omega_t^{\pm}$ separated by an interface $\Sigma_t = \{(x,\eta(t,x)) : x \in \mathbb{R}^d\}$.
		\begin{figure}[h]\label{figure1}
			\centering
			\includegraphics[width=6cm,height=6cm]{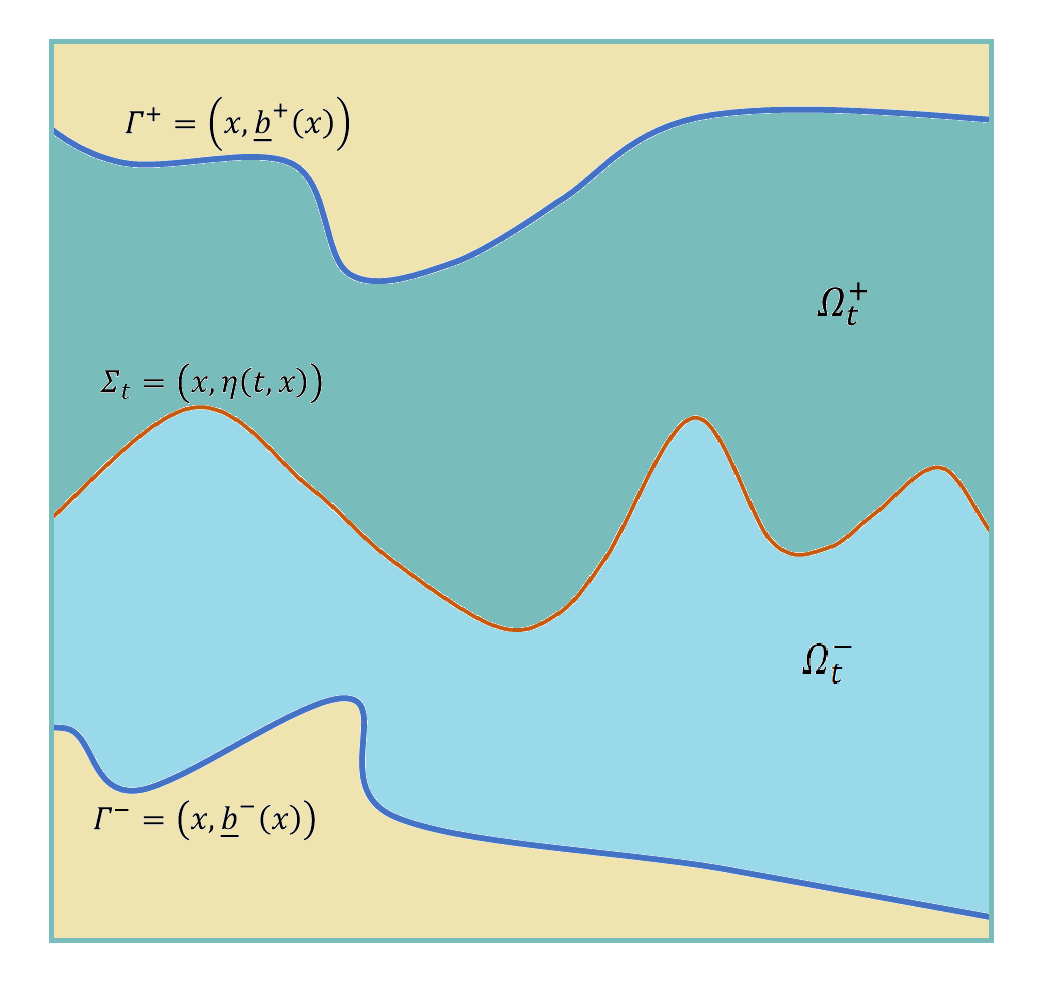}
			\caption{The Muskat problem}
		\end{figure}
        
		The domains are given by
		\begin{align*}
			&\Omega_{t}^{+}=\left\{(x, z) \in \mathbb{R}^{d} \times \mathbb{R}: \eta(t, x)<z<\underline{b}^{+}(x)\right\}, \\
			&\Omega_{t}^{-}=\left\{(x, z) \in \mathbb{R}^{d} \times \mathbb{R}: \underline{b}^{-}(x)<z<\eta(t, x)\right\},
		\end{align*}
		where $\underline{b}^\pm\in W^{1,\infty}(\mathbb{R}^d)$ are the parameterizations of the rigid boundaries
		$$\Gamma^{\pm}=\left\{\left(x, \underline{b}^{\pm}(x)\right): x \in \mathbb{R}^{d}\right\}.$$
		The incompressible fluid velocity $u^\pm$ in each region is governed by Darcy’s law \cite{DF96}
		\begin{align*}
			\mu_{\pm} u^{\pm}+\nabla_{x, z} p^{\pm}=-\rho^{\pm} \mathfrak{g} \vec{e}_{d+1}, \quad\quad\quad \operatorname{div}_{x, z} u^{\pm}=0 \quad \text { in } \Omega^\pm(t).
		\end{align*}
		Here $p^\pm$ is the pressure, and $\mathfrak{g}$ is the Earth’s gravity. In this section, we simply denote $\nabla=\nabla_x$, and $\nabla_{x,z}=(\nabla,\partial_z)$.
		The normal velocity is continuous at the interface
		\begin{align*}
			u^+\cdot n=	u^-\cdot n,
		\end{align*}
		where $n=\langle\nabla \eta\rangle^{-1}(-\nabla \eta, 1)$ and $\langle a\rangle=\sqrt{1+|a|^2}$.
		The pressure jump at the interface is proportional to the mean curvature, which incorporates the effect of surface tension (capillarity):
		\begin{align*}
			p^--p^+=-\sigma\operatorname{div}\left(\frac{\nabla\eta}{\langle\nabla \eta\rangle}\right),\ \ \ \text{on}\ \Sigma_t.
		\end{align*}
		For simplicity in the subsequent formulation, we assume the surface tension coefficient $\sigma=1$. The interface moves with the fluid
		\begin{equation*}
			\partial_{t} \eta=\left.\langle\nabla \eta\rangle u^{+} \cdot n\right|_{\Sigma_{t}}= \left. \mathbf{n}\cdot u^{+}\right|_{\Sigma_{t}},
		\end{equation*}
		where we denote $\mathbf{n}=\langle\nabla \eta\rangle n$.
		At the two rigid boundaries, the no-penetration boundary conditions are imposed
		\begin{align*}
			u^\pm\cdot \nu^\pm=0,\ \ \text{on}\ \Gamma^\pm,
		\end{align*}
		where $\nu^\pm=\frac{1}{\langle\nabla \underline{b}^\pm\rangle}(-\nabla\underline{b}^\pm,1)$.
		
Introducing the modified pressures
\begin{align*}
&q^\pm(x,z)=p^\pm(x,z)+\rho^\pm \mathfrak g z,\\&
\varrho_0=(\rho^- - \rho^+)\mathfrak g,
\end{align*}
the system can be rewritten as the elliptic transmission problem
		\begin{equation}\label{sysq}
			\begin{aligned}
				&\Delta_{x, z} q^{\pm}=0, \ \ \text{in}\ \Omega^\pm(t),\\
				&q^--q^+=-\operatorname{div}\left(\frac{\nabla \eta}{\langle\nabla \eta\rangle}\right)+\varrho_0\eta,\ \ \text{on}\ \Sigma_t,\\
				&\mathbf{n}\cdot\nabla_{x,z}\left(\frac{1}{\mu_+}q^+-\frac{1}{\mu_-}q^-\right) =0,\ \ \text{on}\ \Sigma_t,\\
				& \nu^\pm\cdot\nabla_{x,z}q^\pm=0,\ \ \text{on}\ \Gamma^\pm.
			\end{aligned} 
		\end{equation}
		The evolution equation of the interface reads
		\begin{align}\label{evo}
			\partial_t\eta=-\frac{1 }{\mu_+}\left.\mathbf{n}\cdot\nabla_{x,z}q^+\right|_{\Sigma_t}.
		\end{align}
        In the infinite-depth case, where no rigid boundaries are present, the graph equation enjoys a natural scaling symmetry:
        $$
\eta(t,x)\mapsto \eta_\lambda(t,x)
:=\lambda^{-1}\eta(\lambda^3 t,\lambda x),
\qquad \lambda>0.
$$
Consequently, both
$
\dot W^{1,\infty}(\mathbb R^d)$  and 
$
\dot H^{1+d/2}(\mathbb R^d)
$
are scale-invariant spaces for the interface. 

        \subsection{Related works}
        The well-posedness and regularity theory for the Muskat problem has been extensively studied under a variety of physical assumptions and analytic frameworks.

    For the two-phase Muskat problem,     in the absence of surface tension, the stability of the problem is governed by the Rayleigh--Taylor condition, which, in its simplest form, requires the denser fluid to lie below the lighter one \cite{Escher12}. If this condition fails, the interface may become unstable, leading to ill-posedness or finite-time singularity formation \cite{SCH04}. Even in initially stable configurations, smooth solutions can develop vertical tangents, known as ``turning waves'', and subsequently lose the Rayleigh--Taylor stability condition in finite time \cite{Castro2012, Castro2013}. We refer to the surveys of Gancedo~\cite{Gancedo} and Granero-Belinch\'on--Lazar~\cite{BelinchLazar} for further discussion of singularity formation.

In the stable regime without surface tension, global solutions have been constructed in various settings; see, for instance, \cite{Cameron2019,Cameron2020,CBS,1Constantin2010,CGSV17,CCG16,DengLei2017,Cordoba2011,Cordoba2013}. Foundational local well-posedness results include the works of Yi~\cite{Yi2003}, Ambrose~\cite{Ambrose2004, Ambrose2007}, C\'ordoba and Gancedo~\cite{Cordoba2007Contour, Cordoba2009}, and Cheng, Granero-Belinch\'on, and Shkoller~\cite{CBS}. More recently, local well-posedness in subcritical spaces was obtained by Constantin--Gancedo--Shvydkoy--Vicol~\cite{CGSV17} in Sobolev spaces $W^{2,p}(\mathbb R)$, by Abels--Matioc~\cite{Ables-Matioc} in $W^{s,p}$ with $s>1+1/p$, by Matioc~\cite{Matioctams,1Matioc2018}, and by Alazard--Lazar~\cite{Alazard-Lazar} in $H^s(\mathbb R)$ with $s>3/2$.

We next recall several developments at critical regularity. The first results in this direction concerned small-data solutions in Wiener-type spaces. Constantin, C\'ordoba, Gancedo, and Strain~\cite{1Constantin2010} established global small-data solutions without viscosity jump, and this was later extended to three dimensions in joint work with Piazza~\cite{CCG16}. Gancedo, Garc\'ia-Ju\'arez, Patel, and Strain~\cite{GGPS} further developed this theory and constructed global small-data strong solutions allowing viscosity jump in both two and three dimensions.

There has also been substantial progress in critical Sobolev spaces. C\'ordoba and Lazar~\cite{Cordoba-Lazar-H3/2}, while working in the subcritical space $\dot H^{5/2}$, derived a priori estimates at the critical level $\dot H^{3/2}$. This was later extended to three dimensions by Gancedo and Lazar~\cite{1Gancedo2020Global}. The first fully critical result was obtained by Alazard and the third author, who constructed two-dimensional solutions with initial data in $H^{3/2}\cap \dot W^{1,\infty}$ \cite{Alazard2020}; see also the log-subcritical work \cite{ThomasHfirst}, where unbounded slopes are allowed. Depending on the size of the data, these results yield either local well-posedness for large data or global well-posedness for small data. In subsequent work, the $L^\infty$ slope assumption was removed, yielding well-posedness directly in $H^{3/2}$ \cite{Alazard2020endpoint}. This analysis was later extended to three dimensions in $\dot H^2\cap W^{1,\infty}$ \cite{TH4}.

Further critical small-data results include the works of Cameron, who first studied well-posedness in $\dot W^{1,\infty}\cap L^2$ in two dimensions \cite{Cameron2019}, and later in three dimensions under merely sublinear growth at infinity \cite{Cameron2020}. A notable feature of Cameron's results is that they allow medium-sized data, in the sense that the slope is bounded by $1$ rather than by a perturbatively small constant. Huy Nguyen~\cite{HuyNguyen2021} established well-posedness in the Besov space $\dot B_{\infty,1}^1$, which embeds into the critical Lipschitz class $\dot W^{1,\infty}$. For genuinely large data, global regularity is not expected in general due to possible singularity formation. In one space dimension, under a monotonicity assumption on the initial interface, Deng--Lei--Lin~\cite{DengLei2017} constructed global weak solutions. Local well-posedness in the critical H\"older space $C^1$ was established by the first and third authors together with Xu~\cite{KeC1}. More recently, \cite{CCHNX} obtained local existence of classical solutions in regimes where the product of maximal and minimal slopes is allowed to be large in arbitrary dimension. We also mention the work of Garc\'ia-Ju\'arez, G\'omez-Serrano, Huy Nguyen, and Pausader~\cite{GGNP} on small self-similar solutions emanating from exact corner singularities, as well as \cite{GGHP} on the desingularization of small moving corners.

	The inclusion of surface tension introduces the mean-curvature term, which yields a third-order parabolic mechanism for graph interfaces. We recall the results most relevant to the surface-tension problem.  Huy Nguyen \cite{NguAdv} and Huy Nguyen--Pausader~\cite{1HuyNguyen2020}  proved local well-posedness in any subcritical Sobolev spaces $H^s(\mathbb{R}^d)$ with $s > 1 + \frac{d}{2}$, considering both one fluid or two fluids settings, with or without viscosity jump, with or without rigid boundaries.  The proof employed a paradifferential approach for the Dirichlet-Neumann operator, see also \cite{ABZ11,ABZ14} for related studies in the context of water waves. In two dimensions, A.-V. Matioc and B.-V. Matioc \cite{MATIOC2022} introduced a reformulation based on layer-potential operators and obtained local well-posedness with surface tension and general viscosities in subcritical $L^p$ based Sobolev spaces. More recently, Lazar \cite{Lazar24} proved a global well-posedness result for the 2D surface-tension problem without viscosity jump, with any regular enough initial data which is small in some critical space but possibly large in Lipschitz.  The general two-phase problem with viscosity jump and rigid boundaries leads to a different difficulty, since the normal velocity is recovered through an elliptic transmission problem rather than an explicit closed contour equation. We also mention the recent work of Wan and Yang
\cite{WanYangElasticMuskat}, where the two-dimensional Muskat problem with
an elastic interface was studied via a Dirichlet--Neumann formulation.
Their model exhibits a higher-order quasilinear parabolic structure, in
contrast to the surface-tension-driven problem considered here.

   For the one-phase Muskat problem, a remarkable global theory for large
Lipschitz graphs was recently developed by Dong, Gancedo and Huy Nguyen \cite{DGN2D}.
In two space dimensions, they proved global well-posedness for arbitrary
periodic Lipschitz initial interfaces by combining quantitative estimates
for layer potentials with pointwise elliptic regularity in Lipschitz
domains. This theory was subsequently extended to the three-dimensional
one-phase problem \cite{DGN3D}, where the analysis relies on delicate structural
properties of the Dirichlet-to-Neumann map and layer potentials in
Lipschitz graph domains. In the presence of surface
tension, Dong and Kwon \cite{DKSurfaceTension} recently proved global
well-posedness for small one-phase data in $H^s$ with  $s>d/2+1$.

In the present work, we study the Muskat problem with surface tension in the general two-phase setting illustrated in Figure~\ref{figure1}, allowing viscosity jump and rigid boundaries. Our aim is to develop a well-posedness theory near the critical Lipschitz threshold for large interfaces, a delicate issue already highlighted in \cite{1HuyNguyen2020}.

The main difficulty is that the interface equation \eqref{evo} is not a closed local equation for $\eta$. Rather, the normal velocity is determined implicitly through the elliptic transmission problem \eqref{sysq} in moving domains. Consequently, the resulting interface operator is both nonlocal and quasilinear. In the presence of viscosity jump and rigid boundaries, the structural simplifications available in the equal-viscosity contour-dynamics formulation are no longer present. The key task is therefore to obtain sharp regularity estimates for the transmission operators associated with the bulk Darcy flow and to couple these estimates with the interface evolution.

Our analysis is rooted in the Schauder framework developed in Part~I of this series~\cite{CHN1}, where a general theory was established for parabolic equations governed by nonlocal pseudo-differential operators. We also refer to the work of Constantin, Ignatova, and the third author~\cite{CIN}, where a related freezing-coefficient strategy was used to study the critical SQG equation in bounded domains and to establish global regularity. The present work, however, requires several new ingredients specific to the Muskat transmission problem. We first recast the free-boundary system as a contour-dynamics equation for the interface, in which the normal velocity is determined implicitly through an elliptic transmission problem in moving domains. We then identify the third-order parabolic principal part generated by surface tension and prove Schauder-type estimates, in the logarithmically refined critical scale, for the elliptic transmission operators associated with the bulk Darcy flow. These estimates allow us to close the nonlinear argument in the class\footnote{On $\mathbb T^d$, the condition
$h\in \dot C^{1,\log^\varkappa}\cap H^1$ is equivalent to
$h\in C^{1,\log^\varkappa}$, since the $H^1$ norm controls the
zeroth-order part and periodicity removes the homogeneous ambiguity.}
\[
\dot C^{1,\log^\varkappa}\cap H^1,
\qquad \varkappa>1,
\]
and to prove short-time well-posedness in arbitrary dimension.

		\subsection{Main result}
        
		For $\beta\in[0,1),\varkappa>0$, $k\in\mathbb{N}$, define
		\begin{equation}\label{ctlog}
			\begin{aligned}
				&\|h\|_{\dot C^{\beta,\log^\varkappa}}:=\sup_{x\neq y}\frac{|h(x)-h(y)|}{\log^{-\varkappa}\left(2+\frac{1}{|x-y|}\right)|x-y|^{\beta}},\\
				&\|h\|_{\dot C^{k+\beta,\log ^\varkappa}}=\|\nabla^k h\|_{\dot C^{\beta,\log^\varkappa}}.
			\end{aligned}
		\end{equation}
      Since
\[
\log^{-\varkappa}\left(2+\frac{1}{|x-y|}\right)
\leq C_\varkappa,
\]
we obtain
\[
\|h\|_{\dot C^\beta}
\lesssim_\varkappa
\|h\|_{\dot C^{\beta,\log^\varkappa}}.
\]
Moreover,
\begin{align*}
\log^{-\varkappa}\left(2+\frac{1}{|x-y|}\right)\to 0~~\text{as}~~|x-y|\to 0,
\end{align*}
so $\dot C^{\beta,\log^\varkappa}$ is a logarithmic refinement of the usual Hölder space
$\dot C^\beta$ at small scales. In particular, the following inclusion is strict:
\[
\dot C^{\beta,\log^\varkappa}\subsetneq \dot C^\beta .
\]
We also use the shorthand notation
		$$\|h\|_{\dot C^{\log^\varkappa}}=\|h\|_{\dot C^{0,\log^\varkappa}}.$$ 
        For later use, we record that, for every $\varkappa>1,$
		\begin{equation*}
			\int_{|\alpha|<1}\|\delta_\alpha f\|_{L^\infty}\frac{d\alpha}{|\alpha|^d}\lesssim\frac{1}{\varkappa-1}\|f\|_{\dot C^{\log^\varkappa}}.
		\end{equation*}
We now state the main result of the paper.
       \begin{theorem}\label{thmgm}
Let $\varkappa>1$, $n\in\mathbb N^+$. Assume that
\[
\eta_0\in H^1(\mathbb R^d)\cap
\dot C^{1,\log^{2\varkappa-1}}(\mathbb R^d)
\]
satisfies the separation condition
\[
\operatorname{dist}(\eta_0,\Gamma^\pm)>2\mathbf r>0.
\]
Then there exists a time
\[
T=T(\eta_0,\varkappa,\mathbf r,m)>0,
\]
such that the system \eqref{sysq}--\eqref{evo} admits a unique solution
\[
\eta\in C([0,T];H^1(\mathbb R^d)\cap\dot C^{1}(\mathbb R^d)).
\]Moreover, this solution satisfies
\begin{align}
\sup_{t\in[0,T]}
\left(
\|\eta(t)\|_{H^1\cap\dot C^{1,\log^\varkappa}}
+
t^{\frac{m}{3}}|\log t|^\varkappa
\|\nabla\eta(t)\|_{\dot C^{m}}
\right)
&\leq
C\|\eta_0\|_{H^1\cap\dot C^{1,\log^{\varkappa}}},
\label{fes2}\\
\inf_{t\in[0,T]}\operatorname{dist}(\eta(t),\Gamma^\pm)
&>\mathbf r,
\label{dist*}
\end{align}
for some constant $C>0$.
\end{theorem}

\begin{remark}
Our method also applies, with only minor modifications, to the Muskat problem without surface tension in the Rayleigh--Taylor stable regime. In this setting, the curvature term is absent and the leading-order dynamics are governed by a first-order parabolic operator generated by the density contrast through the elliptic transmission problem. Combining the freezing-coefficient argument with the corresponding first-order Schauder estimates yields local well-posedness for initial data in
\[
H^1\cap C^{1,\log^{2\varkappa-1}},
\qquad \varkappa>1,
\]
under the same separation condition.
\end{remark}
The estimate \eqref{fes2} reflects the regularizing effect of surface tension, while \eqref{dist*} guarantees that the interface $\Sigma_t$ remains uniformly separated from the rigid boundaries $\Gamma^\pm$ (see Figure~\ref{figure1}).

The regularity assumption in Theorem~\ref{thmgm} is logarithmically critical. More precisely, compared with the critical Lipschitz class, we require the additional logarithmic regularity
\[
f_0\in \dot C^{1,\log^{2\varkappa-1}},\qquad \varkappa>1.
\]
This condition is needed to control singular integral terms generated by the viscosity jump. Our method does not appear to close in the pure Lipschitz class.

To illustrate the difficulty, consider the two-dimensional Muskat problem without rigid boundaries, which admits the contour dynamics formulation (see \cite{Matioctams})

		\begin{equation}\label{mus2d}
			\partial_t f(x) = \int_{\mathbb{R}}
			\frac{1 + \partial_x f(x) \Delta_\alpha f(x)}{\left\langle \Delta_\alpha f(x) \right\rangle^2}
			\omega(x - \alpha) \frac{d\alpha}{\alpha},
		\end{equation}
where
\[
\Delta_\alpha f(x)
=
\frac{f(x)-f(x-\alpha)}{\alpha},
\]
and the vorticity density $\omega$ is determined by
\begin{align}
			& \tilde \sigma \partial_x\left(\frac{\partial_x^2f}{\langle\partial_x f\rangle^3}\right)-\varrho_0 \partial_x f(x)=\omega(x)+(\mu_--\mu_+) K(f)(\omega)(x),\label{ome}\\
			&
			K(f)(\omega)(x)=\frac{1}{\pi} \int_{\mathbb{R}} \frac{\partial_x f(x)-\Delta_\alpha f(x)}{\left\langle\Delta_\alpha f(x)\right\rangle^2} \omega(x-\alpha) \frac{d \alpha}{\alpha}.\nonumber
		\end{align}
Here $\tilde\sigma>0$ denotes the surface tension coefficient and
\[
\varrho_0=(\rho^--\rho^+)\mathfrak g
\]
is proportional to the density contrast. The regime $\varrho_0>0$ corresponds to the Rayleigh--Taylor stable configuration, whereas $\varrho_0< 0$ corresponds to the unstable regime.\\

When the viscosities coincide, namely $\mu_-=\mu_+$, the operator $K(f)$ disappears and \eqref{ome} reduces to the equation studied in \cite{CHN1}. In contrast, when $\mu_-\neq\mu_+$, the term $K(f)(\omega)$ acts as a genuinely nonlinear perturbation of the surface-tension dynamics.\\

The leading dissipative mechanism is still provided by the surface-tension operator
\[
\partial_x\Bigl(
\frac{\partial_x^2f}{\langle\partial_xf\rangle^3}
\Bigr),
\]
which is of third order. To derive a priori estimates, one must control terms involving
\[
\int_{\mathbb R}
\frac{1+\partial_xf(x)\Delta_\alpha f(x)}
{\langle\Delta_\alpha f(x)\rangle^2}
\,K(f)(\omega)(x-\alpha)\,
\frac{d\alpha}{\alpha}.
\]
At the linearized level, this requires controlling the singular coefficient
\[
\int_{\mathbb R}
\frac{\partial_xf(x)-\Delta_\alpha f(x)}
{\langle\Delta_\alpha f(x)\rangle^2}
\,\frac{d\alpha}{\alpha}.
\]
Since
\[
\partial_xf(x)-\Delta_\alpha f(x)
=
O\!\left(
\log^{-\varkappa}\!\left(2+\frac1{|\alpha|}\right)
\right),
\]
for $f\in\dot C^{1,\log^\varkappa}$, the above integral is convergent precisely when $\varkappa>1$. This logarithmic gain is therefore sufficient to make the viscosity-jump contribution integrable, whereas the corresponding integral is only borderline convergent in the pure Lipschitz class. This provides a heuristic explanation for the appearance of the space $\dot C^{1,\log^\varkappa}$ in our analysis.\\

Several difficulties are specific to the general Muskat setting. First, the evolution equation is not local and not closed at the level of the interface: the normal velocity must be reconstructed from an elliptic transmission problem in domains depending on \(\eta\). Second, when viscosity jump and rigid boundaries are present, one loses the structural simplifications available in the equal-viscosity contour-dynamics formulation, so the interface equation involves more intricate nonlocal operators. Third, the analysis is carried out at critical regularity, where one cannot rely on classical Lipschitz estimates and must instead work in the logarithmically corrected space \(\dot C^{1,\log^\varkappa}\), with \(\varkappa>1\). Finally, it is essential to propagate the separation condition between the interface and the rigid boundaries, since otherwise the elliptic transmission problem degenerates and the contour formulation ceases to be valid.

The proof of Theorem \ref{thmgm} begins with a reformulation of the free-boundary problem as a nonlocal evolution equation for the interface \(\eta\). This equation is obtained by solving the elliptic transmission problem for the modified pressures \(q^\pm\) in the moving domains $\Omega_t^\pm$. Unlike in the equal-viscosity case, the resulting evolution is not a closed contour equation with an explicit leading operator. The normal velocity is instead determined implicitly through the bulk Darcy flow, and the corresponding interface operator depends nonlinearly both on the geometry of the interface and on the transmission structure across \(\Sigma_t\).\\

A central step of the proof is to extract the leading surface-tension contribution in a form adapted to this logarithmically critical analysis. We derive a freezing-coefficient contour formulation of the interface equation. Schematically,
\[
\partial_t\eta-\mathcal L_{\nabla\eta}\eta
=
\mathcal R[\eta]+N[\eta],
\]
where $\mathcal L_{\nabla\eta}$ is a third-order nonlocal parabolic operator with coefficients frozen at the local slope $\nabla\eta(x)$. More precisely, its principal part is given by
\begin{align}\label{evoeta0}
&\partial_t \eta (x)
-\frac{\tilde c_d\langle\nabla \eta(x)\rangle^2 }{\mu_++\mu_-}
\int _{\mathbb{R}^d}
\frac{B(\nabla\eta(x)):\delta_\alpha \nabla^2 \eta(x)}
{\langle\hat \alpha \cdot \nabla\eta(x)\rangle^{d+1}}
\frac{d\alpha}{|\alpha|^{d+1}}
\nonumber\\
&\qquad =
\frac{\tilde c_d\langle\nabla \eta(x)\rangle^2 }{\mu_++\mu_-}
\int_{\mathbb{R}^d}
\frac{\delta_\alpha B(\nabla \eta)(x): \nabla^2\eta(x-\alpha)}
{\langle\hat \alpha \cdot \nabla\eta(x)\rangle^{d+1}}
\frac{d\alpha}{|\alpha|^{d+1}}
+N[\eta](x),
\end{align}
as shown in Lemma~\ref{lemevoeta}, where $\tilde c_d>0$ is a  constant and 
\begin{equation}\label{ndefB0}
B(b)=\frac{\mathrm{Id}}{\langle b\rangle}
-\frac{b\otimes b}{\langle b\rangle^3}.
\end{equation}
In dimension $d=1$, this principal operator has the leading behavior
\[
\frac{\tilde c_d}{\mu_++\mu_-}
\frac{|D|^3 \eta}{\langle \partial_x \eta\rangle^3}.
\]

The term $N[\eta]$ in \eqref{evoeta0} contains the remaining nonlinear and transmission contributions. Thus the contour formulation separates the leading dissipative singular operator from the lower-order transmission errors.

This separation is essential at logarithmically critical regularity. One cannot close the estimates by treating the full transmission operator as an unstructured nonlocal nonlinearity. Instead, we exploit the explicit third-order parabolic principal operator and estimate $N[\eta]$ using sharp Schauder-type bounds for elliptic transmission problems in moving domains.

A second main ingredient is therefore an endpoint regularity theory for such transmission problems. After flattening the interface and freezing the geometry, we are led to divergence-form elliptic systems with coefficients depending only on the tangential variable and with jump conditions across a flat interface. More precisely, set
\[
\mathbb R^{d+1}_\tau=
\begin{cases}
\mathbb R^d\times(0,\tau), & \tau>0,\\
\mathbb R^d\times(\tau,0), & \tau<0.
\end{cases}
\]
Let
\[
A(x)\in C^{\log^\varkappa}
\bigl(\mathbb R^d;\mathbb R^{(d+1)\times(d+1)}_{\rm sym}\bigr)
\]
be uniformly elliptic, and set
\[
\mu(z)=\mu_+\mathbf 1_{z\ge 0}+\mu_-\mathbf 1_{z<0}.
\]
We consider the model transmission problem
\[
\begin{aligned}
\operatorname{div}_{x,z}\bigl(\mu(z)^{-1}A(x)\nabla_{x,z}u\bigr)
&=
\operatorname{div}_{x,z}\bigl(\mu(z)^{-1}g\bigr),
&&\text{in } \mathbb R^{d+1}_{2r}\cup \mathbb R^{d+1}_{-2r},\\
[u]_{z=0}
&=h,
&&\text{on } \{z=0\},\\
e_{d+1}\cdot
\bigl[\mu(z)^{-1}(A(x)\nabla_{x,z}u-g)\bigr]_{z=0}
&=0,
&&\text{on } \{z=0\},
\end{aligned}
\]
where
\[
[w]_{z=0}=w(0^-)-w(0^+).
\]
For this model problem, we prove a log-Dini Lipschitz estimate of the form
\[
\begin{aligned}
\|\nabla_{x,z}u\|_{L^\infty(\mathbb R^{d+1}_{r/2}\cup \mathbb R^{d+1}_{-r/2})}
\leq\;&
C\|h\|_{C^{1,\log^\varkappa}(\mathbb R^d)}
+
C\|g\|_{C^{\log^\varkappa}(\mathbb R^{d+1}_{2r}\cup \mathbb R^{d+1}_{-2r})}
\\
&+
C\Bigl(
\|\nabla_{x,z}u\|_{L^2(\mathbb R^{d+1}_{2r}\cup \mathbb R^{d+1}_{-2r})}
+
\|g\|_{L^2(\mathbb R^{d+1}_{2r}\cup \mathbb R^{d+1}_{-2r})}
\Bigr),
\end{aligned}
\]
see Lemma~\ref{lemLip} for the precise statement. This endpoint estimate gives the required control of the transmission error $N[\eta]$ in the critical topology.

With the parabolic principal structure and the endpoint transmission estimates in hand, we derive a priori bounds in
\[
\dot C^{1,\log^{\varkappa}}\cap H^1,
\]
together with time-weighted higher-order Hölder estimates. 
The stronger assumption $\eta_0\in H^1(\mathbb R^d)\cap \dot C^{1,\log^{2\varkappa-1}}(\mathbb R^d)$ in Theorem \ref{thmgm} is used only to construct a smooth reference profile $\phi$ such that
\begin{equation}
\|\eta_0-\phi\|_{\dot C^{1,\log^\varkappa}\cap H^1}
\end{equation}
is sufficiently small. Around this reference profile,
 the principal operator is frozen and the transmission terms are estimated. We then solve the approximate problems, obtain estimates uniform in the approximation parameter, and pass to the limit by compactness. Finally, uniqueness follows from stability estimates for the difference of two solutions.

\subsection{Organization of the paper}

The remainder of the paper is organized as follows.

In Section~\ref{secsch}, we establish the Schauder-type estimates for nonlocal parabolic equations. These estimates provide the main analytical framework for treating the logarithmically critical regularity.

In Section~\ref{seccon}, we derive the contour-dynamics formulation of the Muskat problem. A key point is the freezing-coefficient argument, which allows us to identify the principal third-order parabolic operator and separate it from the lower-order transmission errors.

Section~\ref{secest} is devoted to the estimates of the nonlinear terms. We control both the explicit critical nonlinear contribution and the implicit transmission term coming from the elliptic system.

Finally, in Section~\ref{secpro}, we prove the main theorem by deriving a priori estimates, constructing solutions through smooth approximation, and establishing uniqueness and stability.\vspace{0.3cm}\\
		{\bf{Data availability statement.}} This research does not have any associated data.\vspace{0.1cm}\\
		{\bf {Conflict of Interest.}} The authors declare that there are no conflicts of interest.\vspace{0.1cm}\\
		{\bf{Acknowledgments.}}  The research of Quoc-Hung Nguyen was supported by the CAS Project for Young Scientists in Basic Research, Grant No. YSBR-031; the National Key R\&D Program of China under Grant No. 2021YFA1000800; and the NSFC under Grant Nos. 1251101538 and 12595282.  Ke Chen is supported by the Research Centre for Nonlinear Analysis at The Hong Kong Polytechnic University.\vspace{0.1cm}\\
        {\bf{Note}.} {This manuscript corresponds to one part of a study initially published on arXiv (arXiv:2407.05313). The original comprehensive preprint has been divided into a series of papers, each separately addressing the well-posedness of certain free boundary problems, the general case of the Muskat problem, and problems with fixed boundaries. The present article forms Part II of this series.}

		\section{Schauder-type estimates}\label{secsch}

In this section, we develop the Schauder-type estimates needed to treat the transmission operators arising in the Muskat problem. While the argument builds on the framework introduced in ~\cite{CHN1}, the estimates are adapted here to the logarithmically corrected Hölder scale and to endpoint $L^p$ bounds, both of which play a crucial role in the critical regularity analysis below.\\

The main model is a nonlocal parabolic equation driven by a pseudo-differential operator of order $s>0$. Such estimates will later be applied to the frozen-coefficient operators arising from the surface-tension Muskat equation and to the elliptic transmission errors generated by the viscosity jump. In particular, the logarithmic estimates below are designed to close the nonlinear argument in the space
\[
\dot C^{1,\log^{\varkappa}}\cap H^1,
\qquad \varkappa>1.
\]

We first introduce the notation used throughout the paper.
		\begin{itemize}
			\item \textit{Fourier Transform.} For a function $f \in L^1(\mathbb{R}^d)$, we define its Fourier transform and inverse Fourier transform as
			\begin{equation*}
				\begin{aligned}
					&\mathcal{F}(f)(\xi)=\frac{1}{(2\pi)^{\frac{d}{2}}}\int_{\mathbb{R}^d}f(x)e^{-ix\cdot \xi}dx,\\
					&\mathcal{F}^{-1}(g)(x)=\frac{1}{(2\pi)^{\frac{d}{2}}}\int_{\mathbb{R}^d}g(\xi)e^{ix\cdot \xi}d\xi.
				\end{aligned}
			\end{equation*}
			We also write $\hat f(\xi)=\mathcal{F}(f)(\xi)$ for short. 
			It is easy to check $\mathcal{F}\mathcal{F}^{-1}(f)=\mathcal{F}^{-1}\mathcal{F}(f)=f$.
			Moreover, 
			\begin{align*}
				\mathcal{F}(f\ast g)(\xi)=(2\pi)^\frac{d}{2}\hat f(\xi)\hat g(\xi).
			\end{align*}
					
					\item \textit{Finite Difference Operators.}
					\begin{align*}
						&\delta_\alpha f(x)=f(x)-f(x-\alpha),\\&
						\Delta_\alpha f(x)=\frac{\delta_\alpha f(x)}{\alpha}~~\text{in } ~\mathbb{R},\quad\quad\quad\Delta_\alpha f(x)=\frac{\delta_\alpha f(x)}{|\alpha|}~~\text{in} ~\mathbb{R}^l, ~l\geq 2.
					\end{align*}
					The finite difference operators appearing in this paper are all in the spatial variable, \textit{ i.e.} for $h(t,x)$ defined in $[0,T]\times \mathbb{R}^d$, we denote $\delta_\alpha h(t,x)=(\delta_\alpha h(t,\cdot))(x)$.
					\item \textit{Fractional Laplacian operators.}
					\begin{equation}\label{deffracla}
						\Lambda=(-\Delta)^{\frac{1}{2}},\quad\quad\quad\Lambda^a=(-\Delta)^{\frac{a}{2}}.
					\end{equation}
					\item \textit{Gradient Operator.} For a vector-valued function $f = (f^1, \dots, f^n) \in C^1(\mathbb{R}^d; \mathbb{R}^n)$, the gradient $ \nabla f $ is the $d \times n$ matrix given by
					$$
					(\nabla f)_{ij}=\partial_{i}f^j.
					$$
					\item \textit{Multi-Index Derivatives.}
					For any function $f$, and any $\beta=(\beta_1,...\beta_d)\in\mathbb{N}^d$, we denote $\partial_x^\beta f=\partial_1^{\beta_1}...\partial_d^{\beta_d}f$. With a slight abuse of notation, for $m\in\mathbb{N}$, we denote $\nabla^mf=(\partial^\beta_x f)_{|\beta|=m}$, where $|\beta|=\sum_{i=1}^d\beta_i$.
					\item \textit{Friedrichs Mollifier.} We denote $\rho_\eps$ the standard Friedrichs mollifier with parameter $\eps>0$.
					\item \textit{Integer Part.} For $s\in\mathbb{R}^+$, denote $[s]=\max\{n:n\in\mathbb{N},n\leq s\}$ the integer part of $s$.
					\item \textit{Time-Space Norms.}  For $f:[0,T]\times \mathbb{R}^d\to \mathbb{R}^m$, we denote $\|f\|_{L^\infty(I;X)}=\sup_{t\in I}\|f(t)\|_X$, where $I\subset[0,T]$ and $X$ is a Banach space equipped with norm $\|\cdot\|_X$. Specifically, when $I=[0,T]$, we simply write $\|\cdot\|_{L^\infty_TX}$.
					\item \textit{Indicator Function.} For a measurable set $A \subset \mathbb{R}^d$, define  
					\begin{equation*}\label{indicaf}
						\begin{aligned}
							\mathbf{1}_A(x):=\begin{cases}
								1,\quad\quad x\in A,\\
								0,\quad\quad x\notin A.
							\end{cases}
						\end{aligned}
					\end{equation*}
                    \item \textit{Besov Spaces.} For $f\in \mathcal{S}'(\mathbb{R}^d)$, $p\in[1,\infty]$ and $\kappa\in(0,1)$, we denote the $\dot B_{p,\infty}$ norm as
                    \begin{equation*}
                        \|f\|_{\dot B_{p,\infty}^\kappa}=\sup_{\alpha}\frac{\|\delta_\alpha f\|_{L^p}}{|\alpha|^\kappa},
                    \end{equation*}
                    and the Banach space $\dot B_{p,\infty}^\kappa$ equipped with the corresponding norm. Furthermore, for any $m\in\mathbb{N}$, we denote $\dot B_{p,\infty}^{m+\kappa}=\{f|\nabla^mf\in\dot B_{p,\infty}^\kappa\}$.
				\end{itemize}

				Consider the Cauchy problem
				\begin{equation}\label{eqpara}
					\begin{aligned}
						&	\partial_{t} u(t, x)+\mathcal{L}_{s} u(t, x)=\mathcal{P}_{\gamma} f(t, x)+g(t,x) \quad \text { in } (0, \infty)\times\mathbb{R}^{d},\\
						&	u|_{t=0}=u_0,
					\end{aligned} 
				\end{equation}
				where $0<\gamma<s$ are fixed parameters.\\
                
                The linearized equation associated with \eqref{evoeta0} fits into the form
\eqref{eqpara} with \(s=3\), reflecting the third-order parabolic structure
generated by surface tension. We therefore formulate the Schauder estimates in
the following general pseudo-differential setting, which will later be applied
to the frozen-coefficient operators arising in the Muskat equation.\\

To make the assumptions precise, we impose the following conditions on the parameters, forcing terms, and principal operator:\vspace{0.2cm}\\
				1. {\bf Regularity Parameters:} 
				We  fix $\varkappa>1$, $m\in\mathbb{N}_+$, and  select parameter $\kappa $ such that   \begin{align}
					\label{defpara}
					\kappa,\kappa+\gamma,\kappa+\gamma-s\notin\mathbb{N}, \ \ \ \text{and}  \ \ s-\gamma<\kappa<s.
				\end{align}
				For simplicity, denote $\kappa_0=\gamma-s+\kappa$. \vspace{0.2cm}\\
				2. {\bf Forcing Terms:} The external input $f:(0,T)\times \mathbb{R}^d\to \mathbb{R}^N$ is a given forcing term satisfying $\da_T(f,g)<\infty$, where
				\begin{align}
					\da_T(f,g)= \sup_{t\in[0,T]}&|\log t|^\varkappa (t^{\frac{\kappa}{s}}\|f(t)\|_{\dot C^{\kappa_0}}+t^{\frac{m+\kappa}{s}}\|f(t)\|_{\dot C^{m+\kappa _0}})\nonumber\\
					&+ \int_0^T\|g(\tau)\|_{\dot C^{\log^{\varkappa}}}d\tau+\sup_{t\in[0,T]}|\log t|^\varkappa t^{\frac{m}{s}+1}\|g(t)\|_{\dot C^{m}}.\label{defda1}
				\end{align}
                For the result in $L^p$ setting ($p\in[1,\infty]$), we suppose $\da_T^p(f,g)<\infty$ with
                \begin{equation}\label{defda2}
					\begin{aligned}
						\da_T^p(f,g)=\sup_{t\in[0,T]}(t^{\frac{\kappa}{s}}\|f(t)\|_{\dot B_{p,\infty}^{\kappa_0}}+t^{\frac{m+\kappa}{s}}\|f(t)\|_{\dot B_{p,\infty}^{m+\kappa_0}})+\|g\|_{L^1_TL^p}+\sup_{t\in[0,T]}t^{\frac{m}{s}+1}\|\nabla^mg(t)\|_{L^p}.
					\end{aligned}
				\end{equation}
				3. {\bf Principal Operator:} The Pseudo-differential operators $\mathcal{L}_{s}$  of order $s>0$ is defined by
				\begin{align}\label{defop}
					&\mathcal{L}_su(t,x)=(2\pi)^{-\frac{d}{2}}\int_{\mathbb{R}^d}\A(t,x,\xi)\hat u(t,\xi)e^{ix\cdot\xi}d\xi,\quad\quad\quad
				\end{align}
				where $\hat u(t,\xi)$  denotes the Fourier transform of $u(t,x)$ with respect to the spatial variable $x$, $\A(t,x,\xi)\in \mathbb{R}^{N\times N}$ is a  matrix satisfying 
				\begin{equation}\label{condop}
					\begin{aligned}
						&\quad\A(t,x,\xi)\geq c_0|\xi|^s\mathrm{Id},\\
						&\sum_ {j,l\leq d+ m+s+2 } {|\xi|^{l-s}}{\left|\nabla _x^j\nabla^{l}_\xi \A(t,x,\xi)\right| }\leq c_1,
						\ \ \forall  \xi \neq 0, (t,x)\in(0,T)\times\mathbb{R}^d,
					\end{aligned} 
				\end{equation}
				for some constants $0<c_0<c_1$. Here $\mathrm{Id}$ is the $N\times N$ identity matrix. 
				And $\mathcal{P}_{\gamma}$ is a differential operator defined by 
				\begin{align}\label{defPg}
					\mathcal{P}_{\gamma} f(t,x)=\int_{\mathbb{R}^d} \B(\xi)\hat f(t,\xi)e^{ix\cdot\xi}d\xi,\ \ \ \ 0<\gamma\leq s,
				\end{align}
				where the symbol $\B(\xi)\in \mathbb{R}^{N\times N}$ satisfies the growth condition
				\begin{align*}
					\left|\nabla^{j}_\xi \B(\xi)\right| \lesssim |\xi|^{\gamma-j},\ \ \ \forall 0\leq j\leq d+m+\kappa+2.
				\end{align*}
				~~
				We now apply the freezing-coefficient method to \eqref{eqpara}. Fix
$x_0\in\mathbb R^d$. Freezing the symbol at $x_0$, we rewrite \eqref{eqpara} as
				\begin{equation}\label{eqpafi}
					\begin{aligned}
						&\partial_tu(t,x)+(2\pi)^{-\frac{d}{2}}\int _{\mathbb{R}^d}\A(t,x_0,\xi)\hat{u}(t,\xi)e^{ix\cdot\xi}d\xi=\mathcal{P}_{\gamma} f(t, x)+g(t,x)+R_{x_0}[u](t,x),
					\end{aligned}
				\end{equation}
				with 
				\begin{align}\label{defRx0}
					R_{x_0}[u](t,x)=(2\pi)^{-\frac{d}{2}}\int _{\mathbb{R}^d}(\A(t,x_0,\xi)-\A(t,x,\xi))\hat{u}(t,\xi)e^{ix\cdot\xi} d\xi.
				\end{align}
                Thus the variable-coefficient equation is viewed as a constant-coefficient
equation, with the variation of the coefficients treated as a perturbative
remainder.\\
Let $\K_{x_0}(t,\tau,x)$ denote the fundamental solution of the adjoint frozen
system
				\begin{equation}\label{defk0}
					\begin{aligned}
						&-\partial_\tau \K_{x_0}(t,\tau,x)+(2\pi)^{-\frac{d}{2}}\int _{\mathbb{R}^d}\hat \K_{x_0}(t,\tau,\xi)\A(\tau,x_0,\xi)e^{ix\cdot \xi}d\xi=0, \ \ \ \ \ (\tau,x) \in (0,t)\times \mathbb{R}^d,\\
						&\lim_{\tau\to t^-}\K_{x_0}(t,\tau,x)=\delta(x)\mathrm{Id},\ \ \ \ \ x\in\mathbb{R}^d,
					\end{aligned}
				\end{equation}
				here $\mathrm{Id}$ is the $N\times N$ identity matrix,  $\hat \K_{x_0}(t,\tau,\xi)$ is the Fourier transform of  $\K_{
					x_0}(t,\tau,x)$ in $x$-variable, and $\delta(x)$ is the Dirac delta function. 
                By Duhamel's formula, for every $x_0\in\mathbb R^d$, the solution can be written as
				\begin{align}
					u(t,x)=&\int_{\mathbb{R}^d} \K_{x_0}(t,0,x-z)u_0(z)dz+\int_0^t \int _{\mathbb{R}^d}\K_{x_0}(t,\tau,x-z)(\mathcal{P}_\gamma f+g)(\tau,z)dzd\tau\nonumber \\
					&+\int_0^t\int_{\mathbb{R}^d} \K_{x_0}(t,\tau,x-z)  R_{x_0}[u](\tau,z)dzd\tau\nonumber\\
					:=	&u_{L,x_0}(t,x)+u_{N,x_0}(t,x)+u_{R,{x_0}}(t,x),\label{uform}
				\end{align}
				which holds for any $x_0\in\mathbb{R}^d$.\\
                
                We shall use the following kernel estimates for the frozen fundamental solution. 
They are consequences of the ellipticity and symbol bounds in \eqref{condop}; 
see \cite{CHN1} for the proof.
				\begin{lemma}\label{lemfourierK}The following pointwise estimates hold for $\K_{x_0}(t_1,t_2,x)$:
					\begin{align}
						&\sup_{x_0\in\mathbb{R}^d}|\K_{x_0}(t_1,t_2,x)|\lesssim \frac{t_1-t_2}{((t_1-t_2)^\frac{1}{s}+|x|)^{d+s}},\ \ \ \forall \ 0<t_2<t_1,\ x\in\mathbb{R}^d,\label{ptKd}\\
						&\sup_{x_0\in\mathbb{R}^d}|\nabla_x^l\K_{x_0}(t_1,t_2,x)|\lesssim \frac{1}{((t_1-t_2)^\frac{1}{s}+|x|)^{d+l}},\ \ \forall l\in\mathbb{N}.\label{ptKd1}
					\end{align}
					And 	for any $\alpha\neq 0$, $0\leq \sigma<1$,
					\begin{align}\label{delKL1}
						\left\|\sup_{x_0\in\mathbb{R}^d}	|\nabla _x^l\delta_{\alpha}\K_{x_0}|(t_1,t_2)|\cdot|^{\sigma}\right\|_{L^1_x}\lesssim |\alpha|^{\sigma}\min\left\{1,\frac{|\alpha|}{(t_1-t_2)^{\frac{1}{s}}}\right\}^{1-\sigma} (t_1-t_2)^{-\frac{l}{s}},\ \ \ \ \forall l\in\mathbb{N}.
					\end{align}  
					As a direct result of \eqref{ptKd} and \eqref{ptKd1}, we obtain 
					\begin{align}\label{kerL1}
						\left\|\sup_{x_0\in\mathbb{R}^d}|\nabla _x^l\K_{x_0}|(t_1,t_2)|\cdot|^{\sigma}\right\|_{L^1_x}\lesssim (t_1-t_2)^{-\frac{l-\sigma}{s}},\ \ \forall l\in\mathbb{N},\ 0\leq \sigma<l.
					\end{align} 
					The implicit constants in \eqref{ptKd}- \eqref{kerL1} depend on $c_0$ and $c_1$. Specifically, we have 
					\begin{align}\label{KL1}
						\left\|\sup_{x_0\in\mathbb{R}^d}|\K_{x_0}|(t_1,t_2)\right\|_{L^1}\leq C,
					\end{align}
					with $C>0$ that is independent of $c_0,c_1$.
				\end{lemma}
				
				We prove the following Schauder-type estimate, which is an analogue to \cite[Proposition 2.9]{CHN1} in  H\"{o}lder space with logarithmic derivatives and $L^p$ spaces. We shall need both Hölder-type and \(L^p\)-based versions of the Schauder estimate. The former is used to control the principal nonlocal evolution in the logarithmically critical topology. The latter is needed for the lower-order quantities entering through the elliptic transmission problem. Indeed, the remainder terms in the contour formulation depend implicitly on the bulk variables, and their estimates involve not only pointwise moduli of continuity but also lower-order energy norms, such as the \(L^2\) norm of the elliptic solution and of the forcing terms. For this reason we record a parallel \(L^p\) estimate, which will be used to propagate these lower-order bounds and to control the implicit forcing terms generated by the transmission system.

				\begin{lemma}
					\label{lemlog}
					i) Consider the Cauchy problem \eqref{eqpara} with $\A(t,x,\xi)$ satisfying \eqref{condop}, if $\|u_0\|_{\dot C^{\log^\varkappa}}+\da_T(f,g)<\infty$, then there exists $T>0$ and a unique solution $$u\in C((0,T],\dot C^{\log^\varkappa}_x(\mathbb{R}^d))\cap L^\infty_{{loc}}((0,T],\dot C^{m+\kappa}_x(\mathbb{R}^d)),$$ such that the following estimate holds:\\
					\begin{equation}\label{relog}
						\begin{aligned}
							&\sup_{t\in[0,T]}(\|u(t)\|_{\dot C^{\log^\varkappa}}+|\log t|^{\varkappa}t^{\frac{m+\kappa}{s}}\|u(t)\|_{\dot C^{m+\kappa}})	\lesssim \|u_0\|_{\dot C^{\log^\varkappa}}+\da_T(f,g).
						\end{aligned}
					\end{equation}
                    ii) Consider the Cauchy problem \eqref{eqpara} with $\A(t,x,\xi)$ satisfying \eqref{condop}, if $\|u_0\|_{L^p}+\da_T^p(f,g)<\infty$, then there exists $T>0$ and a unique solution $u\in C((0,T],L^p_x)\cap L^\infty_{{loc}}((0,T],\dot B^{m+\kappa}_{p,\infty})$ such that the following estimate holds:\\
                    \begin{equation}\label{reLp}
						\begin{aligned}
							&\sup_{t\in[0,T]}(\|u(t)\|_{L^p}+t^{\frac{m+\kappa}{s}}\|u(t)\|_{\dot B^{m+\kappa}_{p,\infty}})	\lesssim \|u_0\|_{L^p}+\da_T^p(f,g).
						\end{aligned}
					\end{equation}
				\end{lemma}
				\begin{proof} We only establish a priori estimates. The existence and uniqueness of solution can be obtained following \cite{CHN1}. Let $T\in(0,1)$. Denote
					\begin{equation}\label{defxt}
						\|u\|_{X_T^\varkappa }=\sup_{t\in[0,T]}(\|u(t)\|_{\dot C^{\log^\varkappa}}+|\log t|^\varkappa  t^{\frac{m+\kappa}{s}}\|u(t)\|_{\dot C^{m+\kappa}}),
					\end{equation}
				\begin{equation*}
					\|u\|_{Y_T^p}=\sup_{t\in[0,T]}\left(\|u(t)\|_{L^p}+t^{\frac{m+\kappa}{s}}\|u(t)\|_{\dot B_{p,\infty}^{m+\kappa}}\right).
				\end{equation*}
                By interpolation, for every $t\in(0,T)$, we obtain that for any $n\in\mathbb{N}$, $\sigma\in(0,1)$ such that $n+\sigma\leq m+\kappa$,
                \begin{equation}\label{intp11}
                \begin{aligned}
                &\|u(t)\|_{\dot C^{n+\sigma}}\lesssim |\log t|^{-\varkappa}t^{-\frac{n+\sigma}{s}}\|u\|_{X_T^\varkappa },
                \\
                &\|u(t)\|_{\dot B^{n+\sigma}_{p,\infty}}\lesssim  |\log t|^{-\varkappa}t^{-\frac{n+\sigma}{s}}\|u\|_{Y_T^p}.
                \end{aligned}
                \end{equation}
					By the formula \eqref{uform}, we obtain 
					\begin{equation}\label{deu}
						\begin{aligned}
                         &\| \nabla^n u\|_{L^p}\leq \|(\nabla^nu_{L,x_0})|_{x_0=x}\|_{L^p}+\|(\nabla^nu_{N,x_0})|_{x_0=x}\|_{L^p}+\|(\nabla^nu_{R,x_0})|_{x_0=x}\|_{L^p},\\
                            &\|\delta_\alpha \nabla^n u\|_{L^p}\leq \|(\delta_\alpha \nabla^nu_{L,x_0})|_{x_0=x}\|_{L^p}+\|(\delta_\alpha \nabla^nu_{N,x_0})|_{x_0=x}\|_{L^p}+\|(\delta_\alpha \nabla^nu_{R,x_0})|_{x_0=x}\|_{L^p},
						\end{aligned} 
					\end{equation}
					for any $\alpha\in\mathbb{R}^d$ and any $n\in\mathbb{N}$.\vspace{0.1cm}\\
                    {\bf Step 1. Estimate of $u_{L,x_0}$}\\
					We first consider $$u_{L,x_0}=\int_{\mathbb{R}^d}\K_{x_0}(t,0,z)u_0(x-z)dz.$$
					By \eqref{kerL1}, it holds
                          \begin{equation}\label{uLLp}
                        \|u_{L,x_0}(t,x)|_{x_0=x}\|_{L_x^p}\lesssim \big\|\sup_{z}|\K_{z}(t,0,\cdot)|\big\|_{L^1}\|u_0\|_{L^p}\lesssim\|u_0\|_{L^p},
                    \end{equation}
                    and 
					\begin{align}\label{logL}
						\|(\delta_\alpha u_{L,x_0})(t,x)|_{x_0=x}\|_{L^\infty}
						&\lesssim \sup_{z\in\mathbb{R}^d}\|\K_{z}(t,0,\cdot)\|_{L^1}\|\delta_\alpha u_0\|_{L^\infty_x}
						\lesssim  \log^{-\varkappa}\left(2+|\alpha|^{-1}\right)\|u_0\|_{\dot C^{\log^\varkappa}}. 
					\end{align}
					Moreover, 
					we have 
					\begin{align*}
						|\delta_\alpha \nabla_{x}^{m+[\kappa]}u_{L,x_0}(t,x)|=&\left|\int_{\mathbb{R}^d}\delta_\alpha\nabla_{x}^{m+[\kappa]}\K_{x_0}(t,0,z)u_0(x-z)dz\right|.
					\end{align*}
                  Fix $x_0=x$, and take $L^p$ in $x$ on both sides, from \eqref{delKL1}, we obtain 
                      \begin{equation}\label{uLBe}
                        \begin{aligned}
                            \|(\delta_\alpha \nabla^{m+[\kappa]}u_{L,x_0})(t,x)|_{x_0=x}\|_{L_x^p}&\lesssim\left\|\sup_{z}|\delta_\alpha \nabla^{m+[\kappa]}\K_{z}|(t,0,\cdot)\right\|_{L^1}\|u_0\|_{L^p}\\&\lesssim |\alpha|^{\kappa-[\kappa]} t^{-\frac{m+\kappa}{s}}\|u_0\|_{L^p}.
                        \end{aligned}
                    \end{equation}
                    On the other hand, observe that $\int \delta_\alpha \nabla_{x}^{m+[\kappa]}K_{x_0}(t,0,z)dz=0$, we obtain 
                    	\begin{align*}
						|\delta_\alpha \nabla_{x}^{m+[\kappa]}u_{L,x_0}(t,x)|=&\int_{\mathbb{R}^d}|\delta_\alpha \nabla_{x}^{m+[\kappa]}\K_{x_0}(t,0,z)||u_0(x-z)-u_0(x)|dz\\
						\lesssim& \int_{\mathbb{R}^d}|\delta_\alpha\nabla_{z}^{m+[\kappa]}\K_{x_0}(t,0,z)|\log^{-\varkappa}\left(2+|z|^{-1}\right)dz\|u_0\|_{\dot C^{\log^\varkappa }}.
					\end{align*}
					Using the elementary inequality 
					\begin{equation}\label{ine2}
						|\log(b)|\lesssim\log (2+a^{-1}) \log(2+ab),\ \ \ \forall a>0,b>2,
					\end{equation}
					we obtain 
					$$
					\log^{-\varkappa}\left(2+|z|^{-1}\right)\lesssim \log^{\varkappa}\left(2+(t^{-\frac{1}{s}}|z|)\right) |\log t|^{-\varkappa}.
					$$
					Observe that  $\log^\varkappa(2+(t^{-\frac{1}{s}}|x|))\lesssim 1+(t^{-\frac{1}{s}}|x|)^\eta$ for $0<\eta\ll 1$. Combining this with \eqref{delKL1} yields
					\begin{align*}
						&\int_{\mathbb{R}^d}|\delta_\alpha\nabla_{z}^{m+[\kappa]}\K_{x_0}(t,0,z)|\log^{-\varkappa}\left(2+|z|^{-1}\right)dz\\
						&\quad\quad\quad\lesssim |\log t|^{-\varkappa}\int_{\mathbb{R}^d}|\delta_\alpha\nabla_{z}^{m+[\kappa]}\K_{x_0}(t,0, z)|(1+(t^{-\frac{1}{s}}| z|)^\eta)dz\\
						&\quad\quad\quad\lesssim |\log t|^{-\varkappa}t^{-\frac{m+[\kappa]+\eta}{s}}|\alpha|^{\eta}\min\{1,t^{-\frac{1}{s}}|\alpha| \}^{1-\eta}
						\\&\quad\quad\quad\lesssim |\alpha|^{\kappa-[\kappa]} |\log t|^{-\varkappa} t^{-\frac{m+\kappa}{s}}.
					\end{align*}
					So we obtain
					\begin{align}\label{loghL}
						\sup_\alpha \frac{	\|(\delta_\alpha \nabla_{x}^{m+[\kappa]}u_{L,x_0})(t,x)|_{x_0=x}\|_{L^\infty_x}}{|\alpha|^{\kappa-[\kappa]}}\lesssim |\log t|^{-\varkappa}t^{-\frac{m+\kappa}{s}}\|u_0\|_{\dot C^{\log^\varkappa}}.
					\end{align}
                    We conclude from \eqref{uLLp} and \eqref{uLBe} that 
                    \begin{align}\label{uLLp1}
                      \sup_{t\in[0,T]}\left(    \|u_{L,x_0}(t,x)|_{x_0=x}\|_{L_x^p}+t^\frac{m+\kappa}{s}\sup_\alpha \frac{    \|(\delta_\alpha \nabla^{m+[\kappa]}u_{L,x_0})(t,x)|_{x_0=x}\|_{L_x^p}}{|\alpha|^{\kappa-[\kappa]}}\right)\lesssim \|u_0\|_{L^p},
                    \end{align}
                    and from \eqref{logL}, \eqref{loghL} that 
                    	\begin{equation}\label{ulhol}
						\begin{aligned}
							\sup_{t\in[0,T]}\sup_\alpha&\left(\frac{\|(\delta_\alpha u_{L,x_0})(t,x)|_{x_0=x}\|_{L^\infty_x}}{\log^{-\varkappa}(2+|\alpha|^{-1})}+|\log t|^\varkappa   t^\frac{m+\kappa}{s}\frac{\|(\delta_\alpha \nabla_x ^{m+[\kappa]}u_{L,x_0})(t,x)|_{x_0=x}\|_{L^\infty_x}}{|\alpha|^{\kappa-[\kappa]}}\right)\\
							&\lesssim \|u_0\|_{\dot C^{\log^\varkappa}}.
						\end{aligned}
					\end{equation}
                    \vspace{0.1cm}
                    {\bf Step 2. Estimate of $u_{N,x_0}$}\\
					Then we consider $u_{N,x_0}=u_{N,1,x_0}+u_{N,2,x_0}$ with 
					\begin{align}
						& u_{N,1,x_0}(t,x)= \int_0^t \int _{\mathbb{R}^d}\K_{x_0}(t,\tau,x-z)\mathcal{P}_\gamma f(\tau,z)dzd\tau,\label{esdelL}\\
						&u_{N,2,x_0}(t,x)= \int_0^t \int _{\mathbb{R}^d}\K_{x_0}(t,\tau,x-z)g(\tau,z)dzd\tau\label{esdelg}.
					\end{align}
					First, for $u_{N,1,x_0}$, 
					recalling the definition of $\mathcal{P}_\gamma$ in \eqref{defPg}, we have 
					\begin{align*}
						u_{N,1,x_0}(t,x)&=	\int_0^t \int_{\mathbb{R}^d}  L_{x_0}(t,\tau,x-z) \Lambda^{\vartheta}f(\tau,z)dz d\tau.
					\end{align*}
					Here $L_{x_0}(t,x)$ is defined by  
					$$
					L_{x_0}(t,\tau,x) =\frac{1}{(2\pi)^\frac{d}{2}}\int _{\mathbb{R}^d}\frac{\B(\xi)}{|\xi|^{\vartheta}} \mathcal{F}(\K_{x_0}(t,\tau))(\xi)e^{ix\cdot \xi}d\xi,
					$$
					for some $\vartheta\in((\kappa_0-1)_+,\kappa_0)$.\\
                    Since the Fourier multiplier satisfies
\[
\left.
\frac{\B(\xi)}{|\xi|^\vartheta}
\widehat{\K}_{x_0}(t,\tau,\xi)
\right|_{\xi=0}
=0,
\]
we have
\[
\widehat L_{x_0}(t,\tau,0)=0.
\]
Therefore,
\[
\int_{\mathbb R^d}L_{x_0}(t,\tau,x)\,dx=0.
\]
By the pointwise estimate Lemma 2.1 in \cite{CHN1}, we have  
      \begin{align}\label{S2eq32}
        &\sup_{z}|\nabla^kL_{z}(t,\tau,y)| \lesssim (t-\tau)^{-\frac{d+k+\gamma-\vartheta}{s}}\left(1+\frac{|y|}{(t-\tau)^\frac{1}{s}}\right)^{-d-\gamma+\vartheta},\\
        &\int_{\mathbb{R}^d} \sup_{z}|\delta_\alpha \nabla^kL_{z}(t,\tau,y)| |y|^{\sigma} dy \lesssim |\alpha|^{\sigma}(t-\tau)^{-\frac{k+\sigma+\gamma-\vartheta}{s}}\left\{1,\frac{|\alpha|}{(t-\tau)^{\frac{1}{s}}}\right\}^{1-\sigma},\nonumber
      \end{align}              
for any $\sigma\in(0,\gamma-\vartheta)$. Then, we obtain 
                    \begin{equation}\label{uN1Lp}
                    \begin{aligned}
                    \|u_{N,1,x_0}(t,x)|_{x_0=x}\|_{L^p_x}&\lesssim \int_0^t \int_{\mathbb{R}^d}  \sup_{z}|L_{z}(t,\tau,y)|\|\Lambda^{\vartheta}f(\tau,x-y)-\Lambda^{\vartheta}f(\tau,x)\|_{L^p_x}dy d\tau \\
                    &\lesssim \int_0^t\int_{\mathbb{R}^d}\sup_{z}|L_{z}(t,\tau,y)||y|^{\kappa_0-\vartheta} dy \|\Lambda^\vartheta f(\tau)\|_{\dot B^{\kappa_0-\vartheta}_{p,\infty}}d\tau\\
                    &\overset{\eqref{S2eq32}}{\lesssim}\int_0^t \frac{1}{(t-\tau)^\frac{\gamma-\kappa_0}{s}\tau^{\frac{\kappa}{s}}}d\tau\sup _{t \in[0, T]}  t^{\frac{\kappa}{s}}\|f(t)\|_{\dot{B}^{\kappa_0}_{p,\infty}}\\
                    &\lesssim \da_T^p(f,0),
                    \end{aligned}
                    \end{equation}
                    where in the third inequality we used the standard equivalence
\[
\|\Lambda^\vartheta f(\tau)\|_{\dot B^{\kappa_0-\vartheta}_{p,\infty}}
\sim
\|f(\tau)\|_{\dot B^{\kappa_0}_{p,\infty}}.
\]
					Moreover, for any $\alpha\in\mathbb{R}^d$,
					\begin{align*}
						|\delta_\alpha u_{N,1,x_0}(t,x)|
						\lesssim &\left|\int_0^t \int_{\mathbb{R}^d} \delta_\alpha L_{x_0}(t,\tau,x-z)(\Lambda^{\vartheta}f(\tau,z)-\Lambda^{\vartheta}f(\tau,x))d\tau dz\right|\\
						\lesssim &	\int_0^t\int_{\mathbb{R}^d}| \delta_\alpha L_{x_0}(t,\tau,z)||z|^{\kappa_0-\vartheta}dz\|f(\tau)\|_{\dot{C}^{\kappa_0}}d\tau\\
						\overset{\eqref{S2eq32}}{\lesssim}&\int_0^t \frac{|\alpha|^{\kappa_0-\vartheta}}{(t-\tau)^\frac{\gamma-\vartheta}{s}|\log \tau|^\varkappa\tau^{\frac{\kappa}{s}}}\min\left\{1,\frac{|\alpha|}{(t-\tau)^\frac{1}{s}}\right\}^{1-(\kappa_0-\vartheta)}d\tau\sup _{t \in[0, T]} |\log t|^\varkappa t^{\frac{\kappa}{s}}\|f(t)\|_{\dot{C}^{\kappa_0}}.
					\end{align*}
					Note that $|\log t|\lesssim |\log \tau|$ for any $0<\tau<t<\frac{1}{2}$, then it follows that
					\begin{align*}
						&\int_0^t \frac{|\alpha|^{\kappa_0-\vartheta}}{(t-\tau)^\frac{\gamma-\vartheta}{s}|\log \tau|^\varkappa\tau^{\frac{\kappa}{s}}}\min\left\{1,\frac{|\alpha|}{(t-\tau)^\frac{1}{s}}\right\}^{1-(\kappa_0-\vartheta)}d\tau\\
						&\quad\quad\quad\lesssim |\log t|^{-\varkappa}\min\left\{1,|\alpha| t^{-1/s}\right\}^\varepsilon\lesssim \log^{-\varkappa} (2+|\alpha|^{-1}),
					\end{align*}
					where $0<\varepsilon<\min\{1,\frac{\kappa}{10}\}$. The last inequality can be proved by considering the case $|\alpha|<t^{\frac{1}{s}}$ and $|\alpha|>t^{\frac{1}{s}}$ separately. If $|\alpha|>t^{\frac{1}{s}}$, it holds automatically. Otherwise denote $b=|\alpha|t^{-\frac{1}{s}}<1$. If $b<t^{\frac{1}{2s}}$, then $b^\frac{\varepsilon}{\kappa}\log(2+b^{-1}t^{-\frac{1}{s}})\lesssim b^{\frac{\varepsilon}{\kappa}}\log(2+b^{-\frac{3}{2}})\lesssim 1$. If $b>t^{\frac{1}{2s}}$, then $b^\frac{\varepsilon}{\kappa}\log(2+b^{-1}t^{-\frac{1}{s}})\lesssim\log(2+t^{-\frac{1}{2s}})\lesssim |\log t|$.\\
                    Hence, we deduce that 
					\begin{align}\label{clog1}
						\log^{\varkappa} (2+|\alpha|^{-1}) \|(\delta_\alpha u_{N,1,x_0})(t,x)|_{x_0=x}\|_{L^\infty_x}
						\lesssim & \da_T(f,0).
					\end{align}
            For higher order derivatives, using integration by parts, we obtain 
                    \begin{equation}\label{uN1Be}
                    \begin{aligned}
                    &\|(\delta_\alpha\nabla^{m+[\kappa]}u_{N,1,x_0})(t,x)|_{x_0=x}\|_{L^p_x}\\
                    &\quad\quad\lesssim \int_0^\frac{t}{2} \int_{\mathbb{R}^d}  \sup_{z}|\delta_\alpha \nabla^{m+[\kappa]}L_{z}(t,\tau,y)|\|\Lambda^{\vartheta}f(\tau,x-y)-\Lambda^{\vartheta}f(\tau,x)\|_{L^p_x} dyd\tau\\
                    &\quad\quad\quad\quad+\int_\frac{t}{2}^t \int_{\mathbb{R}^d}  \sup_{z}|\delta_\alpha\nabla^{[\kappa]} L_{z}(t,\tau,y)|\|\Lambda^{\vartheta}\nabla^{m}f(\tau,x-y)-\Lambda^{\vartheta}\nabla^{m}f(\tau,x)\|_{L^p_x} dy d\tau \\
                    &\quad\quad\overset{\eqref{S2eq32}}{\lesssim}|\alpha|^{-\vartheta+\kappa_0}\int_0^\frac{t}{2} \frac{\min\{1,\frac{|\alpha|}{(t-\tau)^{\frac{1}{s}}}\}^{1-\kappa_0+\vartheta}}{(t-\tau)^\frac{m+\gamma+\vartheta-\kappa_0}{s}\tau^{\frac{\kappa}{s}}}d\tau\sup _{t \in[0, T]}  t^{\frac{\kappa}{s}}\|f(t)\|_{\dot{B}^{\kappa_0}_{p,\infty}}\\
                    &\quad\quad\quad\quad+|\alpha|^{-\vartheta+\kappa_0}\int_\frac{t}{2}^t \frac{\min\{1,\frac{|\alpha|}{(t-\tau)^{\frac{1}{s}}}\}^{1-\kappa_0+\vartheta}}{(t-\tau)^\frac{\gamma+\vartheta-\kappa_0}{s}\tau^{\frac{m+\kappa}{s}}}d\tau\sup _{t \in[0, T]}  t^{\frac{m+\kappa}{s}}\|f(t)\|_{\dot{B}^{m+\kappa_0}_{p,\infty}}\\
                    &\quad\quad\lesssim |\alpha|^{\kappa-[\kappa]}t^{-\frac{m+\kappa}{s}}\da_T^p(f,0).
                    \end{aligned}
                    \end{equation}
                    Specifically, in the case $p=\infty$, we also have  
                        \begin{equation}\label{uN1Hol}
                    \begin{aligned}
                    &\|(\delta_\alpha\nabla^{m+[\kappa]}u_{N,1,x_0})(t,x)|_{x_0=x}\|_{L^\infty_x}\\
                     &\quad\quad\overset{\eqref{S2eq32}}{\lesssim} |\alpha|^{-\vartheta+\kappa_0}\int_0^\frac{t}{2} \frac{\min\{1,\frac{|\alpha|}{(t-\tau)^{\frac{1}{s}}}\}^{1-\kappa_0+\vartheta}}{(t-\tau)^\frac{m+\gamma+\vartheta-\kappa_0}{s}\tau^{\frac{\kappa}{s}}|\log \tau|^\varkappa}d\tau\sup _{t \in[0, T]} |\log t|^\varkappa t^{\frac{\kappa}{s}}\|f(t)\|_{\dot{C}^{\kappa_0}}\\
                    &\quad\quad\quad\quad+|\alpha|^{-\vartheta+\kappa_0}\int_\frac{t}{2}^t \frac{\min\{1,\frac{|\alpha|}{(t-\tau)^{\frac{1}{s}}}\}^{1-\kappa_0+\vartheta}}{(t-\tau)^\frac{\gamma+\vartheta-\kappa_0}{s}\tau^{\frac{m+\kappa}{s}}|\log\tau|^\varkappa}d\tau\sup _{t \in[0, T]} |\log t|^\varkappa t^{\frac{m+\kappa}{s}}\|f(t)\|_{\dot{C}^{m+\kappa_0}}\\
                    &\quad\quad\lesssim |\alpha|^{\kappa-[\kappa]}|\log t|^{-\varkappa}t^{-\frac{m+\kappa}{s}}\da_T(f,0).
                    \end{aligned}
                    \end{equation}
                    We conclude from \eqref{uN1Lp}, \eqref{clog1}, \eqref{uN1Be} and \eqref{uN1Hol} that 
                    \begin{align}\label{uN1lp}
                     \sup_{t\in[0,T]}\left(    \|u_{N,1,x_0}(t,x)|_{x_0=x}\|_{L_x^p}+t^\frac{m+\kappa}{s}\sup_\alpha \frac{    \|(\delta_\alpha \nabla^{m+[\kappa]}u_{N,1,x_0})(t,x)|_{x_0=x}\|_{L_x^p}}{|\alpha|^{\kappa-[\kappa]}}\right)\lesssim \da_T^p(f,0),
                    \end{align}
                    	\begin{equation}\label{un1hol}
						\begin{aligned}
							\sup_{t\in[0,T]}\sup_\alpha&\left(\frac{\|(\delta_\alpha u_{N,1,x_0})(t,x)|_{x_0=x}\|_{L^\infty_x}}{\log^{-\varkappa}(2+|\alpha|^{-1})}+|\log t|^\varkappa   t^\frac{m+\kappa}{s}\frac{\|(\delta_\alpha \nabla_x ^{m+[\kappa]}u_{N,1,x_0})(t,x)|_{x_0=x}\|_{L^\infty_x}}{|\alpha|^{\kappa-[\kappa]}}\right)\\
							&\lesssim \da_T(f,0).
						\end{aligned}
					\end{equation}
                    For $u_{N,2,x_0}$, by \eqref{esdelg} and Lemma \ref{lemfourierK} we have 
                    \begin{equation}\label{un2}
                        					\begin{aligned}
                    \| u_{N,2,x_0}(t,x)|_{x_0=x}\|_{L^p_x}&\lesssim \int_0^t \big\|\sup_{z}|\K_{z}(t,\tau)|\big\|_{L^1} \|g(\tau)\|_{L^p}d\tau\lesssim \da_T^p(0,g),\\
						\|(\delta_\alpha u_{N,2,x_0})(t,x)|_{x_0=x}\|_{L^\infty_x}&\lesssim \log ^{-\varkappa}(2+|\alpha|^{-1}) \int_0^t\big\|\sup_{z} |\K_{z}(t,\tau)|\big\|_{L^1} \|g(\tau)\|_{\dot C^{\log^{\varkappa}}}d\tau\\
						&\lesssim \log ^{-\varkappa}(2+|\alpha|^{-1}) \da_T(0,g).
					\end{aligned}
                    \end{equation}
                    			For higher order derivative, using integration by parts and applying \eqref{delKL1}, 
                                   \begin{equation}\label{uN2Be}      \begin{aligned}
                        \|(\delta_\alpha \nabla^{m+[\kappa]}u_{N,2,x_0})(t,x)|_{x_0=x}\|_{L_x^p}&\lesssim\int_0^\frac{t}{2}\big\|\sup_{z}|\delta_\alpha\nabla^{m+[\kappa]}\K_{z}|(t,\tau)\big\|_{L^1}\|g(\tau)\|_{L^p}d\tau\\
                        &\quad\quad+\int_\frac{t}{2}^t\big\|\sup_{z}|\delta_\alpha \nabla^{[\kappa]}\K_{z}|(t,\tau)\big\|_{L^1}\|\nabla^mg(\tau)\|_{L^p}d\tau\\
                        &\lesssim |\alpha|^{\kappa-[\kappa]}\left(t^{-\frac{m+\kappa}{s}}\|g\|_{L_T^1L^p}+\int_\frac{t}{2}^t(t-\tau)^{-\frac{\kappa}{s}}\|\nabla^mg(\tau)\|_{L^p}d\tau\right)\\
                        &\lesssim |\alpha|^{\kappa-[\kappa]}t^{-\frac{m+\kappa}{s}}\da_T^p(0,g).
                    \end{aligned}
                    \end{equation}
                    Similarly, we obtain 
                    \begin{align}\label{un2lp}
                       \|(\delta_\alpha \nabla^{m+[\kappa]}u_{N,2,x_0})(t,x)|_{x_0=x}\|_{L_x^p}\lesssim |\alpha|^{\kappa-[\kappa]}|\log t|^{-\varkappa}t^{-\frac{m+\kappa}{s}}\da_T^p(0,g).
                    \end{align}
             	Combining \eqref{uN1lp}--\eqref{un2lp} yields
					\begin{equation}\label{nlin}
						\begin{aligned}
							\sup_{t\in[0,T]}\sup_{\alpha}\left(\frac{{\|(\delta_\alpha u_{N,x_0})|_{x_0=x}\|_{L^\infty}}}{\log^{-\varkappa}(2+|\alpha|^{-1})}+|\log t|^\varkappa t^{\frac{m+\kappa}{s}}\frac{\|(\delta_\alpha \nabla^{m+[\kappa]}u_{N,x_0})|_{x_0=x}\|_{L^\infty}}{|\alpha|^{\kappa-[\kappa]}}\right)
							\lesssim \da_T(f,g),
						\end{aligned}
					\end{equation}
                    \begin{equation}\label{uNBe}
                        \sup_{t\in[0,T]}\left(\|u_{N,x_0}(t,x)|_{x_0=x}\|_{L_x^p}+\sup_{\alpha}t^{\frac{m+\kappa}{s}
                        }\frac{\|(\delta_\alpha \nabla^{m+[\kappa]}u_{N,x_0})(t,x)|_{x_0=x}\|_{L_x^p}}{|\alpha|^{\kappa-[\kappa]}}\right)\lesssim\da_T^p(f,g).
                    \end{equation}
                    \vspace{0.1cm}
                    {\bf Step 3. Estimate of $u_{R,x_0}$}\\
					Finally, we estimate the remainder term $u_{R,x_0}$. Let $0<\sigma<s$. Denote
					\begin{align*}
						& R^1_{x_0}[u](t,x)=(2\pi)^{-\frac{d}{2}}\int _{\mathbb{R}^d}|\xi|^{-\sigma}(\A(t,x_0,\xi)-\A(t,x,\xi))\widehat{u}(t,\xi)e^{ix\cdot\xi} d\xi,\\
						&R^2[u](t,x)=C_{d,\sigma}\int_{\mathbb{R}^d}\int _{\mathbb{R}^d}|\xi|^{-\sigma}(\A(t,x-z,\xi)-\A(t,x,\xi))\widehat{\tau_zu}(t,\xi)e^{ix\cdot\xi}\frac{ d\xi dz}{|z|^{d+\sigma}},
					\end{align*}
					with $\tau_zu(x)=u(x+z)$. Then
					\begin{align}
						u_{R,{x_0}}(t,x)&= -\int_0^{t}\int_{\mathbb{R}^d}\Lambda^\sigma\K_{x_0}(t,\tau,x-z)R_{x_0}^1[u](\tau,z)dzd\tau-\int_0^{t}\int_{\mathbb{R}^d}\K_{x_0}(t,\tau,x-z)R^2[u](\tau,z)dzd\tau\nonumber\\
						&:=  u_{R,{x_0}}^1(t,x)+  u_{R,{x_0}}^2(t,x).\label{urx0}
					\end{align}
					We need the following technical lemma, the proof of which can be derived from singular integral theory. 
					\begin{lemma}\label{intp}
						Let $\mathbf{m}(x,\xi)$ be a Fourier multiplier of order $\sigma>-d$, with\begin{align*}
							\sup_{x}| \nabla_\xi^k\mathbf{m}(x,\xi)|\leq  |\xi|^{\sigma-k},\ \ \ \forall 0\leq k\leq d+|\sigma|+1.
						\end{align*}
						Define
\[
\mathcal T f(x)
=
\int_{\mathbb R^d}
\mathbf m(x,\xi)\widehat f(\xi)e^{ix\cdot \xi}\,d\xi .
\]
Then, for every $1<p<\infty$ and every sufficiently small $\varepsilon>0$,
\[
\|\mathcal T f\|_{L^p}
\lesssim_\varepsilon
\|\Lambda^{\sigma+\varepsilon}f\|_{L^p}^{1/2}
\|\Lambda^{\sigma-\varepsilon}f\|_{L^p}^{1/2}.
\]
					\end{lemma}
					We now return to the estimates of $R_{x_0}^1[u]$ and $R_{x_0}^2[u]$. By Lemma~\ref{intp} and \eqref{intp11}, for any
$\eta'\in[0,1]$, 
					\begin{equation}\label{ptRx}
						\begin{aligned}
							\vert R_{x_0}^1[u](t,x)\vert&\lesssim |x_0-x|^{\eta'}\|u\|_{\dot C^{s-\sigma-\eps}}^{\frac{1}{2}}\|u\|_{\dot C^{s-\sigma+\eps}}^{\frac{1}{2}}
                            \\&\lesssim |x_0-x|^{\eta'}|\log t|^{-\varkappa}t^{-\frac{s-\sigma}{s}}\|u\|_{X_T^\varkappa}.
						\end{aligned}
					\end{equation}
The estimate for higher derivatives follows in the same way, except that
derivatives may also fall on the $x$-dependent symbol. More precisely, when
$\nabla_x^k$ is applied to $R_{x_0}^1[u]$, we obtain terms of the form
\[
(2\pi)^{-\frac d2}i^{k_2}
\int_{\mathbb R^d}
|\xi|^{-\sigma}
\nabla_x^{k_1}(\A(t,x_0,\xi)-\A(t,x,\xi))
\,\widehat u(t,\xi)\,
\xi^{\otimes k_2}
e^{ix\cdot \xi}\,d\xi,
\qquad k_1+k_2=k .
\]
If $k_1=0$, the previous argument gives the factor
\[
|x_0-x|^{\eta'}
|\log t|^{-\varkappa}
t^{-\frac{k+s-\sigma}{s}}
\|u\|_{X_T^\varkappa}.
\]
If $k_1>0$, the derivative falls on the coefficient. In this case we use
Lemma~\ref{intp} with a symbol of order $k_2+s-\sigma$. This yields
\[
\left|
(2\pi)^{-\frac d2}
\int_{\mathbb R^d}
|\xi|^{-\sigma}
\nabla_x^{k_1}\A(t,x,\xi)
\widehat u(t,\xi)
\xi^{\otimes k_2}
e^{ix\cdot \xi}\,d\xi
\right|
\lesssim t^{1/s}
|\log t|^{-\varkappa}
t^{-\frac{k+s-\sigma}{s}}
\|u\|_{X_T^\varkappa}.
\]
Therefore,
\begin{equation}
|\nabla_x^k R_{x_0}^1[u](t,x)|
\lesssim
\left(|x_0-x|^{\eta'}+t^{1/s}\right)
|\log t|^{-\varkappa}
t^{-\frac{k+s-\sigma}{s}}
\|u\|_{X_T^\varkappa}.\label{ptR2x}
\end{equation}
					For $R^2[u]$, we apply Lemma~\ref{intp} with
\[
\mathbf m_z(x,\xi)
=
|\xi|^{-\sigma}
\frac{\A(t,x-z,\xi)-\A(t,x,\xi)}{\min\{1,|z|\}},
\]
for each fixed $z\neq0$. Then applying \eqref{intp11} yields
					\begin{align*}
						\| R^2[u](t)\|_{L^\infty}\lesssim \int_{\mathbb{R}^d}\frac{\min\{1,z\}}{|z|^{d+\sigma}} dz\|\Lambda^{s-\sigma+\epsilon}u(t)\|_{L^\infty}^{\frac{1}{2}}\|\Lambda^{s-\sigma-\epsilon}u(t)\|_{L^\infty}^{\frac{1}{2}} \lesssim  t^{-\frac{s-\sigma}{s}}\|u\|_{X_T^{\varkappa}}.
					\end{align*}
                    Here we used $\sigma<1$, so that
\[
\int_{\mathbb R^d}
\frac{\min\{1,|z|\}}{|z|^{d+\sigma}}\,dz<\infty .
\]
                    Similarly, we can prove for any $0\leq k\leq m$,
                    \begin{align}\label{ptR2k}
						\| \nabla^kR^2[u](t)\|_{L^\infty}\lesssim  t^{-\frac{k+s-\sigma}{s}}\|u\|_{X_T^{\varkappa}}.
					\end{align}
                    Combining \eqref{ptRx}, \eqref{ptR2x}, and Lemma~\ref{lemfourierK}, we first estimate the low-order part of $u_{R,x_0}^1$. Taking $x_0=x$, we obtain
					\begin{align*}
						\vert u_{R,x_0}^1(t,x)|_{x_0=x}\vert&\lesssim \
						\int_0^t |\log \tau|^{-\varkappa}\tau^{-\frac{s-\sigma}{s}}\int_{\mathbb{R}^d}\sup_z |\Lambda^\sigma \K_{z}(t,\tau,x-y)||x-y|^{\eta'} dy d\tau \|u\|_{X_T^\varkappa}
						\\
						&\lesssim  \int_0^t(t-\tau)^{-\frac{\sigma-\eta'}{s}}|\log \tau|^{-\varkappa}\tau^{-\frac{s-\sigma}{s}}d\tau \|u\|_{X_T^\varkappa}\lesssim  |\log t|^{-\varkappa}T^{\frac{\eta'}{s}}\|u\|_{X_T^\varkappa},\\
						\vert \nabla u_{R,x_0}^1(t,x)|_{x_0=x}\vert&\lesssim  
						\int_0^t |\log \tau|^{-\varkappa}\tau^{-\frac{s-\sigma}{s}}\int_{\mathbb{R}^d}\sup_z|\nabla \Lambda^\sigma \K_{z}(t,\tau,x-y)||x-y| dyd\tau \|u\|_{X_T^\varkappa}
						\\
						&\lesssim  \int_0^t(t-\tau)^{-\frac{\sigma}{s}}|\log \tau|^{-\varkappa}\tau^{-\frac{s-\sigma}{s}}d\tau \|u\|_{X_T^\varkappa}\lesssim  |\log t|^{-\varkappa}\|u\|_{X_T^\varkappa},
					\end{align*}
					for any $0<\eta'<\sigma$. \\
                    Interpolating between these two bounds yields, for some $\sigma'\in(0,\sigma)$,
					\begin{align*}
						\sup_{\alpha} \left(\|(\delta_\alpha u_{R,x_0}^1)(t,x)|_{x_0=x}\|_{L^\infty}\log^{\varkappa}(2+|\alpha|^{-1})\right)\lesssim  T^\frac{\sigma'}{s}\|u\|_{X_T^{\varkappa}}.
					\end{align*}
                    We estimate higher order derivatives. By taking $\sigma\in(s-\kappa,s-[\kappa])$, and decomposing the integral interval into $[0,\frac{t}{2}]\cup[\frac{t}{2},t]$, we obtain 
					\begin{align*}
						\vert \delta_\alpha\nabla^{m+[\kappa]} u_{R,x_0}^{1}(t,x)\vert&\lesssim \left|\int_0^{\frac{t}{2}}\int_{\mathbb{R}^d}\delta_\alpha\nabla^{m+[\kappa]}\Lambda_z^\sigma\K_{x_0}(t,\tau,z)R_{x_0}^1(\tau,x-z)dzd\tau\right|\\
						&\quad+\left|\int_{\frac{t}{2}}^t\int_{\mathbb{R}^d}\delta_\alpha\nabla^{[\kappa]}\Lambda_z^\sigma\K_{x_0}(t,\tau,z)\nabla^mR_{x_0}^1(\tau,x-z)dzd\tau\right|\\
						&\lesssim (|\alpha|^{\kappa-[\kappa]}+|x_0-x|^{\kappa-[\kappa]})T^{\frac{\kappa-[\kappa]}{s}}|\log t|^{-\varkappa}t^{-\frac{m+\kappa}{s}}\|u\|_{X_T^{\varkappa}}.
					\end{align*}
					So we have proved
					\begin{align}\label{ur12}
						&\sup_{t\in[0,T]}\sup_{\alpha} \left(\frac{\|(\delta_\alpha u_{R,x_0}^1)(t,x)|_{x_0=x}\|_{L^\infty}}{\log^{-\varkappa}(2+|\alpha|^{-1})}+ |\log t|^\varkappa t^{\frac{m+\kappa }{s}}\frac{\|(\delta_\alpha\nabla^{m+[\kappa]} u_{R,x_0}^1)(t,x)|_{x_0=x}\|_{L^\infty}}{|\alpha|^{\kappa-[\kappa]}}\right) \nonumber\\
						&\quad\quad\lesssim T^\frac{\sigma'}{s}\|u\|_{X_T^\varkappa},
					\end{align}
                    for some $\sigma'\in(0,\sigma)$.
                    For the corresponding $L^p$ estimates of $u^1_{R,x_0}$, note that by Lemma \ref{intp} and similarly to \eqref{ptRx}, we have
                    \begin{equation}\label{S2eq1}
                        \begin{aligned}
                            &\|R_{x_0}^1[u](t,x-z)|_{x_0=x}\|_{L_x^p}\lesssim |z|^{\sigma'}t^{-\frac{s-\sigma}{s}}\|u\|_{Y_T^p},\\
                            &\|\nabla^kR_{x_0}^1[u](t,x-z)|_{x_0=x}\|_{L_x^p}\lesssim \left(|z|^{\sigma'}+t^{\frac{1}{s}}\right)t^{-\frac{s-\sigma}{s}}\|u\|_{Y_T^p},\quad\forall 0\leq k\leq m,
                        \end{aligned}
                    \end{equation}
                   for some $\sigma'\in(0,\sigma)$, which leads to
                    \begin{equation*}
                        \begin{aligned}
                            \|u_{R,x_0}^1(t,x)|_{x_0=x}\|_{L^p}&\lesssim \int_0^t\|\sup_{z}|\Lambda^\sigma\K_{z}|(t,\tau,y)|y|^{\eta'}\|_{ L_y^1}\||y|^{-\eta'}(R_{x_0}^1[u])(\tau,x-y)|_{x_0=x}\|_{L_z^\infty L_x^p}d\tau\\
                            &\lesssim T^{\frac{\sigma'}{s}}\|u\|_{Y_T^p}.
                        \end{aligned}
                    \end{equation*}
                    Similarly, for the estimate of $\dot B_{p,\infty}^{m+\kappa}$ norm, by \eqref{S2eq1}, we have
                    \begin{equation*}
                        \begin{aligned}
                            &\|(\delta_\alpha \nabla^{m+[\kappa]}u_{R,x_0}^1)(t,x)|_{x_0=x}\|_{L_x^p}\\
                            &\quad\lesssim \int_0^\frac{t}{2}\left\|\sup_{z}|\delta_\alpha\nabla_z^{m+[\kappa]}\Lambda^\sigma\K_{z}|(t,\tau,y)|y|^{\eta'}\right\|_{L_y^1}\left\||y|^{-\eta'}(R_{x_0}^1[u])|_{x_0=x}(\tau,x-y)\right\|_{L_y^\infty L_x^p}d\tau \\
                            &\quad\quad+\int_\frac{t}{2}^t\left\|\sup_{z}|\delta_\alpha\Lambda^\sigma\nabla^{[\kappa]}\K_{z}|(t,\tau,y)(\nabla_z^mR_{x_0}^1[u])|_{x_0=x}(\tau,x-y)\right\|_{L_y^1L_x^p}d\tau\\
                            &\quad\lesssim |\alpha|^{\kappa-[\kappa]}t^{-\frac{m+\kappa-\sigma'}{s}}\|u\|_{Y_T^p}.
                        \end{aligned}
                    \end{equation*}
					So we have proved 
                    \begin{equation}\label{uR1}
                    \sup_{t\in[0,T]}\left(\|u_{R,x_0}^1(t,x)|_{x_0=x}\|_{L_x^p}+t^{\frac{m+\kappa}{s}}\frac{\|(\delta_\alpha\nabla^{m+[\kappa]} u_{R,x_0}^1)(t,x)|_{x_0=x}\|_{L_x^p}}{|\alpha|^{\kappa-[\kappa]}}\right)\lesssim T^{\frac{\sigma'}{s}}\|u\|_{Y_T^p},
                    \end{equation}
                   for some $\sigma'\in(0,\sigma)$. The estimates for $u_{R,x_0}^2$ follow similarly from the $u_{N,2,x_0}$ part of \eqref{nlin} and \eqref{ptR2k}, say,
					\begin{align}\label{ur22}
						&\sup_{t\in[0,T]}\sup_{\alpha} \left(\frac{\|(\delta_\alpha u_{R,x_0}^2)(t,x)|_{x_0=x}\|_{L^\infty}}{\log^{-\varkappa}(2+|\alpha|^{-1})}+ |\log t|^\varkappa t^{\frac{m+\kappa }{s}}\frac{\|(\delta_\alpha\nabla^{m+[\kappa]} u_{R,x_0}^2)(t,x)|_{x_0=x}\|_{L^\infty}}{|\alpha|^{\kappa-[\kappa]}} \right)\nonumber\\
						&\quad\quad\quad\quad\quad\quad\lesssim \int_0^T\|R_2[u](\tau)\|_{\dot C^{\log^\varkappa}}d\tau+\sup_{t\in[0,T]}|\log t|^\varkappa t^{\frac{m}{s}+1}\|R_2[u](t)\|_{\dot C^m}\lesssim T^{\frac{\sigma}{s}}\|u\|_{X_T^{\varkappa}}.
					\end{align}
                    Similarly, by \eqref{un2}, \eqref{uN2Be} and Lemma \ref{intp}, we have
                    \begin{align}\label{uR2}
						&\sup_{t\in[0,T]}\left(\|(u_{R,x_0}^2(t,x))|_{x_0=x}\|_{L_x^p}+t^{\frac{m+\kappa}{s}}\frac{\|(\delta_\alpha\nabla^{m+[\kappa]} u_{R,x_0}^2)(t,x)|_{x_0=x}\|_{L_x^p}}{|\alpha|^{\kappa-[\kappa]}}\right)\nonumber\\
						&\quad\quad\quad\quad\quad\quad\lesssim \int_0^T\|R_2[u](\tau)\|_{L^p}d\tau+\sup_{t\in[0,T]}t^{\frac{m}{s}+1}\|R_2[u](t)\|_{W^{m,p}}\lesssim T^{\frac{\sigma}{s}}\|u\|_{Y_T^p}.
					\end{align}
					We obtain from \eqref{ur12} and \eqref{ur22} that 
					\begin{equation}\label{ur2}
						\begin{aligned}
							\sup_{t\in[0,T]}\sup_\alpha&\left(\frac{\|(\delta_\alpha u_{R,x_0})(t,x)|_{x_0=x}\|_{L^\infty_x}}{\log^{-\varkappa}(2+|\alpha|^{-1})}+|\log t|^\varkappa   t^\frac{m+\kappa}{s}\frac{\|(\delta_\alpha \nabla_x ^{m+[\kappa]}u_{R,x_0})(t,x)|_{x_0=x}\|_{L^\infty_x}}{|\alpha|^{\kappa-[\kappa]}}\right)\\
							&\quad\quad\quad\quad\lesssim T^\frac{\sigma'}{s} \|u\|_{X_T^{\varkappa}},
						\end{aligned}
					\end{equation}
                    and from \eqref{uR1} and \eqref{uR2} that
                    \begin{equation}\label{esuR}
                        \sup_{t\in[0,T]}\left(\|u_{R,x_0}(t,x)|_{x_0=x}\|_{L_x^p}+t^{\frac{m+\kappa}{s}}\sup_\alpha\frac{\|(\delta_\alpha\nabla^{m+[\kappa]} u_{R,x_0})(t,x)|_{x_0=x}\|_{L_x^p}}{|\alpha|^{\kappa-[\kappa]}}\right)\lesssim T^{\frac{\sigma'}{s}}\|u\|_{Y_T^p},
                    \end{equation}
					for some $\sigma'\in(0,\sigma)$. Combining \eqref{ulhol}, \eqref{nlin} and   \eqref{ur2} yields
					\begin{align*}
						\|u\|_{X_T^\varkappa }
						&\lesssim 
						\|u_0\|_{\dot C^{\log^\varkappa}}+\da_T(f,g)+T^\frac{\sigma'}{s}\|u\|_{X_T^\varkappa },
					\end{align*}
                    and \eqref{uLLp1}, \eqref{uNBe}, \eqref{esuR} give 
                    \begin{equation*}
                        \|u\|_{Y_T^p}\lesssim \|u_0\|_{L^p}+\da_T^p(f,g)+T^{\frac{\sigma'}{s}}\|u\|_{Y_T^p}.
                    \end{equation*}
					These imply \eqref{relog} and \eqref{reLp} by taking $T$ small enough.
					\vspace{0.3cm}
				\end{proof}

                \section{Contour dynamics formulation}\label{seccon}

The purpose of this section is to extract the leading singular part of the interface equation. This is one of the main contributions of the paper. In the general two-phase problem with viscosity jump and rigid boundaries, the interface velocity is not given by a closed explicit contour equation. Instead, it is obtained through the elliptic transmission problem \eqref{sysq}. We show that, nevertheless, the evolution contains an explicit third-order parabolic principal operator generated by surface tension, while all effects coming from the viscosity jump, density contrast, and rigid boundaries can be placed into lower-order or perturbative terms.

The key point is to compare the true transmission velocity with a frozen-coefficient velocity obtained by freezing the geometry at the local slope of the interface. This allows us to identify the leading order operator in a singular-integral form. The resulting formula \eqref{evoeta'} is the starting point for the Schauder estimates and the nonlinear well-posedness argument in the logarithmically critical space.

Under the uniform separation condition from the fixed boundaries, the boundary contributions from $\Gamma^\pm$ do not affect the leading singular structure. Similarly, the density contrast term $\varrho_0\eta$ is lower order compared with the surface-tension contribution
\[
-\operatorname{div}\left(\frac{\nabla\eta}{\langle\nabla\eta\rangle}\right).
\]
We therefore first straighten the moving interface and then freeze the coefficients of the resulting elliptic transmission problem. More precisely, 
				define \begin{align}\label{relavq}
					v^\pm(x,z)=v^\pm_\eta(x,z)=q^\pm(x,z+\eta(x)).
				\end{align}
				Then 
				\begin{align*}
					&\nabla_{x,z}  q^\pm(x,z)=\M_0(\nabla\eta(x))(\nabla_{x,z}v^\pm)(x,z-\eta(x)),\\ &\Delta_{x,z} q^\pm(x,z)=\left(\operatorname{div}_{x,z}(\M(\nabla\eta(x))\nabla_{x,z}v^\pm)\right)(x,z-\eta(x)),
				\end{align*}
				with
				\begin{equation}\label{ndefM}
					\begin{aligned}
						\M_0(\nabla\eta)&=\left(\begin{array}{cc}
							\mathrm{	Id}&-\nabla \eta\\
							0&1
						\end{array}\right),\\
						\M(\nabla\eta)&=(\M_0(\nabla\eta))^\top\M_0(\nabla\eta)=\left(\begin{array}{cc}
							\mathrm{	Id}&-\nabla \eta\\
							-(\nabla \eta)^\top&1+|\nabla\eta|^2
						\end{array}\right).
					\end{aligned}    
				\end{equation}
				We obtain
				\begin{equation}\label{elleqeta}
					\begin{aligned}
						&\operatorname{div}_{x,z}(\M(\nabla \eta)\nabla_{x,z} v^\pm )=0,\ \ \text{in}\ \tilde \Omega_\eta^\pm,\\
						&v^--v^+=-\operatorname{div}\left(\frac{\nabla \eta}{\langle\nabla \eta\rangle}\right)+\varrho_0\eta, \quad \text { on } \{z=0\},\\
						&e_{d+1}\cdot (\M(\nabla \eta)(\mu_+^{-1}\nabla_{x,z}v^+-\mu_-^{-1}\nabla_{x,z}v^-))=0, \quad \text { on } \{z=0\},\\
						&\tilde \nu^\pm\cdot (\M(\nabla \eta)\nabla_{x,z} v^\pm) =0,\ \ \ \ \ \ \ \ \ \text{on} \ \tilde \Gamma^\pm_\eta,
					\end{aligned}
				\end{equation}
				where $e_{d+1}$ denotes the (d+1)-th unit vector, 
				\begin{equation}\label{defdomain}
					\begin{aligned}
						&\tilde \Omega_{\eta}^{+}=\left\{(x, z) \in \mathbb{R}^{d} \times \mathbb{R}: 0<z<\underline{b}^{+}(x)-\eta(t,x)\right\}, \\
						&\tilde \Omega_{\eta}^{-}=\left\{(x, z) \in \mathbb{R}^{d} \times \mathbb{R}: \underline{b}^{-}(x)-\eta(t, x)<z<0\right\},\\
						&\tilde \Gamma ^\pm_\eta=\{(x,z):z=\underline{b}^\pm(x)-\eta(t,x)\},
					\end{aligned} 
				\end{equation}
                and $\tilde \nu^\pm=\frac{1}{\langle \nabla \left(\underline{b}^{\pm}(x)-\eta(t,x)\right) \rangle}\left(-\nabla \underline{b}^{\pm}(x)+\nabla \eta(t,x),1\right)$ denote the outward normal vector of $\tilde \Gamma_\eta ^\pm$. 
				Note that we transform the effect of free interface $\Sigma_t=\{(x,z):z=\eta(x)\}$ in the system \eqref{sysq} to the variable coefficient $\M(\nabla\eta)$ in the system \eqref{elleqeta}. 
				We remark that $v^\pm$ is not differentiable at $z=0$. When evaluating a function with a ``$+$" or ``$-$" superscript at $z=0$, it implies approaching the limit from either the positive or negative direction. \textit{i.e.} $ v^\pm(x,0)=\lim_{z=0^\pm}v^\pm(x,z)$.
				
				With this notation, the interface equation \eqref{evo} becomes
				\begin{equation}\label{inievo}
					\begin{aligned}
						\partial_t \eta (x)&=-\frac{1}{\mu_+}e_{d+1}\cdot (\M(\nabla\eta)\nabla_{x,z}v^+(x,z)|_{z=0})
						\\&=\frac{1}{\mu_+}\nabla \eta(x)\cdot \nabla _xv^+(x,z)|_{z=0} -\frac{1}{\mu_+}\langle\nabla \eta(x)\rangle^2\partial_zv^+(x,z)|_{z=0},\\
						\eta|_{t=0}&=\eta_0.
					\end{aligned}
				\end{equation}
                Here and throughout, traces at $z=0$ are understood from the corresponding side of the interface. The formulation \eqref{inievo} is well-defined by the existence and uniqueness of weak solutions to the elliptic transmission problem \eqref{elleqeta}.

                \begin{remark}
The weak solvability of \eqref{elleqeta} follows from a standard variational
argument. Indeed, after removing the prescribed jump
\[
g_\eta:=-\operatorname{div}\left(\frac{\nabla\eta}{\langle\nabla\eta\rangle}\right)
+\varrho_0\eta
\]
across $\{z=0\}$, the transmission system can be written as a divergence-form
elliptic equation with piecewise coefficients and a source supported on the flat
interface.

More precisely, let $v_0^+$ solve
\[
\begin{cases}
\operatorname{div}_{x,z}\bigl(\M(\nabla\eta)\nabla_{x,z}v_0^+\bigr)=0,
& \text{in } \tilde\Omega_\eta^+,\\
v_0^+=g_\eta,
& \text{on } \{z=0\},\\
\widetilde\nu^+\cdot\M(\nabla\eta)\nabla_{x,z}v_0^+=0,
& \text{on } \tilde\Gamma_\eta^+ .
\end{cases}
\]
Define
\[
u^+:=v^+ + v_0^+,\qquad u^-:=v^- .
\]
Since \(v^- - v^+=g_\eta\) on \(\{z=0\}\), we have
\[
u^+=u^-
\qquad\text{on } \{z=0\}.
\]
Thus the function \(u\), defined by
\[
u=
\begin{cases}
u^+, & \text{in } \tilde\Omega_\eta^+,\\
u^-, & \text{in } \tilde\Omega_\eta^-,
\end{cases}
\]
is continuous across the flat interface. It satisfies, in the sense of distributions,
\[
\operatorname{div}_{x,z}\big(\overline\M(\nabla\eta)\nabla_{x,z}u\big)
=
F_\eta\,\delta_{z=0}
\qquad\text{in } \tilde\Omega_\eta,
\]
where $
F_\eta(x)
=
e_{d+1}\cdot
\mu_+^{-1}\M(\nabla\eta(x))
\nabla_{x,z}v_0^+(x,0),
$
together with the boundary condition
\[
\widetilde\nu^\pm\cdot
\overline\M(\nabla\eta)\nabla_{x,z}u=0
\qquad\text{on } \tilde\Gamma_\eta^\pm .
\]
Here
\[
\tilde\Omega_\eta
=
\left\{
(x,z):
\underline b^-(x)-\eta(t,x)<z<\underline b^+(x)-\eta(t,x)
\right\},
\]
and
\[
\overline\M(\nabla\eta)
=
\begin{cases}
\mu_+^{-1}\M(\nabla\eta),
& 0<z<\underline b^+(x)-\eta(t,x),\\[1mm]
\mu_-^{-1}\M(\nabla\eta),
& \underline b^-(x)-\eta(t,x)<z<0.
\end{cases}
\] Since $\M(\nabla\eta)$ is
uniformly elliptic under the Lipschitz bound on $\eta$, the standard
variational theory for uniformly elliptic divergence-form equations gives weak
solvability, up to the usual additive constant.
\end{remark}

We next introduce the frozen transmission problem. The proof relies on a
freezing-coefficient argument, in the spirit of \cite{CHN1},
adapted here to the elliptic transmission structure. The frozen problem can be
solved explicitly at the level of the boundary kernel, and its contribution
yields the leading singular operator in the evolution of $\eta$. The difference
between the true velocity and the frozen velocity will be treated as a
remainder.\\

To isolate the principal part, we freeze the matrix $\M(\nabla\eta)$ at a
constant slope. More precisely, for each fixed $\b\in\mathbb R^d$, let
$\tilde v^\pm_\b=\tilde v^\pm_\b[\mathcal G]$ solve
				\begin{equation}\label{eqmain}
					\begin{aligned}
						&\operatorname{div}_{x,z}(\M(\b)\nabla_{x,z}\tilde v^\pm_\b )=0,\ \ \text{in}\ \mathbb{R}^{d+1}_\pm,\\
						&\tilde v^-_\b-\tilde v^+_\b=\mathcal{G} \quad \text { on } \{z=0\},\\
						&e_{d+1}\cdot (\M(\b)(\mu_+^{-1}\nabla_{x,z}\tilde v^+_\b-\mu_-^{-1}\nabla_{x,z}\tilde v^-_\b))=0 \quad \text { on } \{z=0\},
					\end{aligned}
				\end{equation}
				where we denote $\mathcal{G}=\mathcal{G}[\eta]=-\operatorname{div}\left(\frac{\nabla \eta}{\langle\nabla \eta\rangle}\right)$.
Taking the Fourier transform in the tangential variable $x$, and denoting
				\begin{align*}
					V^\pm_\b(\xi,z)=(2\pi)^{-\frac{d}{2}}\int_{\mathbb{R}^d}\tilde v_\b ^\pm(x,z)e^{-i\xi\cdot x}dx,\ \ \ \xi\in\mathbb{R}^d,\ \ \ \quad z\in\mathbb{R}^\pm.
				\end{align*}
				One has 
				\begin{align*}
					&	(1+|\b|^2)\partial_z^2V^\pm_\b(\xi,z)-|\xi|^2V^\pm_\b(\xi,z)-2i\b\cdot \xi \partial_z V_\b^\pm(\xi,z)=0,\ \ \ (\xi,z)\in\mathbb{R}^{d+1}_\pm,\\
					&	V^-_\b(\xi,0)-V^+_\b(\xi,0)= \hat{\mathcal{G}}(\xi),\\
					&\left(-\frac{ V^+_\b}{\mu_+}+\frac{V^-_\b}{\mu_-}\right)i\xi\cdot \b+\left(\frac{\partial_z V^+_\b}{\mu_+}-\frac{\partial_zV^-_\b}{\mu_-}\right)(|\b|^2+1)=0,\ \ \ \text{on}\ \{z=0\}.
				\end{align*}
				By a direct calculation, we obtain the eigenvalues of the above ODE:
				\begin{align}\label{deflanb}
					\lambda^\pm=	\lambda^\pm(\xi,\b)=\frac{i\b\cdot \xi\pm \sqrt{\langle \b\rangle^2|\xi|^2-(\b\cdot \xi)^2}}{\langle \b\rangle^2}.
				\end{align}
				Imposing decay as $z\to\pm\infty$, we obtain
				\begin{align}\label{deffour}
					V^+_\b(\xi,z)=c^+\hat{\mathcal{G}}(\xi)e^{\lambda^-z}, \ \ \ \ \ V^-_\b(\xi,z)=c^-\hat{\mathcal{G}}(\xi)e^{\lambda^+z},
				\end{align}
				where 
				\begin{align*}
					c^\pm=\frac{\pm\mu_\pm  }{\mu_++\mu_-}.
				\end{align*}
				Taking inverse Fourier transform to obtain 
				\begin{align}\label{foutv}
					\tilde v^\pm_\b(x,z)=(2\pi)^{-\frac{d}{2}}c^\pm \int_{\mathbb{R}^d} \hat {\mathcal{G}}(\xi)e^{\lambda^\mp(\xi,\b) z}e^{ix\cdot \xi}d\xi,\quad\quad\quad z\in\mathbb{R}^\pm.
				\end{align}
				More precisely, we have
				\begin{align*}
					\tilde v^\pm_{\b}(x,z)=c^\pm \mathcal{V}_\b[\mathcal{G}](x,z),\quad\quad z\in\mathbb{R}^\pm,
				\end{align*}
				where 
				\begin{align}\label{defvb}
					\mathcal{V}_\b[\mathcal{G}](x,z)=\int_{\mathbb{R}^d} K_\b(x-y,z) \mathcal{G}(y)dy,
				\end{align}
				with
				\begin{equation}\label{defKor}
					\begin{aligned}
						K_\b(x,z)=\begin{cases}
							\frac{1}{(2\pi)^\frac{d}{2}}\int_{\mathbb{R}^d}\exp\left(ix\cdot\xi+\lambda^-(\xi,\b)z\right)d\xi,& z\geq0,\\
							\frac{1}{(2\pi)^\frac{d}{2}}\int_{\mathbb{R}^d}\exp\left(ix\cdot\xi+\lambda^+(\xi,\b)z\right)d\xi,& z<0.
						\end{cases}	
					\end{aligned}
				\end{equation}
				See Lemma \ref{Z10}  for the explicit formula of $	K_\b(x,z)$.\\
				Moreover, it is easy to check that 
				\begin{align}\label{dx0}
					\frac{\left(	\nabla_x\tilde v_\b^+\right)(x,0)}{\mu_+}=\frac{\nabla \mathcal{G}(x)}{\mu_++\mu_-}.
				\end{align}
				To obtain $ \partial_z\tilde v^-_\b|_{(x,0)}$, we proceed by analyzing its Fourier transform:
                \begin{align}\label{pztilvb}
						\partial_z\tilde {v}_{\b}^{\pm}[f]=c^{\pm}\mathcal{F}^{-1}\left(\lambda^{\mp}(\xi,\b)e^{\lambda^{\mp}(\xi,\b)z}\hat{\mathcal{G}}[f](\xi)\right).
					\end{align} 
                    \vspace{0.2cm}\\
We first record two elementary facts which will be used to identify the normal
derivative of the frozen transmission solution at the flat interface. The first
lemma computes the anisotropic constant appearing in the singular-integral
representation of the Fourier multiplier
\[
\frac{\sqrt{\langle\b\rangle^2|\xi|^2-(\b\cdot\xi)^2}}{\langle\b\rangle^2}.
\]
The second lemma then converts this multiplier into a principal value operator
in physical space.	\begin{lemma}\label{lemz0} Let $e\in \mathbb{S}^{d-1}$, $\b\in\mathbb{R}^d$. There holds, 
						\begin{align}\label{contc}
							\tilde c_d\int_{\mathbb{R}^d} \frac{1-\cos(e\cdot\alpha)}{\left((\alpha\cdot \b)^2+|\alpha|^2\right)^{\frac{d+1}{2}}}d\alpha=\frac{(\langle \b\rangle^2-(\b\cdot e)^2)^{\frac{1}{2}}}{\langle \b\rangle^2},
						\end{align}
						where $\tilde c_d=\pi^{-\frac{d+1}{2}}\Gamma\left(\frac{d+1}{2}\right)>0$.
					\end{lemma}
                  		\begin{proof} By the normalization
\[
\tilde c_d
=
\pi^{-\frac{d+1}{2}}
\Gamma\left(\frac{d+1}{2}\right),
\]
we have
\[
\tilde c_d
\int_{\mathbb R^d}
\frac{1-\cos(\alpha_1)}{|\alpha|^{d+1}}\,d\alpha
=1.
\]
For $d=1$, since $e=\pm1$, the identity reduces to
\[
\tilde c_1
\int_{\mathbb R}
\frac{1-\cos\alpha}{(1+|\b|^2)|\alpha|^2}\,d\alpha
=
\frac1{1+|\b|^2},
\]
which proves the claim.\\
Assume now $d\ge2$. If $\b$ is parallel to $e$, the result follows directly by a
one-dimensional scaling in the $e$ direction, or by continuity from the
non-parallel case. Hence we may assume that $\b$ is not parallel to $e$.

We first use the following elementary rotation identity: for any
$\theta_1,\theta_2\in\mathbb S^{d-1}$ and any integrable function $G$,
					\begin{equation}\label{Z2}
						\int_{\mathbb{R}^d} G(\theta_1\cdot x,\theta_2\cdot x,|x|)dx=\int_{\mathbb{R}^d} G((1-(\theta_1\cdot\theta_2)^2)^{\frac{1}{2}}x_1+(\theta_1\cdot\theta_2)x_2,x_2,|x|)dx.
					\end{equation} 
					Indeed, the unit vectors 
					\begin{equation*}
						\theta_2, \tilde{\theta}_1= \frac{\theta_1-(\theta_1\cdot\theta_2)\theta_2}{(1-(\theta_1\cdot\theta_2)^2)^{\frac{1}{2}}}\in \mathbb{S}^{d-1},
					\end{equation*}
					are orthogonal. 
					Hence,  by rotation, 
					\begin{align*}
						\int_{\mathbb{R}^d} G(\theta_1\cdot x,\theta_2\cdot x,|x|)dx&=\int_{\mathbb{R}^d} G((\tilde{\theta}_1\cdot\theta_1)x_1+(\theta_1\cdot\theta_2)x_2,x_2,|x|)dx\\&=\int_{\mathbb{R}^d} G((1-(\theta_1\cdot\theta_2)^2)^{\frac{1}{2}}x_1+(\theta_1\cdot\theta_2)x_2,x_2,|x|)dx.
					\end{align*}
                      From \eqref{Z2}, we deduce that 
						\begin{align*}
							&\int_{\mathbb{R}^d} \frac{1-\cos(e\cdot \alpha)}{\left((\alpha\cdot \b)^2+|\alpha|^2\right)^{\frac{d+1}{2}}}d\alpha=\int_{\mathbb{R}^d} \frac{1-\cos(\alpha_1)}{\left(((\b\cdot e)\alpha_1+(|\b|^2-(\b\cdot e)^2)^{\frac{1}{2}}\alpha_2)^2+|\alpha|^2\right)^{\frac{d+1}{2}}}d\alpha\\&=\int_{\mathbb{R}^d} \frac{1-\cos(\alpha_1)}{\left((1+(\b\cdot e)^2)\alpha_1^2+(\langle \b\rangle^2-(\b\cdot e)^2)\alpha_2^2+2(\b\cdot e)(|\b|^2-(\b\cdot e)^2)^{\frac{1}{2}}\alpha_1\alpha_2+|\alpha'|^2\right)^{\frac{d+1}{2}}}d\alpha,
						\end{align*}
						where we denote $\alpha'=\alpha-\alpha_1e_1-\alpha_2e_2$.
						Note that 
						\begin{align*}
							&\left(1+(\b\cdot e)^2\right)\alpha_1^2+\left(\langle \b\rangle^2-(\b\cdot e)^2\right)\alpha_2^2+2(\b\cdot e)\left(|\b|^2-(\b\cdot e)^2\right)^{\frac{1}{2}}\alpha_1\alpha_2\\&=\frac{\langle \b\rangle^2}{\langle \b\rangle^2-(\b\cdot  e)^2}\alpha_1^2+\left((\langle \b\rangle^2-(\b\cdot e)^2)^{\frac{1}{2}}\alpha_2+\frac{(\b\cdot e)(|\b|^2-(\b\cdot e)^2)^{\frac{1}{2}}}{(\langle \b\rangle^2-(\b\cdot e)^2)^{\frac{1}{2}}}\alpha_1\right)^2.
						\end{align*}
						Using elementary linear transformations, we obtain 
						\begin{align*}
							\int_{\mathbb{R}^d} \frac{1-\cos(e\cdot \alpha)}{\left((\alpha\cdot \b)^2+|\alpha|^2\right)^{\frac{d+1}{2}}}d\alpha	&=\frac{1}{(\langle \b\rangle^2-(\b\cdot e)^2)^{\frac{1}{2}}}\int_{\mathbb{R}^d} \frac{1-\cos(\alpha_1)}{\left(\frac{\langle \b\rangle^2}{\langle \b\rangle^2-(\b\cdot e)^2}\alpha_1^2+\alpha_2^2+|\alpha'|^2\right)^{\frac{d+1}{2}}}d\alpha\\&
							=\frac{(\langle \b\rangle^2-(\b\cdot e)^2)^{\frac{1}{2}}}{\langle \b\rangle^2} \int_{\mathbb{R}^d} \frac{1-\cos(\alpha_1)}{|\alpha|^{d+1}}d\alpha=\tilde c_d^{-1}\frac{(\langle \b\rangle^2-(\b\cdot e)^2)^{\frac{1}{2}}}{\langle \b\rangle^2}.
						\end{align*}
						This completes the proof of the lemma.
					\end{proof}
					\begin{lemma}\label{lemfdz0}
						Let $\b\in \mathbb{R}^d$ and $	\lambda^\pm(\xi,\b)=\frac{i\b\cdot\xi\pm\sqrt{(1+|\b|^2)|\xi|^2-(\b\cdot \xi)^2}}{1+|\b|^2}$. Then
						\begin{align}\label{opL}
							\mathfrak{L}_{\pm,\b} f(x):=\mathcal{F}^{-1}\left(\lambda^\pm(\xi,\b)\hat f(\xi)\right)(x)=\frac{\b\cdot\nabla f(x)}{\langle \b\rangle^2}\pm\tilde c_d \mathrm{P.V.}\int_{\mathbb{R}^d} \frac{\delta_\alpha f(x)}{\langle \hat\alpha\cdot \b\rangle^{d+1}}\frac{d\alpha}{|\alpha|^{d+1}},
						\end{align}
						where $\tilde c_d $ is the positive constant in \eqref{contc}.
					\end{lemma}
                   
					\begin{proof}
						We have
						\begin{equation}\label{fopz}
							\begin{aligned}
								&	\mathcal{F}^{-1}\left(\lambda^\pm(\xi,\b)\hat f(\xi)\right)=\frac{1}{(2\pi)^\frac{d}{2}}\int_{\mathbb{R}^d}\lambda^\pm(\xi,\b)\hat f (\xi)\exp\left(ix\cdot\xi\right)d\xi \\
								&\ \ \ \ \ =\frac{\b\cdot\nabla f(x)}{\langle \b\rangle^2}\pm\frac{1}{(2\pi)^\frac{d}{2}}\int_{\mathbb{R}^d}\frac{\sqrt{\langle \b\rangle^2|\xi|^2-(\b\cdot\xi)^2}}{\langle \b\rangle^2}\hat{f}(\xi)\exp\left(ix\cdot \xi\right)d\xi .
							\end{aligned}
						\end{equation}
						 By Lemma~\ref{lemz0}, with $e=\hat\xi$, we obtain
\[
\frac{\sqrt{\langle \b\rangle^2|\xi|^2-(\b\cdot\xi)^2}}{\langle \b\rangle^2}
=
\tilde c_d |\xi|
\int_{\mathbb R^d}
\frac{1-\cos(\hat\xi\cdot\alpha)}
{\langle\hat\alpha\cdot\b\rangle^{d+1}}
\frac{d\alpha}{|\alpha|^{d+1}}.
\]
Equivalently, since the imaginary part is odd in $\alpha$,
\[
\frac{\sqrt{\langle \b\rangle^2|\xi|^2-(\b\cdot\xi)^2}}{\langle \b\rangle^2}
=
\tilde c_d
\int_{\mathbb R^d}
\frac{1-e^{-i\xi\cdot\alpha}}
{\langle\hat\alpha\cdot\b\rangle^{d+1}}
\frac{d\alpha}{|\alpha|^{d+1}},
\]
in the principal value sense.
						This implies that 
						\begin{align*}
							\frac{1}{(2\pi)^\frac{d}{2}}\int_{\mathbb{R}^d}\frac{\sqrt{\langle \b\rangle^2|\xi|^2-(\b\cdot\xi)^2}}{\langle \b\rangle^2}\hat f(\xi)\exp(ix\cdot\xi)d\xi=\tilde c_d\int_{\mathbb{R}^d} \frac{\delta_\alpha f(x)}{\langle \hat\alpha\cdot \b\rangle^{d+1}}\frac{d\alpha}{|\alpha|^{d+1}}.
						\end{align*}
						Combining this with \eqref{fopz} yields \eqref{opL}.
						This completes the proof of the lemma.
					\end{proof}\vspace{0.2cm}\\
				Applying Lemma \ref{lemfdz0} to obtain that 
				\begin{equation}\label{hz0}
					\begin{aligned}
						\lim_{z\to 0^+}	\int_{\mathbb{R}^d} \partial_zK_\b(x-y,z)\mathcal{G}(y)dy&=\frac{1}{(2\pi)^\frac{d}{2}}\int_{\mathbb{R}^d}\lambda^-(\xi,\b)\hat {\mathcal{G}}(\xi)\exp\left(ix\cdot\xi\right)d\xi \\
						&=\tilde c_d\int_{\mathbb{R}^d} \frac{\delta_\alpha \mathcal{G}(x)}{\langle \hat\alpha\cdot \b\rangle^{d+1}}\frac{d\alpha}{|\alpha|^{d+1}}+\frac{\b\cdot\nabla \mathcal{G}(x)}{\langle \b\rangle^2}.
					\end{aligned}
				\end{equation}
				Here $\tilde c_d$ is a constant depends only on dimension $d$.
				This implies that
				\begin{align}\label{z0}
					\frac{	\left(\partial_z \tilde v_\b^+\right)(x,0)}{\mu_+}=\frac{1}{\mu_++\mu_-}\left(\tilde c_d \int _{\mathbb{R}^d}\frac{\delta_\alpha\mathcal{G}(x)}{\langle\hat \alpha \cdot \b\rangle^{d+1}}\frac{d\alpha}{|\alpha|^{d+1}}+\frac{\b\cdot \nabla \mathcal{G}(x)}{\langle \b\rangle^2}\right).
				\end{align}
				For simplicity, denote 
				\begin{align}\label{ndefB}
					B(\b)=\frac{\mathrm{Id}}{\langle \b\rangle}-\frac{\b\otimes \b}{\langle\b\rangle^3}, \ \ \ \forall \b\in\mathbb{R}^d.
				\end{align}
				It is easy to check that \begin{align}\label{Bellip}
					\frac{1}{\langle \b\rangle^3} \mathrm{Id}\leq   B(\b)\leq 	\frac{1}{\langle \b\rangle} \mathrm{Id}.
				\end{align} And   \begin{align}\label{ndefG}
					\mathcal{G}[f]=-\operatorname{div}\left(\frac{\nabla f }{\langle\nabla f \rangle}\right) =-B(\nabla f):\nabla ^2 f=-\frac{\Delta f}{\langle \nabla f\rangle}+\frac{(\nabla f)^\top  \nabla ^2 f\nabla f}{\langle \nabla f\rangle^3}.
				\end{align}
				
				We add subscript in order to emphasize the dependence of the solution to elliptic systems on data. More precisely,  we denote $v_h^\pm(x,z)$ the solution to system \eqref{elleqeta} with data $\mathcal{G}=\mathcal{G}[h]$, 
				for any $h: \mathbb{R}^d\to \mathbb{R}$.
				With a slight abuse of notation, we simply denote 
				\begin{align}\label{deftvb}
					\tilde v^\pm_{\b}[\eta](x,z)=c^\pm  {\mathcal{V}}_{\b}[\mathcal{G}[\eta]](x,z).
				\end{align}
				In the evolution equation \eqref{inievo}, we approximate $(\nabla_{x,z}v^+_\eta)(x,0)$ by
\[
\left.(\nabla_{x,z}\tilde v^+_{\b}[\eta])(x,0)\right|_{\b=\nabla\eta(x)}.
\]
This is the central step in the freezing-coefficient argument.\\
 We denote \begin{align}
					\label{ndefbfb}
					\A(\b,\alpha)=\frac{\mathbf{B}(\b)}{\langle\hat \alpha \cdot \mathbf{b}\rangle^{d+1}},\quad\quad\quad	\mathbf{B}(\mathbf{b})=\frac{\tilde c_d }{\mu_++\mu_-}{\langle \mathbf{b}\rangle^2}B(\mathbf{b}).
				\end{align}	
                Define 
                \begin{equation}\label{defHim1}
                \H^{im}[\eta](x)=-\frac{1}{\mu_+}e_{d+1}\cdot\left. \left(\M(\nabla\eta(x))\left((\nabla_{x,z}v^+_\eta)(x,z)-(\nabla_{x,z}\tilde v^+_{\b}[\eta])(x,z)\big|_{\b=\nabla\eta(x)}\right)\right)\right|_{z=0}.
                \end{equation}
                The following lemma gives the desired decomposition of the interface equation into an explicit parabolic principal part and a remainder coming from the transmission error.
\begin{lemma}[Frozen-coefficient contour formulation]\label{lemevoeta} The equation \eqref{evo} can be rewritten as the following form:
\begin{align}\label{evoeta'}
					&\partial_t \eta (x)-\frac{\tilde c_d\langle\nabla \eta(x)\rangle^2 }{\mu_++\mu_-}\int _{\mathbb{R}^d}\frac{B(\nabla\eta(x)):\delta_\alpha \nabla^2 \eta(x)}{\langle\hat \alpha \cdot \nabla\eta(x)\rangle^{d+1}}\frac{d\alpha}{|\alpha|^{d+1}}\nonumber\\
						&\quad\quad\quad\quad\quad\quad\quad\quad=\frac{\tilde c_d\langle\nabla \eta(x)\rangle^2 }{\mu_++\mu_-}\int_{\mathbb{R}^d}\frac{\delta_\alpha B(\nabla \eta)(x): \nabla^2\eta(x-\alpha)}{\langle\hat \alpha \cdot \nabla\eta(x)\rangle^{d+1}}\frac{d\alpha}{|\alpha|^{d+1}}+ \H^{im}[\eta](x),
				\end{align} 
                where $\tilde c_d=\pi^{-\frac{d+1}{2}}\Gamma\left(\frac{d+1}{2}\right)>0$.
\end{lemma}
\begin{proof}
For any fixed   $\b\in\mathbb{R}^d$,
				the evolution equation \eqref{inievo} can be written as 
				\begin{equation}\label{eqeta1}
					\begin{aligned}
						\partial_t \eta (x)=&-\frac{1}{\mu_+}e_{d+1}\cdot (\M(\nabla\eta(x))(\nabla_{x,z}\tilde v^+_{\b}[\eta])(x,z))|_{z=0}\\
						&\quad\quad\quad-\frac{1}{\mu_+}e_{d+1}\cdot\left. \left(\M(\nabla\eta(x))\big(\nabla_{x,z}v^+_\eta(x,z)-\nabla_{x,z}\tilde v^+_{\b}[\eta](x,z)\big)\right)\right|_{z=0}.
					\end{aligned}
				\end{equation}
               Note that here the second term contributes only as a lower-order remainder term.\\
				By \eqref{dx0}, \eqref{z0}, the first term can be written as
				\begin{equation}\label{spli4}
					\begin{aligned}
						&-\frac{1}{\mu_+}e_{d+1}\cdot (\M(\nabla\eta(x))(\nabla_{x,z}\tilde v^+_{\b}[\eta])(x,z))|_{z=0}\\
						&=\frac{1}{\mu_+}\left.\left(\nabla \eta(x)\cdot \nabla_x \tilde v^+_{\b}[\eta](x,z)-\langle \nabla \eta(x)\rangle^2\partial_z\tilde v^+_{\b}[\eta](x,z)\right)\right|_{z=0}\\
						&=\frac{\langle\nabla \eta(x)\rangle^2 }{\mu_++\mu_-}\left(-\tilde c_d\int_{\mathbb{R}^d}\frac{\delta_\alpha \mathcal{G}[\eta](x)}{\langle\hat \alpha \cdot \mathbf{b}\rangle^{d+1}}\frac{d\alpha}{|\alpha|^{d+1}}+\left(\frac{\nabla \eta(x)}{\langle\nabla \eta(x)\rangle^2 }-\frac{\mathbf{b}}{\langle\mathbf{b}\rangle^2}\right)\cdot \nabla\mathcal{G}[\eta](x)\right)\\
						&=\int_{\mathbb{R}^d}\A(\b,\alpha):\delta_\alpha \nabla^2\eta(x)\frac{d\alpha}{|\alpha|^{d+1}}+\int_{\mathbb{R}^d}\frac{(\mathbf{B}(\nabla \eta(x))-\mathbf{B}(\mathbf{b})):\delta_\alpha \nabla^2\eta(x)}{\langle\hat \alpha \cdot \mathbf{b}\rangle^{d+1}}\frac{d\alpha}{|\alpha|^{d+1}}\\
						&\ \quad\quad+\frac{\tilde c_d\langle\nabla \eta(x)\rangle^2 }{\mu_++\mu_-}\int_{\mathbb{R}^d}\frac{\delta_\alpha B(\nabla \eta): \nabla^2\eta(x-\alpha)}{\langle\hat \alpha \cdot \mathbf{b}\rangle^{d+1}}\frac{d\alpha}{|\alpha|^{d+1}}+\frac{\langle\nabla \eta(x)\rangle^2 }{\mu_++\mu_-}\left(\frac{\nabla \eta(x)}{\langle\nabla \eta(x)\rangle^2 }-\frac{\mathbf{b}}{\langle \mathbf{b}\rangle^2}\right)\cdot \nabla\mathcal{G}[\eta](x),
					\end{aligned}
				\end{equation}
                with $\mathcal{G},\A,\mathbf{B}$ defined in \eqref{ndefG} and \eqref{ndefbfb}. We now isolate the principal part of the evolution equation. It is the first term on the right-hand side of \eqref{spli4}, evaluated at $\b=\nabla\eta(x)$. We therefore obtain 
                \begin{align*}
                -\frac{1}{\mu_+}e_{d+1}&\cdot (\M(\nabla\eta(x))(\nabla_{x,z}\tilde v^+_{\b}[\eta])(x,z))\big|_{z=0,\b=\nabla\eta(x)}\\
                &=\int_{\mathbb{R}^d}\A(\nabla\eta(x),\alpha):\delta_\alpha \nabla^2\eta(x)\frac{d\alpha}{|\alpha|^{d+1}}+\frac{\tilde c_d\langle\nabla \eta(x)\rangle^2 }{\mu_++\mu_-}\int_{\mathbb{R}^d}\frac{\delta_\alpha B(\nabla \eta): \nabla^2\eta(x-\alpha)}{\langle\hat \alpha \cdot \nabla\eta(x)\rangle^{d+1}}\frac{d\alpha}{|\alpha|^{d+1}}.
                \end{align*}
                Combining this with \eqref{eqeta1} yields \eqref{evoeta'} and completes the proof of the lemma. 
                   \end{proof}
                   \vspace{0.5cm}\\
                 We now use \eqref{evoeta'} to put the interface equation into the form of the
abstract parabolic equation studied in Section~\ref{secsch}. 
Assume $\phi$ is a smooth function that will be fixed later. Then \eqref{eqeta1} can be rewritten as
				\begin{align}\label{evoeta}
					\partial_t \eta (x)-\int _{\mathbb{R}^d}\A(\nabla\phi(x),\alpha):\delta_\alpha \nabla^2 \eta(x)\frac{d\alpha}{|\alpha|^{d+1}}=\H[\eta](x),
				\end{align} 
		where
\[
\H[\eta]=\H^{ex}[\eta]+\H^{im}[\eta].
\]
Here $\H^{im}$ is defined in \eqref{defHim1}, and the explicit error is
				\begin{equation}\label{defHgm}
					\begin{aligned}
						\H^{ex}[\eta](x)=	&\int _{\mathbb{R}^d}(\A(\nabla\eta(x),\alpha)-\A(\nabla\phi(x),\alpha)):\delta_\alpha \nabla^2 \eta(x)\frac{d\alpha}{|\alpha|^{d+1}}\\
						&	\quad\quad\quad\quad+\frac{\tilde c_d\langle\nabla \eta(x)\rangle^2 }{\mu_++\mu_-}\int_{\mathbb{R}^d}\frac{\delta_\alpha B(\nabla \eta): \nabla^2\eta(x-\alpha)}{\langle\hat \alpha \cdot \nabla\eta(x)\rangle^{d+1}}\frac{d\alpha}{|\alpha|^{d+1}}.
					\end{aligned}
				\end{equation}
                Thus the left-hand side of \eqref{evoeta} is a frozen-coefficient third-order
parabolic operator, while $\H^{ex}$ and $\H^{im}$ collect respectively the
explicit commutator error and the implicit transmission error. This is precisely
the form to which the Schauder estimates of Section~\ref{secsch} will be applied.
\vspace{0.2cm}\\
				Denote 
				\begin{align}\label{defwb}
					\omega_\mathbf{b}^\pm[\eta]=v^\pm_\eta-\tilde v^\pm_{\b}[\eta],\quad\quad \mathcal{Q}[\eta](x)=(\nabla_{x,z}\omega^+_\b[\eta])(x,z)\big|_{\b=\nabla\eta(x), z=0}.
				\end{align} Then 
				\begin{align}\label{defHim}
					\H^{im}[\eta](x)=-\frac{1}{\mu_+}e_{d+1}\cdot \left(\M(\nabla\eta(x))\mathcal{Q}[\eta](x)\right).
				\end{align} 
				Moreover,  from \eqref{elleqeta} and  \eqref{eqmain}, we know that  
				$\omega_\mathbf{b}^\pm=\omega_\mathbf{b}^\pm[\eta]$ solves the elliptic system
				\begin{equation}	\label{elliw}
					\begin{aligned}
						&\operatorname{div}_{x,z}(\M(\b)\nabla_{x,z} \omega^\pm_{\b})=\operatorname{div} F_{1,\b}^\pm,\ \ \text{in}\ \tilde\Omega_\eta^\pm,\\
						&\omega^-_{\b}-\omega^+_{\b}=F_2, \quad \text { on } \{z=0\},\\
						&e_{d+1}\cdot ((\mu_+^{-1}(\M(\b)\nabla_{x,z} \omega_\mathbf{b}^+-F_{1,\b}^+)-\mu_-^{-1}(\M(\b)\nabla_{x,z} \omega^-_{\b}-F_{1,\b}^-)))=0, \quad \text { on } \{z=0\},\\
						&\tilde \nu^\pm\cdot (\M(\nabla \eta)\nabla _{x,z} \omega^\pm_{\b}) =F_{3,\b}^\pm,\ \ \quad\quad\quad\quad\ \text{on}\  \tilde \Gamma_\eta^\pm,
					\end{aligned}
				\end{equation}	
				where $\tilde \Omega_\eta^\pm,\tilde \Gamma_\eta^\pm$ are defined in \eqref{defdomain}, and
				\begin{equation}\label{fortems}
					\begin{aligned}
						&F_{1,\b}^\pm=\left(\M(\mathbf{b})-\M(\nabla \eta)\right)\nabla_{x,z}  v^\pm_\eta,\\
						&F_2=\varrho_0\eta,\\
						&F_{3,\b}^\pm=\left.\tilde \nu^\pm\cdot (\M(\nabla \eta)(\nabla_{x,z}  v^\pm_\eta-\nabla_{x,z}\tilde v_{\mathbf{b}}^\pm[\eta])) \right|_{ \tilde \Gamma_\eta^\pm}.
					\end{aligned}
				\end{equation}
               
				\begin{remark}
		We note that the operator
\[
-\int_{\mathbb{R}^d}\A(\nabla\phi(x),\alpha):\delta_\alpha \nabla^2 \eta(x)\,\frac{d\alpha}{|\alpha|^{d+1}}
\]
can be written as a pseudo-differential operator of the form \eqref{defop}. Indeed,
					\begin{align*}
						-\int _{\mathbb{R}^d}\A(\nabla\phi(x),\alpha):\delta_\alpha \nabla^2 \eta(x)\frac{d\alpha}{|\alpha|^{d+1}}=(2\pi)^{-\frac{d}{2}}\int_{\mathbb{R}^d} \tilde \A(x,\xi)\hat \eta(\xi)e^{ix\cdot \xi}d\xi,
					\end{align*}
					where 
					\begin{align*}
						\tilde 	\A(x,\xi)=&\int_{\mathbb{R}^d}\A(\nabla\phi(x),\alpha):(\xi\otimes\xi)(1-e^{-i\alpha\cdot\xi})\frac{d\alpha}{|\alpha|^{d+1}}\\
						=&\frac{1}{2}\int_{\mathbb{R}^d}\A(\nabla\phi(x),\alpha):(\xi\otimes\xi)(2-e^{i\alpha\cdot\xi}-e^{-i\alpha\cdot\xi})\frac{d\alpha}{|\alpha|^{d+1}}.
					\end{align*}
                    In the second identity we used the symmetry
\[
\A(\mathbf b,\alpha)=\A(\mathbf b,-\alpha).
\]
By the ellipticity property \eqref{Bellip}, it follows that
					\begin{align*}
						&\frac{\tilde c_d}{(\mu_++\mu_-)\langle \nabla\phi(x)\rangle^{d+2}}\mathrm{Id}\leq \frac{\tilde \A(x,\xi)}{|\xi|^3}\leq 	\frac{\tilde c_d}{(\mu_++\mu_-)\langle \nabla\phi(x)\rangle}\mathrm{Id},\\ &|\nabla_x^{n_1}\nabla_{\xi}^{n_2}\tilde \A(x,\xi)|\lesssim (1+\|\nabla\phi\|_{C^{n_1}})^{n_1+10} |\xi|^{3-n_2},\ \forall n_1,n_2\in\mathbb{N}.
					\end{align*}
					Therefore, the operator on the left-hand side of \eqref{evoeta} satisfies condition \eqref{condop} with
\[
s=3,\qquad
c_0=\frac{\tilde c_d}{(\mu_++\mu_-)(1+\|\nabla\phi\|_{L^\infty})^{d+2}},
\qquad
c_1=\bigl(1+\|\nabla\phi\|_{C^{d+m+4}}\bigr)^{d+m+10}.
\]
					
				\end{remark}\vspace{0.3cm}

				\section{Estimates of nonlinear terms}\label{secest}

In this section, we estimate the nonlinear terms appearing on the right-hand side of the contour dynamics formulation \eqref{evoeta}. The goal is to show that the error terms $\H^{ex}$ and $\H^{im}$ satisfy the regularity assumptions required by the Schauder theory developed in Section~\ref{secsch}. In particular, we establish estimates that are compatible with the logarithmically critical topology and exhibit the smallness needed to close the nonlinear argument for sufficiently small times.

The analysis naturally splits into two parts. The term $\H^{ex}$ is explicit and arises from the freezing-coefficient decomposition of the principal operator. Its estimate relies on commutator-type bounds and the logarithmic regularity of the interface. The term $\H^{im}$ is more delicate, since it depends implicitly on the solution of the elliptic transmission problem. Its control requires the endpoint Schauder estimates for transmission systems established in the previous sections.

To measure the regularity of the interface and the nonlinear terms, we introduce the following norms.
				For $a\in(0,1), k\in\mathbb{N}$, define 
				\begin{align*}
					\|f\|_{\HH^{m+a}}:=\sup_{\alpha}\frac{\|\delta_\alpha \nabla^mf\|_{L^2}}{|\alpha|^a}.
				\end{align*}
                Note that $\|f\|_{\HH^{m+a}}$ coincide with the homogeneous Besov norm $\dot B^{m+a}_{2,\infty}$. 
				We fix regularity parameters 
				\begin{align}\label{consgm}
					\varkappa>1,\ \ \ m\in\mathbb{N},\ \ \ \kappa\in(2,3)\ \text{such that}\ 0<3-\kappa\ll 1.
				\end{align}
				For $h:[0,T]\times \mathbb{R}^d\to \mathbb{R}$ with $T\in (0,\frac{1}{2}],$ define
				\begin{equation}	\label{defnorgm}
					\begin{aligned}
						&\|h\|_T:=\sup_{t\in[0,T]}\left(\|\nabla h(t)\|_{ \dot C^{\log ^\varkappa}\cap L^2}+t^\frac{m+\kappa}{3}(|\log t|^\varkappa  \|\nabla h(t)\|_{\dot C^{m+\kappa}}+\|\nabla h(t)\|_{\HH^{m+\kappa}})\right),\\
						&\|h\|_{X_T}:=\|h\|_{L^\infty_TL^2}+\| h\|_T.
					\end{aligned}
				\end{equation}
                We also define the following non-endpoint norm.
                \begin{equation}
                	\begin{aligned}\label{defts}
						\|h\|_{T,*}:=&\sup_{t\in[0,T]}|\log t|^\varkappa(t^\frac{1}{15}  \|\nabla h(t)\|_{\dot C^{\frac{1}{5}}}+t^\frac{m+\kappa}{3} \|\nabla h(t)\|_{\dot C^{m+\kappa}})\\
						&+\sup_{t\in[0,T]}(t^\frac{1}{15}  \|\nabla h(t)\|_{\HH^{\frac{1}{5}}}+t^\frac{m+\kappa}{3} \|\nabla h(t)\|_{\HH^{m+\kappa}})+T^\frac{1}{15}\|h\|_T.
					\end{aligned}
                    \end{equation}
                    For a function $\phi$ on $\mathbb{R}^d$, we can define $\|\phi\|_T, \|\phi\|_{X_T}$ by considering it as a function defined on $[0,T]\times \mathbb{R}^d$ and is invariant with respect to the time variable.
                    Note that 
                    \begin{align}\label{stsm}
						\|f\|_{T,*}\lesssim \|f-\phi\|_{T}+\|\phi\|_{T,*}\lesssim \|f-\phi\|_{T}+T^\frac{1}{20}\|\nabla\phi\|_{C^{m+3}\cap H^{m+3}}.
					\end{align}
                    The norm $\|\cdot\|_{T,*}$ is a non-endpoint auxiliary norm used to quantify the smallness of the nonlinear terms. In particular, \eqref{stsm} shows that $\|f\|_{T,*}$ is small whenever $f$ remains close to a smooth reference profile $\phi$ in the $\|\cdot\|_T$ topology and the time interval is sufficiently short.
                    \vspace{0.2cm}\\
We first record several elementary composition estimates that will be used repeatedly in the nonlinear estimates below.
				\begin{lemma}\label{lemcom}
					Let $m\in\mathbb{N}$, $\alpha\in(0,1)$. Consider $g,g_1,g_2:\mathbb{R}^d\to\mathbb{R}$, and $h:\mathbb{R}\to \mathbb{R}$ satisfying 
					\begin{align*}
						\sum_{k=0}^{m+2}\|h^{(k)}\|_{L^\infty}\lesssim 1.
					\end{align*}
					Then 
					\begin{align*}
						\|\nabla^m(h\circ g)\|_{L^\infty}\lesssim  &\|g\|_{\dot C^1}^m+\| g\|_{\dot C^m},\\
						\|\nabla^m(h\circ g)\|_{\dot C^\alpha}\lesssim  &\| g\|_{\dot C^\alpha}^\frac{m+\alpha}{\alpha}+\|g\|_{\dot C^{m+\alpha}},\\
						\|\nabla^m(h\circ g_1-h\circ g_2)\|_{L^\infty}\lesssim&\sum_{n=0}^{m} \|g_1-g_2\|_{\dot C^n}(\|(g_1,g_2)\|_{\dot C^1}^{m-n}+\mathbf{1}_{m>n}\|(g_1,g_2)\|_{\dot C^{m-n}}),\\
						\|\nabla^m(h\circ g_1-h\circ g_2)\|_{\dot C^\alpha}\lesssim &
						\sum_{n=0}^{m}\left\{\|g_1-g_2\|_{\dot C^{n+\alpha}}(\|(g_1,g_2)\|_{\dot C^1}^{m-n}+\mathbf{1}_{m>n}\|(g_1,g_2)\|_{\dot C^{m-n}})\right.\\
						&\left.\quad\quad+\|g_1-g_2\|_{\dot C^n}(\|(g_1,g_2)\|_{\dot C^\alpha}^\frac{m-n+\alpha}{\alpha}+\|(g_1,g_2)\|_{\dot C^{m-n+\alpha}})\right\}.
					\end{align*}

				\end{lemma}
				\begin{proof}
We only indicate the standard argument. By the Faà di Bruno formula, for any multi-index
$\beta$ with $|\beta|=m$,
\[
\partial^\beta (h\circ g)
=
\sum_{1\leq q\leq m} h^{(q)}(g)
\sum_{\substack{\beta_1+\cdots+\beta_q=\beta\\ |\beta_i|\geq 1}}
C_{\beta_1,\ldots,\beta_q}
\prod_{i=1}^q \partial^{\beta_i}g .
\]
Since the derivatives of $h$ are bounded, it follows that
\[
\|\nabla^m(h\circ g)\|_{L^\infty}
\lesssim
\sum_{q=1}^m
\sum_{\substack{k_1+\cdots+k_q=m\\ k_i\geq 1}}
\prod_{i=1}^q \|g\|_{\dot C^{k_i}} .
\]
Using the elementary interpolation inequality
\[
\|g\|_{\dot C^k}
\lesssim
\|g\|_{\dot C^1}^{\frac{m-k}{m-1}}
\|g\|_{\dot C^m}^{\frac{k-1}{m-1}},
\qquad 1\leq k\leq m,
\]
and Young's inequality, each product is bounded by
\[
C\bigl(\|g\|_{\dot C^1}^{m}+\|g\|_{\dot C^m}\bigr),
\]
which gives the first estimate.
For the Hölder estimate, applying the $\dot C^\alpha$ seminorm to the same expansion and using
\[
[f_1\cdots f_q]_{\dot C^\alpha}
\lesssim
\sum_{j=1}^q [f_j]_{\dot C^\alpha}
\prod_{i\neq j}\|f_i\|_{L^\infty},
\]
one obtains
\[
\|\nabla^m(h\circ g)\|_{\dot C^\alpha}
\lesssim
\sum
\|g\|_{\dot C^{k_j+\alpha}}
\prod_{i\neq j}\|g\|_{\dot C^{k_i}},
\qquad
k_1+\cdots+k_q=m, k_i\in \mathbb{N}.
\]
By interpolation between $\dot C^\alpha$ and $\dot C^{m+\alpha}$, namely
\[
\|g\|_{\dot C^{k}}
\lesssim
\|g\|_{\dot C^\alpha}^{\frac{m+\alpha-k}{m}}
\|g\|_{\dot C^{m+\alpha}}^{\frac{k-\alpha}{m}},
\qquad \alpha\leq k\leq m+\alpha,
\]
and Young's inequality, the right-hand side is bounded by
\[
C\bigl(
\|g\|_{\dot C^\alpha}^{\frac{m+\alpha}{\alpha}}
+
\|g\|_{\dot C^{m+\alpha}}
\bigr).
\]
This proves the second estimate.\\
We write
\[
h(g_1)-h(g_2)
=
(g_1-g_2) H,
\qquad
H(x):=\int_0^1 h'\bigl(g_2(x)+\theta(g_1(x)-g_2(x))\bigr)\,d\theta .
\]
Then, by Leibniz' rule, for every multi-index $\beta$ with $|\beta|=m$,
\[
\partial^\beta \bigl(h(g_1)-h(g_2)\bigr)
=
\sum_{\beta_1+\beta_2=\beta}
C_{\beta_1,\beta_2}
\partial^{\beta_1}(g_1-g_2)\,
\partial^{\beta_2}H .
\]
Equivalently, in schematic notation,
\[
\nabla^m\bigl(h(g_1)-h(g_2)\bigr)
=
\sum_{n=0}^m
\nabla^n(g_1-g_2)\,\nabla^{m-n}H ,
\]
where the products stand for finite sums of contractions of derivatives of the indicated orders.
The first two estimates already proved, applied to $H$, imply
\[
\|\nabla^{m-n}H\|_{L^\infty}
\lesssim
\|(g_1,g_2)\|_{\dot C^1}^{m-n}
+\mathbf{1}_{m>n}
\|(g_1,g_2)\|_{\dot C^{m-n}},
\]
and
\[
\|\nabla^{m-n}H\|_{\dot C^\alpha}
\lesssim
\|(g_1,g_2)\|_{\dot C^\alpha}^{\frac{m-n+\alpha}{\alpha}}
+
\|(g_1,g_2)\|_{\dot C^{m-n+\alpha}} .
\]
Here, in the case $m=n$, the first estimate is understood by using the trivial bound
\[
\|H\|_{L^\infty}
\lesssim \|h'\|_{L^\infty}
\lesssim 1
=
\|(g_1,g_2)\|_{\dot C^1}^{0}.
\]
Thus the term involving $\|(g_1,g_2)\|_{\dot C^{m-n}}$ is needed only when $m>n$.\\
Combining these bounds with the product estimate in $L^\infty$ gives
\[
\|\nabla^m(h\circ g_1-h\circ g_2)\|_{L^\infty}
\lesssim
\sum_{n=0}^m
\|g_1-g_2\|_{\dot C^n}
\Bigl(
\|(g_1,g_2)\|_{\dot C^1}^{m-n}
+\mathbf{1}_{m>n}
\|(g_1,g_2)\|_{\dot C^{m-n}}
\Bigr),
\]
while the Hölder product estimate gives
\[
\begin{aligned}
\|\nabla^m(h\circ g_1-h\circ g_2)\|_{\dot C^\alpha}
\lesssim
\sum_{n=0}^{m}
&\Bigl[
\|g_1-g_2\|_{\dot C^{n+\alpha}}
\Bigl(
\|(g_1,g_2)\|_{\dot C^1}^{m-n}
+\mathbf{1}_{m>n}
\|(g_1,g_2)\|_{\dot C^{m-n}}
\Bigr)
\\
&\quad+
\|g_1-g_2\|_{\dot C^n}
\Bigl(
\|(g_1,g_2)\|_{\dot C^\alpha}^{\frac{m-n+\alpha}{\alpha}}
+
\|(g_1,g_2)\|_{\dot C^{m-n+\alpha}}
\Bigr)
\Bigr].
\end{aligned}
\]
This completes the proof of the lemma.
\end{proof}\vspace{0.4cm}\\
We shall also use the following elementary estimates for second-order finite differences.
				\begin{lemma}
					\label{lemoalin}
					Denote $\mathcal{O}_\alpha g(x)=\frac{\delta_\alpha\delta_{-\alpha}g(x)}{|\alpha|}$. Then  for any $\nu\in(0,1)$, $\varkappa>1$,
					\begin{align}
						&\int_{\mathbb{R}^d}\|\mathcal{O}_\alpha g\|_{L^\infty}\frac{d\alpha}{|\alpha|^{d}}\lesssim \min\left\{\|g\|_{\dot C^{1+\nu}}^\frac{1}{2} \|g\|_{\dot C^{1-\nu}}^\frac{1}{2}, \|g\|_{ C^{1,\log^\varkappa}}\right\},\label{S2eq2}\\
                        &\int_{\mathbb{R}^d}\|\mathcal{O}_\alpha g\|_{L^2}\frac{d\alpha}{|\alpha|^{d}}\lesssim \|g\|_{\dot{B}^{1}_{2,1}}.\nonumber
					\end{align}
					Moreover, for any $\beta\in\mathbb{R}^d$, and $a\in(0,1)$,
					\begin{align}\label{S2eq3}
						\int_{\mathbb{R}^d}\|\delta_\beta\mathcal{O}_\alpha g\|_{L^\infty}\frac{d\alpha}{|\alpha|^{d}}\lesssim|\beta|^{a} \|g\|_{\dot C^{1+a}},\quad \int_{\mathbb{R}^d}\|\delta_\beta\mathcal{O}_\alpha g\|_{L^2}\frac{d\alpha}{|\alpha|^{d}}\lesssim|\beta|^{a} \|g\|_{\dot {\mathbf{H}}^{1+a}}.
					\end{align}
				\end{lemma}
                \begin{proof}
                    The first inequality in \eqref{S2eq2} follows from the fact that
                    \begin{align*}
                        |\mathcal{O}_\alpha g(x)|&=\left|\int_0^1 \frac{\alpha}{|\alpha|}\cdot\left(\nabla g(x+\lambda\alpha)-\nabla g(x-\lambda\alpha)\right)d\lambda\right|\\
                        &\lesssim \min\left\{|\alpha|^{\nu}\|g\|_{\dot C^{1+\nu}},\min\{\log^{-\varkappa}(2+|\alpha|^{-1}),|\alpha|^{-1}\}\|g\|_{C^{1,\log^\varkappa}}\right\},\\
                        |\mathcal{O}_\alpha g(x)|&\lesssim |\alpha|^{-1}\left(\left|\delta_\alpha g(x)\right|+\left|\delta_{-\alpha }g(x)\right|\right)\lesssim |\alpha|^{-\nu}\|g\|_{\dot C^{1-\nu}}.
                    \end{align*}
                    The second inequality in \eqref{S2eq2} comes from Lemma 2.37 of \cite{Fourierbook}.
                    
                    Then we consider \eqref{S2eq3}. Note that
                    \begin{align*}
                        &\int_{\mathbb{R}^d}\|\delta_\beta\mathcal{O}_\alpha g\|_{L^\infty}\frac{d\alpha}{|\alpha|^{d}}\lesssim \int_{|\alpha|\leq |\beta|}\|\delta_\beta\mathcal{O}_\alpha g\|_{L^\infty}\frac{d\alpha}{|\alpha|^{d}}+\int_{|\alpha|\geq |\beta|}\|\delta_\beta\mathcal{O}_\alpha g\|_{L^\infty}\frac{d\alpha}{|\alpha|^{d}}\\
                        &\quad\lesssim \int_{|\alpha|\leq |\beta|}\|\mathcal{O}_\alpha g\|_{L^\infty}\frac{d\alpha}{|\alpha|^{d}}+|\beta|\int_{|\alpha|\geq |\beta|}\frac{\|\delta_\beta\delta_\alpha g\|_{L^\infty}}{|\beta||\alpha|^{a}}\frac{d\alpha}{|\alpha|^{d+1-a}}\\
                        &\quad\lesssim \|g\|_{\dot C^{1+a}}\int_{|\alpha|\leq |\beta|}|\alpha|^{-d+a}d\alpha +|\beta| \|g\|_{\dot C^{1+a}}\int_{|\alpha|\geq |\beta|}|\alpha|^{-d+a-1}d\alpha \\
                        &\quad\lesssim |\beta|^a\|g\|_{\dot C^{1+a}}.
                    \end{align*}
                    The second inequality in \eqref{S2eq3} can be proved similarly by changing the $L^\infty$ norm to $L^2$ in the estimate above. Precisely,
                    \begin{align*}
                        &\int_{\mathbb{R}^d}\|\delta_\beta\mathcal{O}_\alpha g\|_{L^2}\frac{d\alpha}{|\alpha|^{d}}\lesssim \int_{|\alpha|\leq |\beta|}\|\delta_\beta\mathcal{O}_\alpha g\|_{L^2}\frac{d\alpha}{|\alpha|^{d}}+\int_{|\alpha|\geq |\beta|}\|\delta_\beta\mathcal{O}_\alpha g\|_{L^2}\frac{d\alpha}{|\alpha|^{d}}\\
                        &\quad\lesssim \int_{|\alpha|\leq |\beta|}\left\|\left(\int_0^1\frac{\alpha}{|\alpha|}\cdot\delta_\alpha\nabla g(x-\lambda\alpha)\right)d\lambda\right\|_{L^2}\frac{d\alpha}{|\alpha|^{d}}+|\beta|\int_{|\alpha|\geq |\beta|}\frac{\|\delta_\beta\delta_\alpha g\|_{L^2}}{|\beta||\alpha|^{a}}\frac{d\alpha}{|\alpha|^{d+1-a}}\\
                        &\quad\lesssim\|g\|_{\dot {\mathbf{H}}^{1+a}} \int_{|\alpha|\leq |\beta|}|\alpha|^{-d+a}d\alpha +|\beta| \|g\|_{\dot {\mathbf{H}}^{1+a}}\int_{|\alpha|\geq |\beta|}|\alpha|^{-d+a-1}d\alpha \\
                        &\quad\lesssim |\beta|^a\|g\|_{\dot {\mathbf{H}}^{1+a}}.
                    \end{align*}
                    This completes the proof of the lemma.
                \end{proof}
				\subsection{Estimates of $\H^{ex}[f]$}	

                We first estimate the explicit critical term $\H^{ex}[f]$, which arises from the freezing-coefficient decomposition of the principal surface-tension operator. Although explicit, this term has the same critical scaling as the leading operator and must be estimated sharply in the logarithmically corrected topology.\vspace{0.2cm}\\
We prove the following estimate.
				\begin{lemma}\label{lemH1}
                Let $\H^{ex}[f]$ be defined by \eqref{defHgm}, and let the norm
$\|\cdot\|_T$ be defined by \eqref{defnorgm}, with
$\varkappa,m,\kappa$ fixed as in \eqref{consgm}. Then, for every
$T\in(0,\frac12)$,
					\begin{align}
						\sup_{t\in[0,T]}&|\log t|^\varkappa(t^\frac{\kappa}{3}\|\H^{ex}[f](t)\|_{\dot C^{\kappa-2}\cap \HH^{\kappa-2}}+t^\frac{m+\kappa}{3}\|\nabla^m\H^{ex}[f](t)\|_{\dot C^{\kappa-2}\cap \HH^{\kappa-2}})\nonumber\\
						&\quad\quad\lesssim (\|f-\phi\|_T+T^\frac{1}{20}\mathfrak{M}_\phi)^2(1+\|(f,\phi)\|_T)^{2m+5}, \label{m}
					\end{align}
                    where $\mathfrak{M}_\phi=\|\phi\|_{H^{m+4}\cap C^{m+4}}$.
				\end{lemma}
\begin{proof}
				For brevity, we omit the harmless constant
$\frac{\tilde c_d}{\mu_++\mu_-}$ in the definition of $\H^{ex}$. We decompose
					\[
\H^{ex}[f]=\mathcal P[f]+\mathcal W[f],
\]
where
\[
\begin{aligned}
\mathcal P[f](x)
&=
\int_{\mathbb R^d}
(\A(\nabla f(x),\alpha)-\A(\nabla\phi(x),\alpha))
:\delta_\alpha\nabla^2 f(x)
\frac{d\alpha}{|\alpha|^{d+1}},\\
\mathcal W[f](x)
&=
\langle\nabla f(x)\rangle^2
\int_{\mathbb R^d}
\frac{\delta_\alpha B(\nabla f)(x):\nabla^2f(x-\alpha)}
{\langle\hat\alpha\cdot\nabla f(x)\rangle^{d+1}}
\frac{d\alpha}{|\alpha|^{d+1}} .
\end{aligned}
\]
The term $\mathcal P[f]$ measures the error produced by freezing the coefficient
at $\nabla\phi$, while $\mathcal W[f]$ is the explicit critical nonlinear term
coming from the variation of the curvature coefficient.

Set
\[
a=\kappa-2\in(0,1).
\]
It is enough to prove that, for every $0\le n\le m$,
\begin{equation}\label{difne}
\begin{aligned}
\sup_{t\in[0,T]}
t^{\frac{n+\kappa}{3}}|\log t|^\varkappa
\|\nabla^n\H^{ex}[f](t)\|_{\dot C^a\cap\HH^a}
\lesssim
(\|f-\phi\|_T+\|f\|_{T,*})\|f\|_{T,*}
(1+\|(f,\phi)\|_T)^{2m+5}.
\end{aligned}
\end{equation}
Indeed, the desired estimate \eqref{m} follows from \eqref{difne} together
with \eqref{stsm}.\vspace{0.1cm}\\
	Now we start to prove \eqref{difne}.\\  
					\textbf{Step 1:} Estimate $\mathcal{P}[f]$. By the symmetry
$\A(\cdot,\alpha)=\A(\cdot,-\alpha)$, the first-order difference cancels and
we may rewrite $\mathcal P[f]$ in terms of the second-order difference operator
$\mathcal O_\alpha$:
					\begin{align*}
						\mathcal{P}[f](x)=&\int _{\mathbb{R}^d}(\A(\nabla f(x),\alpha)-\A(\nabla\phi(x),\alpha)):\delta_\alpha \nabla^2 f(x)\frac{d\alpha}{|\alpha|^{d+1}}\\
						=&-\frac{1}{2}\int _{\mathbb{R}^d}(\A(\nabla f(x),\alpha)-\A(\nabla\phi(x),\alpha)):\mathcal{O}_\alpha \nabla^2 f(x)\frac{d\alpha}{|\alpha|^{d}},
					\end{align*}
					where  $\mathcal{O}_\alpha g(x)$ is defined in Lemma \ref{lemoalin}.\\
Applying $\delta_\beta\nabla^n$, with $\beta\in\mathbb R^d$ and $0\le n\le m$, we split the resulting terms according to whether the finite difference $\delta_\beta$ falls on the second-order difference of $f$ or on the coefficient. This gives
					\begin{align*}
						|\delta_\beta 	\nabla^n\mathcal{P}[f](x)|\lesssim& \sum_{n_1=0}^n\int _{\mathbb{R}^d}\left|\nabla ^{n_1}(\A(\nabla f(x),\alpha)-\A(\nabla\phi(x),\alpha)):\delta_\beta\mathcal{O}_\alpha  \nabla^{2+n-n_1} f(x)\right|\frac{d\alpha}{|\alpha|^{d}}\\ &+\sum_{n_1=0}^n\int _{\mathbb{R}^d}|\delta_\beta\nabla ^{n_1}(\A(\nabla f(\cdot),\alpha)-\A(\nabla\phi(\cdot),\alpha))(x)\mathcal{O}_\alpha \nabla^{2+n-n_1}f(x-\beta)|\frac{d\alpha}{|\alpha|^{d}}\\
						:=&\mathrm{I}_{1}+\mathrm{I}_2.
					\end{align*}
					We first consider $\mathrm{I}_{1}$.
					We have 
					\begin{align*}
						&\|\mathrm{I}_{1}\|_{L^\infty\cap L^2}\lesssim \sum_{n_1=0}^n	\|\nabla ^{n_1}(\A(\nabla f(x),\alpha)-\A(\nabla\phi(x),\alpha))\|_{L^\infty_{x,\alpha}\cap L^\infty_\alpha L^2_{x}}\int _{\mathbb{R}^d}\left\|\delta_\beta\mathcal{O}_\alpha  \nabla^{2+n-n_1}f(x)\right\|_{L^\infty}\frac{d\alpha}{|\alpha|^{d}}.
					\end{align*}
					It follows from  Lemma \ref{lemcom} that
					\begin{equation}\label{esdA}
						\begin{aligned}
							&\|\nabla ^{k}(\A(\nabla f(x),\alpha)-\A(\nabla\phi(x),\alpha))\|_{L^\infty_{x,\alpha}\cap L^\infty_{\alpha}L^2_x}\\
							&\quad\quad\quad\quad\quad\quad\quad\quad\quad\lesssim \sum_{l=0}^{k}\|\nabla (f-\phi)\|_{\dot C^{l}\cap \dot \HH^l}(\|(\nabla f,\nabla\phi)\|_{\dot C^1}^{k-l}+\|(\nabla f,\nabla\phi)\|_{\dot C^{k-l}})\\
							&\quad\quad\quad\quad\quad\quad\quad\quad\quad\lesssim t^{-\frac{k}{3}} \|f-\phi\|_T(1+\|(f,\phi)\|_T)^k,\ \ \ \ \forall\ 0\leq k\leq m+2.
						\end{aligned}
					\end{equation}
					Moreover, from Lemma \ref{lemoalin} we obtain
					\begin{align*}
						\int_{\mathbb{R}^d}	\|\delta_\beta\mathcal{O}_\alpha  \nabla^{2+n-n_1} f\|_{L^\infty}\frac{d\alpha}{|\alpha|^d}\lesssim  |\beta|^a|\log t|^{-\varkappa} t^{-\frac{2+n-n_1+a}{3}}\|f\|_{T,*}.
					\end{align*}
					Hence, we deduce that  
					\begin{equation}\label{i1b}
						\begin{aligned}
							\|\mathrm{I}_{1}\|_{L^\infty\cap L^2}	\lesssim &|\beta|^a|\log t|^{-\varkappa}t^{-\frac{2+n+a}{3}}\|f\|_{T,*} \|f-\phi\|_T(1+\|(f,\phi)\|_{T})^{n}.
						\end{aligned}
					\end{equation}
					For $\mathrm{I_2}$, from \eqref{esdA} we deduce that  
					\begin{align*}
						\|\delta_\beta\nabla ^{n_1}(\A(\nabla f(\cdot),\alpha)-\A(\nabla\phi(\cdot),\alpha))(x)\|_{L^\infty_{x,\alpha}}&\lesssim |\beta|^a|\log t|^{-\varkappa} t^{-\frac{n_1+a}{3}} \|f-\phi\|_T(1+\|(f,\phi)\|_T)^{n_1+1}.\end{align*}
					Applying Lemma \ref{lemoalin} again, we obtain
					\begin{align*}
						\int_{\mathbb{R}^d}	\|\mathcal{O}_\alpha \nabla^{2+n-n_1} f\|_{L^\infty\cap L^2}\frac{d\alpha}{|\alpha|^d}\lesssim  t^{-\frac{2+n-n_1}{3}}\| f\|_{T,*}.
					\end{align*}
					Hence 
					\begin{align*}
						\|	\mathrm{I}_2\|_{L^\infty\cap L^2}&\lesssim |\beta|^a|\log t|^{-\varkappa}t^{-\frac{2+n+a}{3}}\|f\|_{T,*} \|f-\phi\|_T(1+\|(f,\phi)\|_{T})^{n+1}.
					\end{align*}
					Combining this with \eqref{i1b}, we have
					\begin{align}\label{pbhol}
						\|\delta_\beta 	\nabla^n\mathcal{P}[f]\|_{L^\infty\cap L^2}\lesssim |\beta|^a|\log t|^{-\varkappa}t^{-\frac{2+n+a}{3}}\|f\|_{T,*}(1+\|(f,\phi)\|_{T})^{n+1}.
					\end{align}
                    We now turn to $\mathcal W[f]$. This term has critical order, but it contains
an additional finite difference of $B(\nabla f)$, which provides the gain needed
to close the estimate.\vspace{0.15cm}\\
				\textbf{Step 2:} Estimate $\mathcal{W}[f]$.\\
					Note that 
					\begin{align*}
						\mathcal{W}[f](x)&=	{\langle \nabla f(x)\rangle ^2}\int_{\mathbb{R}^d}\frac{\delta_\alpha B(\nabla f(x)):\nabla^2 f(x-\alpha)}{\langle\hat \alpha \cdot  \nabla f(x)\rangle ^{d+1}} \frac{d\alpha}{|\alpha|^{d+1}}\\
						&=	{\langle \nabla f(x)\rangle ^2}\int_{\mathbb{R}^d}\frac{\delta_\alpha B(\nabla f(x)):\nabla^2 f(x)}{\langle\hat \alpha \cdot  \nabla f(x)\rangle ^{d+1}} \frac{d\alpha}{|\alpha|^{d+1}}-	{\langle \nabla f(x)\rangle ^2}\int_{\mathbb{R}^d}\frac{\delta_\alpha B(\nabla f(x)):\delta_\alpha\nabla^2 f(x)}{\langle\hat \alpha \cdot  \nabla f(x)\rangle ^{d+1}} \frac{d\alpha}{|\alpha|^{d+1}}\\
						&:=\mathcal{W}_{1}[f](x)+\mathcal{W}_{2}[f](x).
					\end{align*}
					We claim that for any $k\in\mathbb{N}$, $k\leq m+1$,
					\begin{equation}\label{esA}
						\begin{aligned}
							&\sup_{t\in[0,T]} |\log t|^\varkappa t^{\frac{k+2}{3}}\|\nabla^k\mathcal{W}[f]\|_{L^\infty\cap L^2}\lesssim \|f\|_{T,*}^2(1+\|f\|_T)^{2k+5}.
						\end{aligned}
					\end{equation}
					By the symmetry $\A(\cdot,\alpha)=\A(\cdot,-\alpha)$, we obtain  
					\begin{align*}
						\mathcal{W}_{1}[f](x)&=\frac{1}{2}	{\langle \nabla f(x)\rangle ^2}\int_{\mathbb{R}^d}\frac{\mathcal{O}_{\alpha} B(\nabla f(x)):\nabla^2 f(x)}{\langle\hat \alpha \cdot  \nabla f(x)\rangle ^{d+1}} \frac{d\alpha}{|\alpha|^{d}}.
					\end{align*}
					For any $k\in\mathbb{N}$, $k\leq m+1$,	by Lemma \ref{lemoalin}, we have 
					\begin{align*}
						\|\nabla^k\mathcal{W}_{1}[f]\|_{L^\infty\cap L^2}&\lesssim \sum_{l=0}^k\sum_{j=0}^{k-l}\left( \int_{\mathbb{R}^d}\|\mathcal{O}_\alpha \nabla^l B(\nabla f)\|_{L^\infty}\frac{d\alpha}{|\alpha|^{d}} \sup_{\alpha'}\left\|\nabla^{k-l-j}\left(\frac{\langle \nabla f\rangle ^2}{\langle\hat \alpha' \cdot  \nabla f(x)\rangle ^{d+1}}\right)\right\|_{L^\infty}\|\nabla^{2+j} f\|_{L^\infty\cap L^2}\right)\\
						&\lesssim \sum_{l=0}^kt^{-\frac{k-l+2}{3}}\| B(\nabla f)\|_{\dot C^{l+\frac{1}{2}}} ^\frac{1}{2} \| B(\nabla f)\|_{\dot C^{l+\frac{3}{2}}} ^\frac{1}{2} \|f\|_{T,*}(1+\|f\|_T)^{k-l+2}.
					\end{align*}
					From  Lemma \ref{lemcom}, we obtain
					\begin{align*}
						\| B(\nabla f)\|_{\dot C^{j+\frac{1}{2}}} \lesssim \|\nabla f\|_{\dot C^\frac{1}{2}}^{2j+1}+\|\nabla f\|_{\dot C^{j+\frac{1}{2}}}\lesssim |\log t|^{-\varkappa}t^{-\frac{j}{3}-\frac{1}{6}}  \|f\|_{T,*}(1+\|f\|_T)^{2j+1}.
					\end{align*}
					This implies 
					\begin{align*}
						\|\nabla^k\mathcal{W}_{1}[f](t)\|_{L^\infty\cap L^2}
						&\lesssim |\log t|^{-\varkappa} t^{-\frac{k+2}{3}}\|f\|_{T,*}^2(1+\|f\|_T)^{2k+5}.
					\end{align*}
					Similarly, for $\mathcal{W}_2[f]$, we have 
					\begin{align*}
						\|\nabla ^k \mathcal{W}_2[f]\|_{L^\infty\cap L^2}&\lesssim \sum_{l=0}^k \sum_{j=0}^l\left(\int_{\mathbb{R}^d}\|\delta_\alpha \nabla^{j}B(\nabla f)\|_{L^\infty}\|\delta_\alpha\nabla^{2+l-j}f\|_{L^\infty\cap L^2}\frac{d\alpha}{|\alpha|^{d+1}}\sup_\alpha \left\|\nabla^{k-l}\left(\frac{\langle \nabla f\rangle^2}{\langle\hat \alpha \cdot \nabla f\rangle^{d+1}}\right)\right\|_{L^\infty}\right)\\
						&\lesssim|\log t|^{-\varkappa} t^{-\frac{k+2}{3}}\|f\|_{T,*}^2(1+\|f\|_T)^{2k+5}.
					\end{align*}
					Thus, we obtain \eqref{esA}.
					Then by interpolation, we obtain 
					\begin{equation*}
						\begin{aligned}
							& \sup_{t\in[0,T]}|\log t|^\varkappa t^{\frac{n+\kappa}{3}}   \|\delta_\beta\nabla^n	\mathcal{W}[f](t)\|_{L^\infty\cap L^2}\lesssim |\beta|^a \|f\|_{T,*}^2(1+\|f\|_T)^{2n+5}, \ \ \ \forall n\in\mathbb{N}, n\leq m. 
						\end{aligned}
					\end{equation*}
					Combining this with \eqref{pbhol} yields \eqref{difne}.
					This completes the proof of the lemma.
				\end{proof}
				
				\subsection{Estimates of $\H^{im}[f]$}
We now estimate the implicit term $\H^{im}[f]$ defined in \eqref{defHim}. This is the most delicate part of the nonlinear analysis, since $\H^{im}[f]$ depends on the correction $\omega^\pm_{\b}[f]$, which solves the elliptic transmission problem \eqref{elliw}. Unlike $\H^{ex}[f]$, this term is not given by an explicit singular integral of the interface alone.

Throughout this subsection, we assume that
\begin{align}\label{condfff}
\|f\|_{X_T}<\infty,
\qquad
\inf_{t\in[0,T]}\operatorname{dist}(f,\Gamma^\pm)\geq \mathbf r .
\end{align}
Now we start to estimate the H\"{o}lder norm of solution to system \eqref{elliw}. 
				Recalling \eqref{defwb}, we have 
				\begin{align}\label{S4eq4}
					v_f^\pm=\omega_\b^\pm[f]+\tilde v_\b^\pm[f].
				\end{align}
				The estimation of force terms in \eqref{elliw} requires estimation of $\tilde v_\b^\pm[f]$ which is defined in 
				\eqref{deftvb}:
				\begin{align}\label{fortvb}
					\tilde v_\b^\pm[f](x,z)=c^\pm\int_{\mathbb{R}^d}K_\b(x-y,z)\mathcal{G}[f](y)dy.
				\end{align}
				As a preparation, we first give some estimates of $\tilde v_\b^\pm[f]$.\vspace{0.2cm}\\
We begin with an explicit representation of the kernel $K_\b$ defined in \eqref{defKor}.
				\begin{lemma}\label{Z10} Let $\b\in \mathbb{R}^d$ and $	\lambda^\pm(\xi,\b)=\frac{i\b\cdot\xi\pm\sqrt{(1+|\b|^2)|\xi|^2-(\b\cdot \xi)^2}}{1+|\b|^2},$
					\begin{equation*}
						K_\b(x,z)=\begin{cases}
							\frac{1}{(2\pi)^\frac{d}{2}}\int_{\mathbb{R}^d}\exp\left(ix\cdot\xi+\lambda^-(\xi,\b)z\right)d\xi,& z\geq0,\\
							\frac{1}{(2\pi)^\frac{d}{2}}\int_{\mathbb{R}^d}\exp\left(ix\cdot\xi+\lambda^+(\xi,\b)z\right)d\xi,& z<0.
						\end{cases}	
					\end{equation*}
					Then
					\begin{equation}\label{defK}
						K_\b(x,z)=c_d\frac{|z|}{\left((x\cdot \b+z)^2+|x|^2\right)^{\frac{d+1}{2}}},\ \ \ \forall x\in\mathbb{R}^d,\ \ z\neq 0,
					\end{equation}
					where the constant $c_d=\frac{\Gamma(\frac{d+1}{2})}{\pi^\frac{d+1}{2}}$ ensures that $\int _{\mathbb{R}^d}K_\b(x,1)dx=1$. Moreover, when $z=0$, there holds $K_\b(x,0)=\delta(x)$.
				\end{lemma}\begin{proof}
The case $z=0$ follows from the definition of the Fourier kernel, and the case
$\b=0$ reduces to the classical Poisson kernel. We therefore assume
$z\neq0$ and $\b\neq0$. We treat the case $z<0$; the case $z>0$ is identical,
with $\lambda^-$ replaced by $\lambda^+$. We first  recall the explicit formula of Poisson kernel (see \cite[Chapter \rm{III}, section 2.1]{Steinbook})
					\begin{equation}\label{Z1}
						\frac{1}{(2\pi)^\frac{d}{2}}\int_{\mathbb{R}^d}\exp\left(ix\cdot\xi+|\xi|z\right)d\xi=c_d\frac{-z}{(|x|^2+|z|^2)^{\frac{d+1}{2}}},~z<0.
					\end{equation}
					1) Case $d=1$, we have, 
					\begin{align*}
						K_\b(x,z)&=\frac{1}{2\pi}\int_{\mathbb{R}}\exp\left(ix\cdot\xi+\frac{i\b\xi+|\xi|}{\langle \b\rangle^2}z\right)d\xi\\&= \frac{1}{2\pi}\int_{\mathbb{R}}\exp\left(i\left(x+\frac{\b z}{\langle \b\rangle^2}\right)\cdot\xi+|\xi|\frac{z}{\langle \b\rangle^2}\right)d\xi.
					\end{align*}
					Applying \eqref{Z1} with $d=1$ to get 
					\begin{align*}
						K_\b(x,z)&=c_1\frac{-\frac{z}{\langle \b\rangle^2}}{\left(x+\frac{\b z}{\langle \b\rangle^2}\right)^2+\left(\frac{z}{\langle \b\rangle^2}\right)^2}=c_1\frac{-z }{\left( x\cdot \b+z\right)^2+|x|^2}.
					\end{align*}
					2) Case $d\geq 2$. 
				Using the transformation \eqref{Z2}, we obtain 
					\begin{align*}
						&	K_\b(x,z)=\frac{1}{(2\pi)^d}\int_{\mathbb{R}^d}\exp\left(ix\cdot\xi+\frac{i\b\cdot\xi+\sqrt{(1+|\b|^2)|\xi|^2-(
								\b\cdot \xi)^2}}{1+|
							\b|^2}z\right)d\xi\\&=\frac{1}{(2\pi)^d}\int_{\mathbb{R}^d}\exp\left(i|x|\left((1-(\frac{x\cdot \b}{|x||
							\b|})^2)^{\frac{1}{2}}\xi_1+\frac{x\cdot \b}{|x||\b|}\xi_2\right)+\frac{i|\b|\xi_2+\sqrt{(1+|\b|^2)|\xi|^2-|\b|^2\xi_2^2}}{1+|\b|^2}z\right)d\xi
						\\&=\frac{1}{(2\pi)^d}\int_{\mathbb{R}^d}\exp\left(i\frac{((|x||\b|)^2-(x\cdot \b)^2)^{\frac{1}{2}}}{|\b|}\xi_1+i\left(\frac{x\cdot \b}{|\b|}+\frac{|\b|z}{\langle \b\rangle^2}\right)\xi_2+\frac{\sqrt{(1+|\b|^2)|\xi|^2-|\b|^2\xi_2^2}}{1+|\b|^2}z\right)d\xi.
					\end{align*}
					By simple changes of variables, we obtain 
					\begin{align*}
						K_\b(x,z)=\frac{1}{(2\pi)^d}\int_{\mathbb{R}^d}\exp\left(i\frac{((|x||\b|)^2-(x\cdot \b)^2)^{\frac{1}{2}}}{|\b|\langle \b\rangle}\xi_1+i\left(\frac{x\cdot \b}{|\b|}+\frac{|\b|z}{\langle \b\rangle^2}\right)\xi_2+|\xi|\frac{z}{\langle \b\rangle^2}\right)\frac{d\xi}{\langle \b\rangle^{d-1}}.
					\end{align*}
					Finally, applying  \eqref{Z1} yields
					\begin{align*}
						K_\b(x,z)	&=\frac{c_d}{\langle \b\rangle^{d-1}}\frac{-\frac{z}{\langle \b\rangle^2}}{\left(\frac{(|x||\b|)^2-(x\cdot \b)^2}{|\b|^2\langle \b\rangle^2}+\left(\frac{x\cdot \b}{|\b|}+\frac{|\b|z}{\langle \b\rangle^2}\right)^2+\left(\frac{z}{\langle \b\rangle^2}\right)^2\right)^{\frac{d+1}{2}}}\\&=
						\frac{c_d}{\langle \b\rangle^{d-1}}\frac{-\frac{z}{\langle \b\rangle^2}}{\left(\frac{(x\cdot \b)^2+2zx\cdot \b +z^2}{\langle \b\rangle^2}+\frac{z^2}{\langle \b\rangle^2}\right)^{\frac{d+1}{2}}}
						=
						c_d\frac{-z}{\left((x\cdot \b+z)^2+|x|^2\right)^{\frac{d+1}{2}}}.
					\end{align*}
					This completes the proof of the lemma.
				\end{proof} 
                \begin{remark}
Since
\[
\int_{\mathbb R^d}K_\b(x,z)\,dx=1,
\qquad z\neq0,
\]
the family $K_\b(\cdot,z)$ forms an approximation of the identity. In
particular,
\[
K_\b(\cdot,z)\to\delta_0
\]
in the sense of distributions as $z\to0$.
\end{remark}
				We next collect several estimates for the frozen Poisson kernel \(K_{\b}\). These bounds will be used to control the frozen velocity \(\tilde v^\pm_{\b}[f]\) and its normal derivative near the interface. The main point is that \(x\)-derivatives can be estimated directly from the explicit kernel, while \(z\)-derivatives require the Fourier representation and the operator \(\mathfrak L_{\pm,\b}\) defined in \eqref{lemfdz0}.
				\begin{lemma}\label{esK11}
					We have,\\
					(1) For any $n\in\mathbb{N}$, $\b\in\mathbb{R}^d$,
					\begin{align}\label{Knxb}
						|\nabla_x^{n}K_\b(x,z)|\lesssim \frac{|z|\langle \b\rangle^{n}}{\left((x\cdot \b+z)^2+|x|^2\right)^{\frac{d+1+n}{2}}},\quad\forall x\in\mathbb{R}^d, z\neq 0.
					\end{align}
					(2)   For any $n_1,n_2\in\mathbb{N}$, $\b\in\mathbb{R}^d$,
					\begin{align}\label{deriK}
						|\nabla_x^{n_1}\partial_z^{n_2}K_\b(x,z)|\lesssim {\langle \b\rangle^{1+n_1+n_2}}{\left(|x|^2+\frac{|z|^2}{\langle \b\rangle^2}\right)^{-\frac{d+n_1+n_2}{2}}},\quad \quad\forall x\in\mathbb{R}^d, z\neq 0.
					\end{align}
					(3) For any $z\neq 0$, $\b\in\mathbb{R}^d$,
					\begin{align}\label{remintk}
						\int_{\mathbb{R}^d}K_\b(x,z)dx=C_d,
					\end{align}
					where $C_d$ is a constant independent of $z$ and $\b$.\\
					(4)	For any $a\in(0,1)$, 	 $\b\in\mathbb{R}^d$, $z,\beta\in\mathbb{R}^+$,
					\begin{equation}\label{esl1}
						\begin{aligned}
							&\int_{\mathbb{R}^d}|K_\b(y,z)-K_\b(y,\beta)||y|^ady \lesssim \langle \b\rangle^2|z-\beta|^a,\ \ \ \forall a\in(0,1).
						\end{aligned}
					\end{equation}
					
					\end{lemma}
					\begin{proof} The point-wise estimates \eqref{Knxb} and \eqref{deriK} follow directly from the formula \eqref{defK}.\\
						It remains to prove \eqref{remintk} and \eqref{esl1}. Observing the scaling invariant property
						$$
						K_\b(x,z)=|z|^{-d}K_\b(|z|^{-1}x,|z|^{-1}z),
						$$
						we obtain that for any $\b\in\mathbb{R}^d$, $z\neq 0$,
						\begin{align*}
							\int_{\mathbb{R}^d}K_\b(x,z)dx=|z|^{-d}\int_{\mathbb{R}^d} K_\b(|z|^{-1}x,|z|^{-1}z)dx=\int_{\mathbb{R}^d}K_\b(x,\pm 1)dx,\ \ \ z\in\mathbb{R}^\pm.
						\end{align*}
						By a rotation transformation and changes of variables, we obtain
						\begin{align*}
							\int_{\mathbb{R}^d}K_\b(x,\pm 1)dx&=c_d \int_{\mathbb{R}^d}\frac{1}{\left((x\cdot \b\pm 1)^2+|x|^2\right)^{\frac{d+1}{2}}}dx=c_d\int_{\mathbb{R}^d}\frac{1}{\left((x_1|\b|\pm 1)^2+|x|^2\right)^{\frac{d+1}{2}}}dx\\
							&=c_d\langle\b\rangle^{-1}\int_{\mathbb{R}^d}\frac{1}{\left(|x|^2+\frac{1}{\langle\b\rangle^{2}}\right)^\frac{d+1}{2}}dx=c_d\int_{\mathbb{R}^d}\frac{1}{\langle|x|\rangle^{d+1}}dx=C.
						\end{align*}
						This completes the proof of \eqref{remintk}. Finally, for any $z,\beta>0$,  by \eqref{deriK} we have
						\begin{align*}
							&|K_\b(y,z)-K_\b(y,\beta)|\lesssim \frac{\langle \b\rangle^{2}}{|
								y|^d} \min\left\{1,\frac{|z-\beta|}{|y|}\right\}.
						\end{align*}
						This implies \eqref{esl1}. Then we complete the proof of the lemma.
					\end{proof} \vspace{0.2cm}
					
					In view of formula \eqref{fortvb}, when estimating $\nabla_x^n \tilde v_\b^\pm[f]$, one can freely interchange the $x$-derivatives between the kernel $K_\b$ and the function $\mathcal{G}[f]$. Thus, Lemma \ref{esK11} suffices to bound derivatives of $\tilde v_\b^\pm[f]$ with respect to $x$ to any order. However, in estimating $\nabla_x^n \partial_z \tilde v_\b^\pm[f]$, we encounter the term $\partial_z K_\b$, which becomes singular near the interface ${z=0}$. Indeed, from \eqref{deriK} we only obtain
					\begin{align*}
						\int_{\mathbb{R}^d}|\partial_z K_\b(x,z)|dx\lesssim \langle\b\rangle^3|z|^{-1}.
					\end{align*}
					To overcome this difficulty, we further analyze $\partial_zK_\b$ via its Fourier transform. Note that 
						\begin{align*}
							\mathcal{F}(\partial_z K_\b(\cdot,z))(\xi)=\lambda^\pm(\xi,\b) \exp(\lambda^\pm(\xi,\b)z),\ \ \ z\in\mathbb{R}^\mp.
						\end{align*}
						Hence we can write 
						\begin{align*}
							\partial_z(K_\b(\cdot,z)\ast g)(x)=(K_\b(\cdot,z)\ast( \mathfrak{L}_{\pm,\b} g)(\cdot))(x),
						\end{align*}
						where   \begin{align*}
                    \mathfrak{L}_{\pm,\b} f(x)=\mathcal{F}^{-1}\left(\lambda^\pm(\xi,\b)\hat f(\xi)\right)(x).
                    \end{align*}
                    We first prove the following estimates for 
                 $\mathfrak{L}_{\pm,\b} f$.
				\begin{lemma}\label{lemholL}
				For any $a\in (0,1)$, $m\in\mathbb{N}$,
						\begin{align}\label{LHol}
							\sup_{\b\in\mathbb{R}^d}	\|\mathfrak{L}_{\pm,\b} f\|_{\dot C^{m+a}_x}\lesssim \|f\|_{\dot C^{m+1+a}},
						\end{align}
						and \begin{align}\label{LLLinf}
							\sup_{\b\in\mathbb{R}^d}	\|\mathfrak{L}_{\pm,\b} f\|_{L^\infty}\lesssim \|f\|_{ C^{1,\log^\varkappa  }}.
						\end{align}
				\end{lemma}
                \begin{proof}
                    	Observe that 
						\begin{align*}
							\nabla_x^m\mathfrak{L}_{\pm,\b} f(x)= \mathfrak{L}_{\pm,\b}(\nabla_x^mf)(x).
						\end{align*}
						From Lemma \ref{lemfdz0}, to prove \eqref{LHol} and \eqref{LLLinf}, it suffices to prove that 	
						\begin{align}\label{gggg1}
							\sup_{\b\in\mathbb{R}^d}\left\|\int_{\mathbb{R}^d} \frac{\delta_\alpha f(x)}{\langle \hat\alpha\cdot \b\rangle^{d+1}}\frac{d\alpha}{|\alpha|^{d+1}}\right\|_{\dot C^a_x}\lesssim \|f\|_{\dot C^{1+a}},\quad\quad\quad 	\sup_{\b\in\mathbb{R}^d}\left\|\int_{\mathbb{R}^d} \frac{\delta_\alpha f(x)}{\langle \hat\alpha\cdot \b\rangle^{d+1}}\frac{d\alpha}{|\alpha|^{d+1}}\right\|_{L^\infty_x}\lesssim \|f\|_{C^{1,\log^\varkappa}}.
						\end{align}
						Indeed, by symmetry, we have 
						\begin{align*}
							\int_{\mathbb{R}^d} \frac{\delta_\alpha f(x)}{\langle \hat\alpha\cdot \b\rangle^{d+1}}\frac{d\alpha}{|\alpha|^{d+1}}=\frac{1}{2}\int_{\mathbb{R}^d} \frac{\mathcal{O}_\alpha f(x)}{\langle \hat\alpha\cdot \b\rangle^{d+1}}\frac{d\alpha}{|\alpha|^{d}},
						\end{align*}
						where $\mathcal{O}_\alpha$ is defined in Lemma \ref{lemoalin}. Then \eqref{gggg1} follows from Lemma \ref{lemoalin}.
                \end{proof}
					\begin{lemma}\label{lemKf}
						Consider $K_\b(x,z)$ defined in \eqref{defK}.	For any fixed $\b\in\mathbb{R}^d$, $m_1\in\{0,1\}$, $a\in(0,1)$, there hold\footnote{{Note that $\|\cdot \|_{\dot C^a_{x,z}(\mathbb{R}^{d+1}_+\cup \mathbb{R}^{d+1}_-)}=\|\cdot \|_{\dot C^a_{x,z}(\mathbb{R}^{d+1}_+)}+\|\cdot \|_{\dot C^a_{x,z}(\mathbb{R}^{d+1}_-)}$ is different from  $\|\cdot\|_{\dot C^a_{x,z}(\mathbb{R}^{d+1})}$.}}
						\begin{align}
							&\left\|(\partial_z^{m_1}K_\b(\cdot,z)\ast g)(x)\right\|_{\dot C^a_{x,z}(\mathbb{R}^{d+1}_+\cup \mathbb{R}^{d+1}_-)}\lesssim \langle \b\rangle^2 \|g\|_{\dot C^{m_1+a}},\label{hhhkf}\\
							&\left\|(\partial_z K_\b(\cdot,z)\ast g)(x)\right\|_{L^\infty_{x,z}(\mathbb{R}^{d+1}_+\cup \mathbb{R}^{d+1}_-)}\lesssim  \|g\|_{C^{1,\log^\varkappa}}.\label{lllkf}
						\end{align}
					\end{lemma}
					\begin{proof}
						Note that 
						\begin{align*}
							\mathcal{F}(\partial_z^{m_1}K_\b(\cdot,z))(\xi)=\left(\lambda^\pm(\xi,\b)\right)^{m_1}\exp(\lambda^\pm(\xi,\b)z),\ \ \ z\in\mathbb{R}^\mp.
						\end{align*}
						Hence we can write 
						\begin{align*}
							\partial_z^{m_1}(K_\b(\cdot,z)\ast g)(x)=(K_\b(\cdot,z)\ast( \mathfrak{L}^{m_1}_{\pm,\b} g)(\cdot))(x),
						\end{align*}
						where the operator $\mathfrak{L}_{\pm,\b}$ is defined in \eqref{opL}.
						By \eqref{LLLinf} and \eqref{remintk}, we obtain \eqref{lllkf}.\\
						By \eqref{LHol}, we obtain 
						\begin{align}\label{holx}
							\|	\partial_z^{m_1}(K_\b(\cdot,z)\ast g)(x)	\|_{L^\infty_z(\mathbb{R})\dot C^a_{x}(\mathbb{R}^d)}\lesssim\|\mathfrak{L}^{m_1}_{\pm,\b} g\|_{\dot C^a}\lesssim \|g
							\|_{\dot C^{m_1+a}}.
						\end{align}
						Then we estimate the  H\"{o}lder norm in $z$-variable. For simplicity, we only consider $z\in\mathbb{R}^+$. Then for any  $\beta\in \mathbb{R}^+$,
						\begin{align*}
							&\partial_z^{m_1}(K_\b(\cdot,z)\ast g)(x)-	\partial_z^{m_1}(K_\b(\cdot,\beta)\ast g)(x)\\
							&\quad\quad=\int_{\mathbb{R}^{d}}(K_\b(y,z)-K_\b(y,\beta))( \mathfrak{L}_{-,\b}^{m_1}g)(x-y)dy\\
							&\quad\quad=\int_{\mathbb{R}^{d}}(K_\b(y,z)-K_\b(y,\beta))\left(( \mathfrak{L}_{-,\b}^{m_1}g)(x-y)-( \mathfrak{L}_{-,\b}^{m_1}g)(x)\right)dy.
						\end{align*}
                        where in the last identity we used \eqref{remintk} in Lemma \ref{esK11}  to get
						\begin{align*}
					\int_{\mathbb{R}^{d}}\big(K_\b(y,z)-K_\b(y,\beta)\big)dy	=0.
						\end{align*}
						Then, by \eqref{esl1} and \eqref{LHol}, we obtain that 
						\begin{align*}
							&|((\partial_z^{m_1}K_\b)(\cdot,z)\ast g)(x)-	((\partial_z^{m_1}K_\b)(\cdot,\beta)\ast g)(x)|\\
							&\quad\quad\lesssim \int_{\mathbb{R}^{d}}|K_\b(y,z)-K_\b(y,\beta)||y|^ady \ \| \mathfrak{L}_{\pm,\b}^{m_1}g\|_{\dot C^a_x}\\
							&\quad\quad\lesssim \langle \b\rangle^2|z-\beta|^a\ \|g
							\|_{\dot C^{m_1+a}},
						\end{align*}
						which implies that 
						\begin{align*}
							\|\partial_z^{m_1}(K_\b(\cdot,z)\ast g)(x)\|_{L^\infty(\mathbb{R}^d,\dot C^a_z(\mathbb{R}^+))}\lesssim \langle \b\rangle^2\|g
							\|_{\dot C^{m_1+a}}.
						\end{align*}
						Combining this with \eqref{holx}, we get 
						\begin{align*}
							\sup_{z>0}\|\partial_z^{m_1}(K_\b(\cdot,z)\ast g)(x)\|_{\dot C^a(\mathbb{R}^{d+1}_+)}\lesssim \langle \b\rangle^2\|g
							\|_{\dot C^{m_1+a}}.
						\end{align*}
						The estimate in $\mathbb{R}^{d+1}_-$ follows by the same argument.		This completes the proof of the lemma.
					\end{proof}\vspace{0.2cm}\\
                   We now derive estimates for the frozen velocity $\tilde v^\pm_{\b}[f]$.

					\begin{lemma}\label{lemtvbhol}
						For any $\b\in\mathbb{R}^d$ and $\tilde v_\b^\pm[f]$ defined in \eqref{fortvb}, and any $n\in\mathbb{N}$,  it holds
						\begin{align}
							&\sum_{+,-}	\|\nabla_x^n\nabla_{x,z}\tilde v_\b^\pm[f]\|_{L^\infty(\mathbb{R}^{d+1}_\pm)}\lesssim \|\nabla_x^n\mathcal{G}[f]\|_{ C^{1,\log^\varkappa}},\ \ \varkappa>1,\label{Linftv}\\
							&\sum_{+,-}	\|\nabla_x^n\nabla_{x,z}\tilde v_\b^\pm[f]\|_{\dot C^a(\mathbb{R}^{d+1}_\pm)}\lesssim \langle\b \rangle^2\|\nabla_x^n\mathcal{G}[f]\|_{\dot C^{1+a}},\ \ \ a\in (0,1),\label{Holtv}\\
							&\sum_{+,-}	\|\nabla_x^n\nabla_{x,z}\tilde v_\b^\pm[f]\|_{L^2(\mathbb{R}^{d+1}_\pm)}\lesssim \langle\b \rangle^2\|\nabla_x^n\mathcal{G}[f]\|_{\dot H^{\frac{1}{2}}}.\label{L2tv}
						\end{align}
					\end{lemma}
					\begin{proof}
The estimates \eqref{Linftv} and \eqref{Holtv} follow directly from
Lemma~\ref{lemKf}, applied to $g=\nabla_x^n\mathcal G[f]$.
It remains to prove \eqref{L2tv}.
					By \eqref{foutv} and Parseval's Theorem, we obtain 
						\begin{align*}
							\|\nabla_x^n\nabla_{x,z}\tilde v_\b^+[f]\|_{L^2(\mathbb{R}^{d+1}_+)}^2&\lesssim \int_0^\infty \int_{\mathbb{R}^d}|\xi|^{2n}(|\lambda^-(\xi,\b)|+|\xi|)^2 |e^{\lambda^-(\xi,\b)z}|^2 |\hat {\mathcal{G}}[f](\xi)|^2 d\xi dz\\
							&\lesssim \langle \b\rangle^2\int_{\mathbb{R}^d}|\xi|^{2n+1} |\hat {\mathcal{G}}[f](\xi)|^2 d\xi \\
							&\lesssim \langle \b\rangle^2 \|\nabla_x^n\mathcal{G}[f]\|_{\dot H^\frac{1}{2}}^2.
						\end{align*}
						This completes the proof of the lemma.
					\end{proof}\vspace{0.2cm}\\
                    We next need a basic energy estimate for the full transmission solution
$v_f^\pm$. This estimate is subcritical and will be used only to control the
$L^2$ part of the forcing terms in the equation for $\omega_\b^\pm[f]$.
					
					\begin{lemma}\label{wL2}
						Let $v^\pm=v_f^\pm$ be the weak solution to system \eqref{elleqeta}, then 
						\begin{align}
							\label{vl2}	&	\sum_{+,-}		\|\nabla _{x,z}v^\pm\|_{L^2(\tilde \Omega_f^\pm)}
							\lesssim (1+\|\nabla f\|_{L^\infty})^8\|f\|_{H^{5/2}}.
						\end{align}	
					\end{lemma}
					\begin{proof}
						Recall that $v^\pm=v^\pm_f$ is the weak solution to the elliptic system 
						\begin{equation}\label{eqmap}
							\begin{aligned}
								&\operatorname{div}_{x,z}(\M(\nabla f)\nabla_{x,z} v^\pm )=0,\ \ \text{in}\ \tilde \Omega^\pm_f,\\
								&v^--v^+=\mathcal{G}[f]+\varrho_0f, \quad \text { on } \{z=0\},\\
								&e_{d+1}\cdot (\M(\nabla f)(\mu_+^{-1}\nabla_{x,z}v^+-\mu_-^{-1}\nabla_{x,z}v^-))=0, \quad \text { on } \{z=0\},\\
								&\tilde \nu^\pm\cdot (\M(\nabla f)\nabla_{x,z} v^\pm) =0,\ \ \ \ \ \ \ \ \ \text{on} \ \tilde \Gamma^\pm_f.
							\end{aligned}
						\end{equation}
						Using $v^\pm$ as a test function to \eqref{eqmap}, we obtain  
						\begin{align*}
							0=&\mu_+^{-1}\int_{\tilde \Omega_f^+}\operatorname{div}_{x,z}(\M(\nabla f)\nabla_{x,z} v^+ ) v^+dxdz+\mu_-^{-1}\int_{\tilde \Omega_f^-}\operatorname{div}_{x,z}(\M(\nabla f)\nabla_{x,z} v^-) v^-dxdz\\
							=&-\mu_+^{-1}\int_{\tilde \Omega_f^+}(\nabla_{x,z} v^+)^\top   \M(\nabla f)\nabla_{x,z} v^+dxdz-\mu_-^{-1}\int_{\tilde \Omega_f^-}(\nabla_{x,z} v^-)^\top   \M(\nabla f)\nabla_{x,z} v^-dxdz\\
							&-\mu_+^{-1}\int_{\mathbb{R}^d}e_{d+1}\cdot (\M(\nabla f)\nabla_{x,z}v^+(x,0))v^+(x,0)dx+\mu_-^{-1}\int_{\mathbb{R}^d}e_{d+1}\cdot (\M(\nabla f)\nabla_{x,z}v^-(x,0))v^-(x,0)dx.
						\end{align*}
						By the ellipticity of $\M(\nabla f)$, we have 
						\begin{align*}
							&	(1+\|\nabla f\|_{L^\infty})^{-2}\left(\|\nabla _{x,z}v^+\|_{L^2(\tilde \Omega_f^+)}^2+\|\nabla _{x,z}v^-\|_{L^2(\tilde \Omega_f^-)}^2\right)\\
							&\quad\leq \int_{\tilde \Omega_f^+}(\nabla_{x,z} v^+)^\top   \M(\nabla f)\nabla_{x,z} v^+dxdz+\int_{\tilde \Omega_f^-}(\nabla_{x,z} v^-)^\top   \M(\nabla f)\nabla_{x,z} v^-dxdz.
						\end{align*}
						On the other hand, with the boundary conditions in \eqref{eqmap}, we obtain that 
						\begin{align*}
							&-\mu_+^{-1}\int_{\mathbb{R}^d}e_{d+1}\cdot (\M(\nabla f)\nabla_{x,z}v^+(x,0))v^+(x,0)dx+\mu_-^{-1}\int_{\mathbb{R}^d}e_{d+1}\cdot (\M(\nabla f)\nabla_{x,z}v^-(x,0))v^-(x,0)dx\\
							&\quad=\mu_+^{-1}\int_{\mathbb{R}^d}e_{d+1}\cdot (\M(\nabla f)\nabla_{x,z}v^+(x,0))\left(\mathcal{G}[f]+\varrho_0f\right)dx.
						\end{align*}
                        By the trace theorem, there exists an extension $\mathbf f$ in $\tilde\Omega_f^+$
with trace $\mathcal G[f]+\varrho_0 f$ on $\{z=0\}$ and
						\begin{align*}
							\|\mathbf{f}\|_{ H^1(\mathbb{R}^{d+1}_+)}\lesssim \left\|\mathcal{G}[f]+\varrho_0f\right\|_{  H^\frac{1}{2}},
						\end{align*}
						then applying Green's formula and \eqref{eqmap} yields
						\begin{align*}
							\int_{\mathbb{R}^d}e_{d+1}\cdot (\M(\nabla f)\nabla_{x,z}v^+(x,0))\left(\mathcal{G}[f]+\varrho_0f\right)(x)dx&= \int_{\tilde \Omega_f^+}\operatorname{div}_{x,z}\left(\M(\nabla f)\nabla_{x,z}v^+(x,z)\mathbf{f}(x,z)\right)dx\\
							&=\int_{\tilde \Omega_f^+} (\nabla _{x,z}\mathbf{f}(x,z))^\top  \M(\nabla f)\nabla_{x,z}v^+(x,z)dxdz.
						\end{align*}
						Hence, it follows from  H\"{o}lder's inequality 
						\begin{align*}
								\|\nabla _{x,z}v^+\|_{L^2(\tilde \Omega_f^+)}^2&+\|\nabla _{x,z}v^-\|_{L^2(\tilde \Omega_f^-)}^2\\
							&\leq C(1+||\nabla f||_{L^\infty})^4\|\nabla_{x,z}v^+\|_{L^2(\tilde \Omega_f^+)}\|\nabla _{x,z}\mathbf{f}\|_{L^2(\mathbb{R}^{d+1})}\\&\leq \frac{1}{10}\|\nabla_{x,z}v^+\|_{L^2(\tilde \Omega_f^+)}^2+C(1+||\nabla f||_{L^\infty})^8\left\|\mathcal{G}[f]+\varrho_0f\right\|_{  H^\frac{1}{2}}^2.
						\end{align*}
						Then we obtain 
						\begin{align*}
							\|\nabla _{x,z}v^+\|_{L^2(\tilde \Omega_f^+)}+\|\nabla _{x,z}v^-\|_{L^2(\tilde \Omega_f^-)}&\lesssim (1+||\nabla f||_{L^\infty})^8(\|\mathcal{G}[f]\|_{H^\frac{1}{2}}+\|f\|_{H^\frac{1}{2}}).
						\end{align*}
						Then we obtain \eqref{vl2} and  complete the proof of the lemma.
					\end{proof}\vspace{0.3cm}\\	
                    The next lemma shows that, away from the interface $\{z=0\}$, the frozen velocity contributes only lower-order terms. This follows from the decay of the Poisson kernel in the normal variable.
					\begin{lemma}\label{estilv1low}
						For any $\b\in\mathbb{R}^d$, let $\tilde v_\b^\pm[f]$ be as defined in \eqref{fortvb}. Then for any $\sigma>0$, $n\in\mathbb{N}$, and any $a\in(0,1)$,
						\begin{align*}
							\|\nabla_x^n\nabla_{x,z}\tilde v_\b^\pm[f]\|_{L^\infty( \mathbb{R}^{d+1}_\pm\cap\{|z|\geq \sigma\})}&+		\|\nabla_x^n\nabla_{x,z}\tilde v_\b^\pm[f]\|_{\dot C^a( \mathbb{R}^{d+1}_\pm\cap\{|z|\geq \sigma\})}\\
							&\lesssim (1+\sigma)^{-(n+2)}\langle\b\rangle^{2n+8}(1+\|\nabla f\|_{L^\infty})^{5}\|\nabla f\|_{\dot C^1}.
						\end{align*}
					\end{lemma}
					\begin{proof}
						It follows from \eqref{deriK} that 
						\begin{align*}
							\sup_{|z|\geq \sigma}\|\nabla_x^{n_1}\partial_z^{n_2}K_\b(x,z)\|_{L_x^1}\lesssim (1+\sigma)^{-(1+n_1+n_2)}\langle\b\rangle^{3+2(n_1+n_2)},\ \ \ \forall n_1,n_2\in\mathbb{N}.
						\end{align*}
						Then by the formula of $\tilde v_\b^\pm[f]$ in \eqref{fortvb}, one has 
						\begin{align*}
							\sup_{|z|\geq \sigma}|\nabla_x^n\nabla_{x,z}\tilde v_\b^\pm[f]|\lesssim 	\sup_{|z|\geq \sigma}\|\nabla_x^n\nabla_{x,z}K_\b(x,z)\|_{L^1_x} \|\mathcal{G}[f]\|_{L^\infty}\lesssim (1+\sigma)^{-(n+2)}\langle\b\rangle^{2n+5}\|\mathcal{G}[f]\|_{L^\infty}.
						\end{align*}
						Moreover,
\[
\|\mathcal G[f]\|_{L^\infty}
\lesssim
(1+\|\nabla f\|_{L^\infty})^5
\|\nabla f\|_{\dot C^1}.
\]
The Hölder estimate is obtained similarly. Using the $L^1$ bounds for one
additional derivative of the kernel and interpolation, we get
\[
\|\nabla_x^n\nabla_{x,z}\tilde v_\b^\pm[f]\|_{\dot C^a(\{|z|\geq\sigma\})}
\lesssim
(1+\sigma)^{-(n+2)}
\langle\b\rangle^{2n+8}
\|\mathcal G[f]\|_{L^\infty}.
\]
Combining the last two estimates gives the desired result.
					\end{proof}
				~\vspace{0.3cm}\\
                With Lemmas~\ref{lemtvbhol} and \ref{estilv1low} available, we now turn to the
elliptic correction $\omega^\pm_{\b}[f]$ and estimate the boundary quantity
\[
\mathcal Q[f](x)
=
\left.(\nabla_{x,z}\omega^+_{\b}[f])(x,z)\right|_{\b=\nabla f(x),\,z=0}.
\]
This is the key estimate for the implicit nonlinear term $\H^{im}[f]$. The main
difficulty is that $\omega^\pm_{\b}[f]$ is not explicit; it solves the
transmission problem \eqref{elliw}. We combine the endpoint transmission
estimate, the bounds for the frozen velocity, and the separation condition
\eqref{condfff}.
					\begin{lemma}\label{lemH2w}Consider  $\|\cdot\|_{X_T}$ defined in \eqref{defnorgm}, 	
						then for any $T\in(0,\frac{1}{2})$, and any function $f$ satisfying \eqref{condfff}, we have 
						\begin{align}\label{omeb0}
							& \sup_{t\in[0,T]}t^\frac{2}{3}\|\mathcal{Q}[f](t)\|_{L^\infty}\lesssim (\|f-\phi\|_{X_T}+T^\frac{1}{20}\mathfrak{M}_\phi)(1+\|f\|_{X_T})^{10},\\
							&\sup_{t\in[0,T]} t^\frac{m+\kappa}{3} (|\log t|^\varkappa\|\nabla^m\mathcal{Q}[f](t)\|_{\dot C^{\kappa-2}}+\|\nabla^m\mathcal{Q}[f](t)\|_{ \HH^{\kappa-2}}	)\lesssim ( \| f-\phi\|_{X_T}+T^\frac{1}{20}\mathfrak{M}_\phi)^2(1+\| f\|_{X_T})^{2m+10},\label{omeba}
						\end{align}
                          where $\mathfrak{M}_\phi=\|\phi\|_{H^{m+4}\cap C^{m+4}}$.
					\end{lemma}
					\begin{proof} We split the proof into two parts. First, we obtain the zeroth-order
$L^\infty$ bound \eqref{omeb0} by applying the endpoint Lipschitz estimate to
the original transmission solution $v_f^\pm$. This avoids differentiating the
equation for $\omega^\pm_{\b}$ directly. Second, we prove the critical
$\dot C^a\cap\HH^a$ estimate, where $a=\kappa-2$, by applying the frozen
transmission representation to \eqref{elliw}.\vspace{0.2cm}\\
                    We denote \begin{align*}
					    &\mathbb{R}^{d+1}_\tau=\mathbb{R}^{d}\times (0,\tau)\ \text{if}\ \tau>0,\ \quad\quad \mathbb{R}^{d+1}_\tau=\mathbb{R}^{d}\times (\tau,0)\ \text{if}\ \tau<0.
					\end{align*}
						We first prove \eqref{omeb0}.  Note that $$\|\mathcal{Q}[f]\|_{L^\infty}\leq \sup_\b\|\nabla_{x,z} \omega_\b^+[f]\|_{L^\infty(\mathbb{R}^{d+1}_{r/2})},$$
                           where the parameter $r<\frac{\mathbf{r}}{10}$ will be fixed later. 
						To avoid dealing with the remainder term and the issue of losing derivatives, we go back to the equation of $v_f^\pm$ rather than directly estimating $\omega_\b^\pm$. From \eqref{elleqeta}, we have 
						\begin{align*}
							&\operatorname{div}_{x,z}(\M(\nabla f)\nabla_{x,z} v^\pm)=0,\ \ \text{in}\ \mathbb{R}^{d+1}_{-4r}\cup \mathbb{R}^{d+1}_{4r},\\
							&v^--v^+=\mathcal{G}[f]+\varrho_0f, \quad \text { on } \{z=0\},\\
							&e_{d+1}\cdot (\M(\nabla f)\nabla_{x,z} (\mu_+^{-1}v^+-\mu_-^{-1}v^-))=0, \quad \text { on } \{z=0\}.
						\end{align*}
						By Lemma \ref{lemLip}, we can take $r$ sufficiently small such that 
						\begin{align*}
							\|\nabla_{x,z} v\|_{L^\infty(\mathbb{R}^{d+1}_{-r/2}\cup\mathbb{R}^{d+1}_{r/2})}\lesssim &\|\mathcal{G}[f]\|_{\dot C^{1,\log^\varkappa}}+\|f\|_{\dot C^{1,\log^\varkappa}}
							+C(r)\|\nabla_{x,z}v\|_{ L^2(\mathbb{R}^{d+1}_{-3r}\cup\mathbb{R}^{d+1}_{3r})}.
						\end{align*}
						Combining this with \eqref{vl2}, we deduce that 
						\begin{equation}\label{vlinfty}
							\begin{aligned}
								\|\nabla_{x,z} v\|_{L^\infty(\mathbb{R}^{d+1}_{-r/2}\cup\mathbb{R}^{d+1}_{r/2})}&\lesssim \|\mathcal{G}[f]\|_{\dot C^{1,\log^\varkappa}}+\|f\|_{\dot C^{1,\log^\varkappa}}+(1+\|\nabla f\|_{L^\infty})^{7}\|f\|_{H^\frac{5}{2}}\\
								&\lesssim(\|\nabla f\|_{\dot C^{2,\log^\varkappa}\cap \dot C^{\log^\varkappa}}+\|f\|_{H^\frac{5}{2}})(1+\|\nabla f\|_{L^\infty})^{7}\\
								&\lesssim t^{-\frac{2}{3}}(\|f\|_{T,*}+T^\frac{1}{2}\|f\|_{X_T})(1+\|f\|_{T})^7,
							\end{aligned}
						\end{equation}
                        	where $\|\cdot\|_{T,*}$ is defined in \eqref{defts}. 
						Combining this with Lemma \ref{lemtvbhol} and the definition $\omega_\b^\pm=v^\pm-\tilde v_\b^\pm$ in \eqref{S4eq4}, we obtain 
						\begin{align*}
							\sup_{\b}\|\nabla_{x,z} \omega_\b^\pm[f]\|_{L^\infty(\mathbb{R}^{d+1}_{\pm r/2})}\lesssim  t^{-\frac{2}{3}}(\|f\|_{T,*}+T^\frac{1}{2}\|f\|_{X_T})(1+\|f\|_{T})^7,
						\end{align*}
						which yields  \eqref{omeb0} as a result of \eqref{stsm}.

                        Then we prove \eqref{omeba}. For simplicity, we only consider the case
$m=0$, namely the estimate of $\|\mathcal Q[f]\|_{\dot C^a\cap\HH^a}$.
The general case follows by applying horizontal derivatives $\nabla_x^m$ to
the equation and repeating the same argument.
						Note that 
						\begin{equation}\label{eqQ1}
							\begin{aligned}
								\|\mathcal{Q}[f]\|_{\dot C^a}\lesssim &\sup_{\beta}\frac{\|\delta_\beta^x(\nabla_{x,z}\omega_\b^+)(x,0)|_{\b=\nabla f(x)}\|_{L^\infty}}{|\beta|^a}\\
								&+\sup_{\beta}\frac{\|(\nabla_{x,z}\omega_\b^+)(x,0)|_{\b=\nabla f(x)}-(\nabla_{x,z}\omega_\b^+)(x,0)|_{\b=\nabla f(x-\beta)}\|_{L^\infty}}{|\beta|^a}\\
								:=&\mathbf{Q}_1+\mathbf{Q}_2,
							\end{aligned}
						\end{equation}
						and 
							\begin{equation}\label{eqQ11}
							\begin{aligned}
								\|\mathcal{Q}[f]\|_{\HH^a}\lesssim &\sup_{\beta}\frac{\|\delta_\beta^x(\nabla_{x,z}\omega_\b^+)(x,0)|_{\b=\nabla f(x)}\|_{ L^2}}{|\beta|^a}\\
								&+\sup_{\beta}\frac{\|(\nabla_{x,z}\omega_\b^+)(x,0)|_{\b=\nabla f(x)}-(\nabla_{x,z}\omega_\b^+)(x,0)|_{\b=\nabla f(x-\beta)}\|_{ L^2}}{|\beta|^a}\\
								:=&\mathbf{Q}_1'+\mathbf{Q}_2'.
							\end{aligned}
						\end{equation}
						We estimate the right hand side of \eqref{eqQ1} term by term. 
						Applying Lemma \ref{leminterf} to \eqref{elliw} with $A(x)=\M(\b)$, $(g,h)=(F_{1,\b},F_2)$, we obtain for any $(x,z)\in\mathbb{R}^{d+1}_{r}\cup\mathbb{R}^{d+1}_{-r}$, 
						\begin{align*}
							\omega_\b(x,z)=\T_{1,\b}F_{1,\b}(x,z)+ \T_{2,\b} F_2(x,z)+\mathrm{LO}(\omega_\b, F_{1,\b})(x,z),
						\end{align*}
						with
						\begin{align*}
							&\T_{1,\b}F_{1,\b}(x,z)=-\frac{\mu_+\mathbf{1}_{z\in(0,r)}+\mu_-\mathbf{1}_{z\in(-r,0)}}{\mu_++\mu_-}\frac{1}{r}\int_{r}^{2r}\int_{\mathbb{R}^{d+1}_{\tau}\cup \mathbb{R}^{d+1}_{-\tau}} F_{1,\b}(y,w)\cdot\nabla_{y,w}\G_{\b}^D(x,z,y,w)dydwd\tau\\
							&\quad\quad\quad-\frac{\mu_+\mu_-}{\mu_++\mu_-}\sum_{+,-}\left(\frac{1}{r\mu_\pm}\int_{r}^{2r}\int_{\mathbb{R}^{d+1}_{\pm\tau}} F_{1,\b}(y,w)\cdot\nabla_{y,w}\G_{\b}^N(x,z,y,w)dydwd\tau\right),\\
							&\T_{2,\b}F_{2}(x,z)=\frac{\mu_+\mathbf{1}_{z\in(0,r)}+\mu_-\mathbf{1}_{z\in(-r,0)}}{\mu_++\mu_-}\int_{\mathbb{R}^d}e_{d+1}\cdot(\M(\b)\nabla_{y,w}\G_{\b}^D(x,z,y,0)F_2(y))dy,
						\end{align*}
					where  the kernels $\G_\b^N $, $\G_\b^D$ are defined by \eqref{defGtiG} with $Q_{x_0}$ replaced by $\M(\b)$. From Lemma \ref{lem:param-selection}, we know that the lower order term $\mathrm{LO}(\omega_\b,F_{1,\b})$ satisfies
						\begin{align*}
							&\|\nabla_x^n\nabla_{x,z}\mathrm{LO}(\omega_\b,F_{1,\b})(x,0^+)|_{\b=\nabla f(x)}\|_{L^\infty\cap L^2(\mathbb{R}^{d})}\\
                            &\quad\quad\lesssim\sup_{\b}\|\nabla_\b^l\nabla_x^n\nabla_{x,z}\mathrm{LO}(\omega_\b,F_{1,\b})(x,0^+)\|_{L^\infty\cap L^2(\mathbb{R}^{d})}\\
							&\quad\quad\lesssim C(r,n) \sup_\b(\|\nabla_{x,z}\omega_\b\|_{L^2(\mathbb{R}^{d+1}_{2r}\cup \mathbb{R}^{d+1}_{-2r})}+\| F_{1,\b}\|_{L^2(\mathbb{R}^{d+1}_{2r}\cup\mathbb{R}^{d+1}_{-2r})}),\ \ \forall n\in\mathbb{N},
						\end{align*}
                        where $l\geq \frac{d}{2}+1$. 
						Combining this with Lemma \ref{wL2} yields
						\begin{align*}
							\|\nabla_x^n\nabla_{x,z}\mathrm{LO}(\omega_\b,F_{1,\b})(x,0^+)|_{\b=\nabla f(x)}\|_{L^\infty\cap L^2(\mathbb{R}^{d})}\lesssim C(r,n)(1+\|\nabla f\|_{L^\infty})^{10}\|f\|_{H^\frac{5}{2}},\ \ \forall n\in\mathbb{N}.
						\end{align*}
						Denote 
						$$
						\L_{1}(x,\b)=(\nabla_{x,z}\T_{1,\b}F_{1,\b}(x,z))|_{z=0^+},\ \ \ \ \L_{2}(x,\b)=(\nabla_{x,z}\T_{2,\b}F_{2}(x,z))|_{z=0^+}.
						$$
						Then it follows that 
						\begin{align*}
							\mathcal{Q}[f](x)=\L_1(x,\nabla f(x))+\L_2(x,\nabla f(x))+R(x),
						\end{align*}
						where $R(x)$ satisfies 
						\begin{align}\label{rereLO}
							\|R\|_{\dot C^{a}\cap \HH^{a}}\lesssim (1+\|\nabla f\|_{L^\infty})^{10}\|f\|_{H^\frac{5}{2}}\lesssim t^{-\frac{1}{2}} (\|f\|_{T,*}+T^\frac{1}{2}\|f\|_{X_T})(1+\|f\|_T)^{10}.
						\end{align}
						By Lemma \ref{remF2} and \eqref{hhhkf} in Lemma \ref{lemKf}, we have 
						\begin{align}\label{rereF2}
							\|\L_2(x,\nabla f(x))\|_{\dot C^a\cap\HH^a}&\lesssim (1+\|\nabla f\|_{L^\infty})^2\|F_2\|_{\dot C^{1+a}}\lesssim (1+\|\nabla f\|_{L^\infty})^2\|\nabla f\|_{\dot C^{a}\cap \HH^{a}}\\
							&\lesssim t^{-\frac{a}{3}}\|f\|_{T,*}(1+\|f\|_T)^{10}.\nonumber
						\end{align}
						To control $\L_1$, it suffices to consider  
						\begin{align*}
							&\L^i_1(x,\b)=\int_{\mathbb{R}^{d+1}}K_i(x,z,y,w)\tilde F_{1,\b}(y,w)dydw\Big|_{z=0^+},\ \ i=1,2,
						\end{align*}
						with $\tilde F_{1,\b}(y,w)=F_{1,\b}(y,w)\mathbf{1}_{|w|\leq \tau}$ for $\tau\in(r,2r)$, and 
						\begin{align*}
							&K_1(x,z,y,w)=(\nabla_{x,z}\nabla_{y,w}\G_{\b}^D)(x,z,y,w),\\
							&K_2(x,z,y,w)=(\nabla_{x,z}\nabla_{y,w} \G_{\b}^N)(x,z,y,w).
						\end{align*}
						By Lemma \ref{lemsgit}, we obtain for any $p\in\{2,\infty\}$,
						\begin{equation}\label{Tf1}
							\begin{aligned}
								&\sum_{i=1,2}\sup_{
									\beta}\frac{\|\L^i_{1}(x,\nabla f(x))-\L^i_{1}(x-\beta,\nabla f(x))\|_{L^p_x}}{|\beta|^a}\\
									&\quad\quad\quad \lesssim 	\sup_{\substack{y,u\in\mathbb{R}^d, z,w\in\mathbb{R}\\
										|u|\leq |(y,w)| }}\frac{\|(\tilde F_{1,\b}(x,z)-\tilde F_{1,\b}(x-y,z-w))|_{\b=\nabla f(x+u)}\|_{L_x^p}}{|(y,w)|^a}.
							\end{aligned}
						\end{equation}
						Recalling the definition of $F_{1,\b}$ in \eqref{fortems}, we have 
						\begin{align*}
					&	|	\tilde F_{1,\b}(x,z)|_{\b=\nabla f(x+u)}|+	|	\tilde F_{1,\b}(x-y,z-w)|_{\b=\nabla f(x+u)}|\\
						&\quad\quad\quad\leq( |\M(\nabla f(x))-\M(\nabla f(x+u))|+ |\M(\nabla f(x-y))-\M(\nabla f(x+u))|)\|\nabla_{x,z} v_f\|_{L^\infty(\mathbb{R}^{d+1}_{2r}\cup\mathbb{R}^{d+1}_{-2r})}. 
						\end{align*}
						Note that 
						\begin{align*}
							\|\M(\nabla f(x))-\M(\nabla f(x'))\|_{L^\infty\cap L^2}
							&\lesssim |x-x'|^a\|\nabla f\|_{\dot C^a\cap \HH^a}(1+\|\nabla f\|_{L^\infty}).
						\end{align*}
				By \eqref{vlinfty},  the right hand side of \eqref{Tf1} is bounded by
						\begin{align*}
							 \|\nabla f\|_{\dot C^a}\|\nabla_{x,z}v_f\|_{L^\infty(\mathbb{R}^{d+1}_{2r}\cup\mathbb{R}^{d+1}_{-2r})}&(1+\|\nabla f\|_{L^\infty})\\
                            &\lesssim |\log t|^{-\varkappa}t^{-\frac{2+a}{3}} \|f\|_{T,*}(\|f\|_{T,*}+T^\frac{1}{2}\|f\|_{X_T})(1+\|f\|_T)^{10},\ \ \ p=\infty,\\
							\|\nabla f\|_{\HH^a}\|\nabla_{x,z}v_f\|_{L^\infty(\mathbb{R}^{d+1}_{2r}\cup\mathbb{R}^{d+1}_{-2r})}&(1+\|\nabla f\|_{L^\infty})\\
                        &\lesssim t^{-\frac{2+a}{3}} \|f\|_{T,*}(\|f\|_{T,*}+T^\frac{1}{2}\|f\|_{X_T})(1+\|f\|_T)^{10},\ \ \ p=2.
						\end{align*}
						Combining this with   \eqref{rereLO}, \eqref{rereF2}, and the definitions \eqref{eqQ1}, \eqref{eqQ11}, we obtain 
						\begin{equation}\label{q1}
							\begin{aligned}
								\mathbf{Q}_1&\lesssim \sup_{|\beta|\leq r}\frac{\|\L_1(x,\nabla f(x))-\L_1(x-\beta,\nabla f(x))\|_{L^\infty}}{|\beta|^a}+\|\L_2(x,\nabla f(x))\|_{\dot C^a}+\|R\|_{\dot C^a}\\
								&\lesssim |\log t|^{-\varkappa}t^{-\frac{2+a}{3}} \|f\|_{T,*}(\|f\|_{T,*}+T^\frac{1}{20}(1+\|f\|_{X_T}))(1+\|f\|_{T})^{10},
							\end{aligned}
						\end{equation}
						and 
							\begin{equation}\label{q1'}
							\begin{aligned}
								\mathbf{Q}_1'&\lesssim \sup_{|\beta|\leq r}\frac{\|\L_1(x,\nabla f(x))-\L_1(x-\beta,\nabla f(x))\|_{L^2}}{|\beta|^a}+\|\L_2(x,\nabla f(x))\|_{\HH^a}+\|R\|_{\HH^a}\\
								&\lesssim t^{-\frac{2+a}{3}} \|f\|_{T,*}(\|f\|_{T,*}+T^\frac{1}{20}(1+\|f\|_{X_T}))(1+\|f\|_{T})^{10}.
							\end{aligned}
						\end{equation}
					Then we estimate  $\mathbf{Q}_2$ and $\mathbf{Q}_2'$. Using the property that 
						$\tilde v_\b+\omega_\b$ is independent of $\b$, and applying \eqref{dx0}, \eqref{z0}, we obtain 
						\begin{align*}
							&\left|(\nabla_{x,z}\omega_\b)(x,0)|_{\b=\nabla f(x)}-(\nabla_{x,z}\omega_\b)(x,0)|_{\b=\nabla f(x-\beta)}\right|\\
							&=\left|(\nabla_{x,z}\tilde v_\b)(x,0)|_{\b=\nabla f(x)}-(\nabla_{x,z}\tilde v_\b)(x,0)|_{\b=\nabla f(x-\beta)}\right|\\
							&\lesssim \left|\int_{\mathbb{R}^d}\delta_\alpha \mathcal{G}(x)\left( \frac{1}{\langle \hat \alpha\cdot \nabla f(x)\rangle^{d+1}}-\frac{1}{\langle \hat \alpha\cdot \nabla f(x-\beta)\rangle^{d+1}}\right)\frac{d\alpha}{|\alpha|^{d+1}}\right|\\
							&\quad\quad+\left|\frac{\nabla f(x)}{\langle \nabla f(x)\rangle^2}-\frac{\nabla f(x-\beta)}{\langle \nabla f(x-\beta)\rangle^2}\right| \|\nabla\mathcal{G}\|_{L^\infty}\\
							&:=I_1+I_2.
						\end{align*}
						We have 
						\begin{align*}
						\|	I_1\|_{L^\infty\cap L^2}&\lesssim |\beta|^\frac{a}{2}\|\nabla f\|_{\dot C^\frac{a}{2}\cap \HH^\frac{a}{2}}\int_{|\alpha|\leq |\beta|} \|\delta_\alpha\delta_{-\alpha} \mathcal{G}\|_{L^\infty}\frac{d\alpha}{|\alpha|^{d+1}}+|\beta|\|\nabla f\|_{\dot C^1\cap \HH^1}\int_{|\alpha|\geq |\beta|} \|\delta_\alpha \mathcal{G}\|_{L^\infty}\frac{d\alpha}{|\alpha|^{d+1}}\\
							&\lesssim |\beta|^a(\|\nabla f\|_{\dot C^\frac{a}{2}\cap \HH^\frac{a}{2}}\|\mathcal{G}\|_{\dot C^{1+\frac{a}{2}}}+\|\nabla f\|_{\dot C^1\cap \HH^1}\|\mathcal{G}\|_{\dot C^a}).
						\end{align*}
						And 
						\begin{align*}
						\|	I_2\|_{L^\infty\cap L^2}\lesssim |\beta|^a\|\nabla f\|_{\dot C^a\cap \HH^a}\|\nabla \mathcal{G}\|_{L^\infty}.
						\end{align*}
						Combining this with \eqref{ndefG} yields
						\begin{align*}
							\|	I_1\|_{L^\infty\cap L^2}+	\|	I_2\|_{L^\infty\cap L^2}	\lesssim |\beta|^a |\log t|^{-\varkappa}t^{-\frac{2+a}{3}} \|f\|_{T,*}^2(1+\|f\|_{X_T})^{10},
						\end{align*}
						which implies 
						\begin{align}\label{q3}
							\mathbf{Q}_2+	\mathbf{Q}_2'\lesssim |\log t|^{-\varkappa}t^{-\frac{2+a}{3}} \|f\|_{T,*}^2(1+\|f\|_{X_T})^{10}.
						\end{align}
						Then we conclude from \eqref{eqQ1}, \eqref{eqQ11}, \eqref{q1}, \eqref{q1'} and \eqref{q3} that 
						\begin{align*}
						&	\|\mathcal{Q}[f]\|_{\dot C^a}\lesssim |\log t|^{-\varkappa}t^{-\frac{2+a}{3}} \|f\|_{T,*}(\|f\|_{T,*}+T^\frac{1}{20}(1+\|f\|_{X_T}))(1+\|f\|_{T})^{10},\\
							&	\|\mathcal{Q}[f]\|_{\HH^a}\lesssim t^{-\frac{2+a}{3}} \|f\|_{T,*}(\|f\|_{T,*}+T^\frac{1}{20}(1+\|f\|_{X_T}))(1+\|f\|_{T})^{10}.
						\end{align*}
						Combining this with \eqref{stsm}, we deduce \eqref{omeba} and complete the proof of the lemma.
					\end{proof}\vspace{0.25cm}\\
                    
            We now translate the estimates for the boundary quantity $\mathcal Q[f]$ into estimates for the implicit nonlinear term $\H^{im}[f]$.
					\begin{lemma}
						\label{lemh2}
Let $\H^{im}[f]$ be defined by \eqref{defHim}, and let the norm
$\|\cdot\|_{X_T}$ be defined by \eqref{defnorgm}, with
$\varkappa,m,\kappa$ fixed as in \eqref{consgm}. Then there exist
$\sigma,T>0$, depending only on
$\|\nabla\phi\|_{C^{m+4}}$, $\|\nabla\underline b^\pm\|_{C^1}$, and the
separation constant $\mathbf r$, such that for every
$f\in\mathcal X^\sigma_{T,\phi}$,
						\begin{align}
							&\sup_{t\in[0,T]}|\log t|^\varkappa(t^\frac{\kappa}{3}\|\H^{im}[f](t)\|_{\dot C^{\kappa-2}}+t^\frac{m+\kappa}{3}\|\nabla_x^m\H^{im}[f](t)\|_{\dot C^{\kappa-2}})\nonumber\\
                            &\quad\quad\quad+\sup_{t\in[0,T]}(t^\frac{\kappa}{3}\|\H^{im}[f](t)\|_{\HH^{\kappa-2}}+t^\frac{m+\kappa}{3}\|\nabla_x^m\H^{im}[f](t)\|_{\HH^{\kappa-2}})\nonumber\\
							&\quad\quad\quad\quad\quad\quad\quad\quad\quad\lesssim (\|f-\phi\|_{X_T}+T^\frac{1}{20}\mathfrak{M}_\phi)^2(1+\|f\|_{X_T})^{2(m+10)}. \label{m2}
						\end{align}
					\end{lemma}
					\begin{proof}Recall that
\[
\H^{im}[f]
=
-\frac1{\mu_+}e_{d+1}\cdot
\bigl(\M(\nabla f)\mathcal Q[f]\bigr).
\]
Thus, by the product estimates in Hölder and Besov spaces, for $j=0,m$,
\begin{align*}\|\nabla_x^j\H ^{im}[f]\|_{\dot C^{\kappa-2}}\lesssim \sum_{j_1+j_2=j}\left(\|\M(\nabla f)\|_{\dot C^{j_1}}\|\nabla_x^{j_2}\mathcal{Q}[f]\|_{\dot C^{\kappa-2}}+\|\M(\nabla f)\|_{\dot C^{j_1+\kappa-2}}\|\nabla_x^{j_2}\mathcal{Q}[f]\|_{L^\infty}\right),
						\end{align*}
and similarly,
\begin{align*}
\|\nabla_x^j\H ^{im}[f]\|_{\HH^{\kappa-2}}\lesssim \sum_{j_1+j_2=j}\left(\|\M(\nabla f)\|_{\dot C^{j_1}}\|\nabla_x^{j_2}\mathcal{Q}[f]\|_{\HH^{\kappa-2}}+\|\M(\nabla f)\|_{\HH^{j_1+\kappa-2}}\|\nabla_x^{j_2}\mathcal{Q}[f]\|_{L^\infty}\right).
						\end{align*}
						Applying Lemma \ref{lemcom}, we  obtain 
						\begin{align*}
							&\|\M(\nabla f)\|_{\dot C^{j_1}}\lesssim t^{-\frac{j_1}{3}}(1+\|f\|_T)^{j_1+2},\\
							&|\log t|^{\varkappa}\|\M(\nabla f)\|_{\dot C^{j_1+\kappa-2}}+\|\M(\nabla f)\|_{\HH^{j_1+\kappa-2}}\lesssim  t^{-\frac{j_1+\kappa-2}{3}}\|f\|_{T,*}(1+\|f\|_T)^{j_1+2}.
						\end{align*}
						Combining this with the estimates of $\mathcal{Q}[f]$ in Lemma \ref{lemH2w}, we conclude that 
						\begin{align*}
							\sup_{t\in[0,T]} t^\frac{j+\kappa}{3}&(|\log t|^\varkappa\|\nabla_x^j\H ^{im}[f](t)\|_{\dot C^{\kappa-2}}+\|\nabla_x^j\H ^{im}[f](t)\|_{\HH^{\kappa-2}})\\
                            &\quad\quad\quad\quad\quad\quad\lesssim (\|f-\phi\|_{X_T}+T^\frac{1}{20}\mathfrak{M}_\phi)^2(1+\|f\|_{X_T})^{2(m+10)}.
						\end{align*}
						This completes the proof of the lemma.
					\end{proof}\vspace{0.1cm}\\
					\subsection{Contraction estimates}
We now establish contraction estimates for the nonlinear terms in \eqref{evoeta}. These estimates are used in the uniqueness argument and in the stability of the approximation scheme. The explicit term $\H^{ex}$ is handled by the same commutator estimates as in Lemma~\ref{lemH1}, while the implicit term $\H^{im}$ requires comparing the corresponding elliptic transmission problems.
                   
                    \begin{lemma}\label{lemH1con}
                	Let $\H^{ex}[f]$ be as defined in \eqref{defHgm}, the norm $\|\cdot\|_{T}$ be as defined in \eqref{defnorgm} with constants $\varkappa,m,\kappa$ fixed in \eqref{consgm}. Then for any $T\in(0,\frac{1}{2})$, there holds
					\begin{align*}
						\sup_{t\in[0,T]}&|\log t|^\varkappa(t^\frac{\kappa}{3}\|(\H^{ex}[f]-\H^{ex}[g])(t)\|_{\dot C^{\kappa-2}\cap \HH^{\kappa-2}}+t^\frac{m+\kappa}{3}\|\nabla^m(\H^{ex}[f]-\H^{ex}[g])(t)\|_{\dot C^{\kappa-2}\cap \HH^{\kappa-2}})\\
						&\quad\quad\lesssim \|f-g\|_T(\|(f-\phi,g-\phi)\|_T+T^\frac{1}{20}\mathfrak{M}_\phi)(1+\|(f,g,\phi)\|_T)^{2m+5}, 
					\end{align*}
                    where $\mathfrak{M}_\phi=\|\phi\|_{H^{m+4}\cap C^{m+4}}$.
                    \end{lemma}
					The proof is identical to that of Lemma~\ref{lemH1}, with one factor replaced by the difference $f-g$.\vspace{0.2cm}\\

                    Now we consider the implicit term $\H^{im}[f]-\H^{im}[g]$. As preparation, we first renormalize the 
					transmission elliptic system. Suppose that $f,g$ satisfy \eqref{condfff}, and $v^\pm_f,v^\pm_g$ are solutions to \eqref{elleqeta} with interfaces $f$ and $g$, respectively.  Define $a^\pm(x)=\frac{\underline{b}^\pm(x)-f(x)}{\underline{b}^\pm(x)-g(x)}$ and $\bar{v}^\pm_g(x,z)=v^\pm_g(x,a^\pm(x)z)$, which are defined in the same domain as  $v_f^\pm$. Then we have 
                \begin{equation}\label{ellenew}
					\begin{aligned}
						&\operatorname{div}_{x,z}(\tilde {\M}^\pm(x,z)\nabla_{x,z} \bar{v}_g^\pm )=0,\ \ \text{in}\ \tilde \Omega_f^\pm,\\
						&\bar{v}_g^--\bar{v}_g^+=\mathcal{G}[g]+\varrho_0g, \quad \text { on } \{z=0\},\\
						& e_{d+1}\cdot ((\mu_+^{-1}\tilde {\M}^+(x,z)\nabla_{x,z}\bar{v}_g^+-\mu_-^{-1}\tilde {\M}^-(x,z)\nabla_{x,z}\bar{v}_g^-))=0, \quad \text { on } \{z=0\},\\
						&\tilde \nu_f^\pm\cdot (\tilde {\M}^\pm(x,z)\nabla_{x,z} \bar{v}_g^\pm) =0,\ \ \ \ \ \ \ \ \ \text{on} \ \tilde \Gamma^\pm_f,
					\end{aligned}
				\end{equation}
					where 
                    \begin{align}\label{S4eq48}
                    \tilde {\M}^\pm(x,z)=a^\pm(x)J^\pm(x,z)\M(\nabla g(x))(J^\pm(x,z))^\top,\ \ \ \quad\quad J^\pm(x,z)=\left(\begin{array}{cc}\mathrm{Id} & 0 \\ -\frac{z}{a^\pm(x)} \nabla a^\pm(x)^{\top} & \frac{1}{a^\pm(x)}\end{array}\right).
                    \end{align}
Note that 
\begin{align}
    \label{difco}
    \|\tilde {\M}^\pm(x,z)-\M(\nabla f(x))\|_{L^\infty(\tilde \Omega_f^\pm)}\lesssim \|f-g\|_{W^{1,\infty}}(1+\|g\|_{W^{1,\infty}})^2\lesssim \|f-g\|_{X_T}(1+\|g\|_{X_T})^2.
\end{align}
Denote $\mathbf{v}^\pm=v^\pm_f-\bar{v}_g^\pm$. Then 
      \begin{equation}\label{ellenew1}
					\begin{aligned}
						&\operatorname{div}_{x,z}({\M}(\nabla f)\nabla_{x,z} \mathbf{v}^\pm )=\operatorname{div}_{x,z}(\mathbf{F}^\pm),\ \ \text{in}\ \tilde \Omega_f^\pm,\\
						&\mathbf{v}^--\mathbf{v}^+=\mathcal{G}[f]-\mathcal{G}[g]+\varrho_0(f-g), \quad \text { on } \{z=0\},\\
						&e_{d+1}\cdot ({\M}(\nabla f)(\mu_+^{-1}\nabla_{x,z}\mathbf{v}^+-\mu_-^{-1}\nabla_{x,z}\mathbf{v}^-))=e_{d+1}\cdot (\mu_+^{-1}\mathbf{F}^+-\mu_-^{-1}\mathbf{F}^-), \quad \text { on } \{z=0\},\\
						&\tilde \nu_f^\pm\cdot ({\M}(\nabla f)\nabla_{x,z} \mathbf{v}^\pm) =\tilde \nu_f^\pm\cdot \mathbf{F}^\pm,\ \ \ \ \ \ \ \ \ \text{on} \ \tilde \Gamma^\pm_f,
					\end{aligned}
				\end{equation}
                where $\mathbf{F}^\pm=(\tilde {\M}^\pm(x,z)-\M(\nabla f))\nabla_{x,z} \bar{v}_g^\pm$.
We have the following estimates of $\mathbf{v}^\pm$.
\begin{lemma}\label{lemdif}
For any $T\in(0,\frac{1}{2})$, and $f,g$ satisfy \eqref{condfff}, it holds
\begin{align*}
 \sup_{t\in[0,T]}\left(  t^\frac{1}{2} \sum_{+,-}\|\nabla_{x,z}\mathbf{v}^\pm\|_{L^2(\tilde \Omega_f^\pm)}+ t^\frac{2}{3}\sum_{+,-}\|\nabla_{x,z}\mathbf{v}^\pm|_{z=0}\|_{L^\infty(\mathbb{R}^d)}\right)\lesssim \|f-g\|_{X_T}(1+\|(f,g)\|_{X_T})^{10}.
\end{align*}
\end{lemma}
\begin{proof}
Repeating the energy estimate used in Lemma \ref{wL2}, we obtain 
\begin{align*}
  \sum_{+,-}\|\nabla_{x,z}\mathbf{v}^\pm\|_{L^2(\tilde \Omega_f^\pm)}&\lesssim \sum_{+,-}\|(\tilde {\M}(x,z)-\M(\nabla f))\nabla_{x,z} \bar{v}_g^\pm\|_{L^2(\tilde \Omega_f^\pm)}+\|\mathcal{G}[f]-\mathcal{G}[g]+\varrho_0(f-g)\|_{H^\frac{1}{2}}\\
  &\lesssim t^{-\frac{1}{2}}\|f-g\|_{X_T}(1+\|(f,g)\|_{X_T})^{10},
\end{align*}
where the last inequality follows from \eqref{vl2}, \eqref{difco} and interpolation. 

To control the $L^\infty$ norm of the trace $\nabla_{x,z}\mathbf{v}^\pm|_{z=0}$, we apply Lemma \ref{leminterf}, which yields
\begin{align*}
\|\nabla_{x,z}\mathbf{v}^\pm\|_{L^\infty(\mathbb{R}^{d+1}_{2r})}&\lesssim \|\mathcal{G}[f]-\mathcal{G}[g]\|_{\dot C^{1,\log^\varkappa}}+\|f-g\|_{\dot C^{1,\log^\varkappa}}+\|\nabla_{x,z}\mathbf{v}^\pm\|_{L^2(\mathbb{R}^{d+1}_{4r})}\\
&\lesssim t^{-\frac{2}{3}}\|f-g\|_{X_T}(1+\|(f,g)\|_{X_T})^{10},
\end{align*}
where $0<r<\frac{\mathbf{r}}{10}$. 
This completes the proof of the lemma.
\end{proof}

We have the following contraction estimate for the operator $\mathcal{Q}[f]$, which is defined in \eqref{defwb}.
						\begin{align*}
							& \sup_{t\in[0,T]}t^\frac{2}{3}\|(\mathcal{Q}[f]-\mathcal{Q}[g])(t)\|_{L^\infty}\lesssim \|f-g\|_{X_T}(1+\|(f,g)\|_{X_T})^{10},\\
							&\sup_{t\in[0,T]} t^\frac{m+\kappa}{3} (|\log t|^\varkappa\|\nabla^m(\mathcal{Q}[f]-\mathcal{Q}[g])(t)\|_{\dot C^{\kappa-2}}+\|\nabla^m(\mathcal{Q}[f]-\mathcal{Q}[g])(t)\|_{ \HH^{\kappa-2}}	)\nonumber\\
                            &\quad\quad\quad\quad\lesssim \|f-g\|_{X_T}( \|( f-\phi,g-\phi)\|_{X_T}+T^\frac{1}{20}\mathfrak{M}_\phi)(1+\| (f,g)\|_{X_T})^{2m+10}.
						\end{align*}
       The $L^\infty$ estimate follows from Lemma  \ref{lemdif} and Lemma \ref{lemtvbhol}, and the H\"{o}lder estimate can be obtained similarly as \eqref{omeba}.   This implies the following estimate of 
       $$\H^{im}[f]-\H^{im}[g]=-\frac{1}{\mu_+}e_{d+1}\cdot \left((\M(\nabla f)-\M(\nabla g))\mathcal{Q}[f]+\M(\nabla g)(\mathcal{Q}[f]-\mathcal{Q}[g])\right).$$
       \begin{lemma}\label{lemH2con}
            	Let $\H^{im}[f]$ be as defined in \eqref{defHgm}, the norm $\|\cdot\|_{X_T}$ be as defined in \eqref{defnorgm} with constants $\varkappa,m,\kappa$ fixed in \eqref{consgm}. Then for any $T\in(0,\frac{1}{2})$, and any $f,g$ satisfy \eqref{condfff}, the following estimates hold.
					\begin{align*}
						&\sup_{t\in[0,T]}|\log t|^\varkappa(t^\frac{\kappa}{3}\|(\H^{im}[f]-\H^{im}[g])(t)\|_{\dot C^{\kappa-2}}+t^\frac{m+\kappa}{3}\|\nabla_x^m(\H^{im}[f]-\H^{im}[g])(t)\|_{\dot C^{\kappa-2}})\nonumber\\
                            &+\sup_{t\in[0,T]}(t^\frac{\kappa}{3}\|(\H^{im}[f]-\H^{im}[g])(t)\|_{\HH^{\kappa-2}}+t^\frac{m+\kappa}{3}\|\nabla_x^m(\H^{im}[f]-\H^{im}[g])(t)\|_{\HH^{\kappa-2}})\nonumber\\
							&\quad\quad\quad\lesssim \|f-g\|_{X_T}(\|(f-\phi,g-\phi)\|_{X_T}+T^\frac{1}{20}\mathfrak{M}_\phi)(1+\|(f,g)\|_{X_T})^{2(m+10)}.
					\end{align*}
                    where $\mathfrak{M}_\phi=\|\phi\|_{H^{m+ 4}\cap C^{m+4}}$.
       \end{lemma}
					\section{Proof of Theorem \ref{thmgm}}\label{secpro}
We first prove a slightly more precise statement, in which the initial interface
is assumed to be close to a smooth reference profile.                    \begin{theorem}\label{thmgm'} Let $\kappa$ be such that $0<3-\kappa\ll 1$, and $m\in\mathbb{N}^+$.
			Take $\eta_0\in H^1(\mathbb{R}^d)\cap\dot C^{1,\log^\varkappa  }(\mathbb{R}^d)$ with $\varkappa>1$ such that \begin{align*}
				\operatorname{dist}(\eta_0,\Gamma^\pm)>2\mathbf{r}>0.
			\end{align*}
			There exists $0<\varepsilon_0\ll 1$ such that if $\eta_0$ further satisfies 
			\begin{align}
				\label{gmcon}
				\|\eta_0-\eta_0\ast \rho_{\eps_1}\|_{H^1\cap \dot C^{1,\log^\varkappa  }}\leq \varepsilon_0,
			\end{align}
			for some $\eps_1>0$, then there exist $T>0$, depending on $\|\eta_0\|_{H^1\cap \dot C^{1,\log^\varkappa  }}, \|\eta_0-\eta_0\ast \rho_{\varepsilon_1}\|_{H^1\cap \dot C^{1,\log^\varkappa  }}, \mathbf{r}, \varepsilon_0,\eps_1$, $m$,  and a unique solution $\eta$ to the system \eqref{sysq}-\eqref{evo} in $[0,T]$ satisfying 
			\begin{align}
				&	\sup_{t\in[0,T]}(\|\eta(t)-\eta_0\ast \rho_{\eps_1}\|_{H^1\cap \dot C^{1,\log^\varkappa  }}+t^\frac{m+\kappa}{3}|\log t|^\varkappa   \|\nabla (\eta(t)-\eta_0\ast \rho_{\eps_1})\|_{\dot C^{m+\kappa}})\leq C\|\eta_0-\eta_0\ast \rho_{\eps_1}\|_{H^1\cap \dot C^{1,\log^\varkappa  }},\label{fes1'}\\
				&\inf_{t\in[0,T]}\operatorname{dist}(\eta(t),\Gamma^\pm)>\mathbf{r}.
			\end{align}
		\end{theorem}
   We begin by deriving Theorem~\ref{thmgm} from Theorem~\ref{thmgm'}.
 Since
\[
\eta_0\in H^1(\mathbb R^d)\cap
\dot C^{1,\log^{2\varkappa-1}}(\mathbb R^d),
\]
and $2\varkappa-1>\varkappa$, we have
\begin{equation}\label{mollconv}
\lim_{\varepsilon\to0}
\|\eta_0-\eta_0*\rho_\varepsilon\|_{H^1\cap\dot C^{1,\log^\varkappa}}
=0.
\end{equation}
Hence, choosing $\varepsilon_1>0$ sufficiently small, condition \eqref{gmcon}
is satisfied. Applying Theorem~\ref{thmgm'}, we obtain a unique solution $\eta$
on some interval $[0,T]$ satisfying \eqref{fes1'} and \eqref{dist*}.

It remains only to recover the estimate \eqref{fes2} in Theorem~\ref{thmgm}. Since
\[
\|\eta_0*\rho_{\varepsilon_1}\|_{H^1\cap\dot C^{1,\log^\varkappa}}
\lesssim
\|\eta_0\|_{H^1\cap\dot C^{1,\log^\varkappa}},
\]
and, for every $0<T'\leq T$,
\begin{align*}
\sup_{t\in[0,T']}
t^{\frac{m+\kappa}{3}}|\log t|^\varkappa
\|\nabla(\eta_0*\rho_{\varepsilon_1})\|_{\dot C^{m+\kappa}}
&\lesssim
T'^{\frac{m+\kappa}{3}}|\log T'|^\varkappa
\varepsilon_1^{-m-1}
\|\eta_0\|_{H^1\cap\dot C^{1,\log^\varkappa}}\\&
\leq \|\eta_0\|_{H^1\cap\dot C^{1,\log^\varkappa}},
\end{align*}
after decreasing $T'$ if necessary. This implies \eqref{fes2} and 
Theorem~\ref{thmgm} follows.\vspace{0.5cm}\\
We now prove Theorem~\ref{thmgm'}. The proof consists of three steps: a priori
estimates, construction by smooth approximation, and stability.\vspace{0.2cm}\\
					{\bf Step 1: A priori Estimates}\\
                   We first take
				\begin{align}\label{defphigm}
					\phi=\eta_0\ast \rho_{\eps_1},
				\end{align}
                and consider the system \eqref{evoeta}.
				By the condition \eqref{gmcon}, we have $\phi\in C^\infty(\mathbb{R}^d)\cap L^2(\mathbb{R}^d)$, and 
				\begin{align}\label{gmcon1}
					\|\eta_0-\phi\|_{H^1\cap \dot C^{1,\log^\varkappa  }}\leq \varepsilon_0.
				\end{align}
					We first establish a priori estimates for the solution to \eqref{evoeta}. Let $T\in(0,\frac{1}{2})$, and  $\eta\in C([0,\bar T];L^2)\cap L^\infty((0,\bar T];W^{m,\infty})$ be a solution to \eqref{evoeta} with initial data $\eta_0$. Suppose \eqref{condfff} holds. 

                    Let $0<\varepsilon_0\ll 1$ be a parameter that will be fixed later in the proof. And denote $\phi=\eta_0\ast \rho_{\varepsilon_1}$ such that \begin{align*}
                 \|\eta_0-\phi\|_{in}:=   \|\eta_0-\phi\|_{H^1\cap \dot C^{1,\log^\varkappa}}\leq \varepsilon_0.
                    \end{align*}

					By adding some terms on both sides of \eqref{evoeta}, and taking one spatial derivative $\partial_i=\partial_{x_i}$, $i=1,2,\cdots,d$, we obtain 
					\begin{equation}\label{infaeq11}
						\begin{aligned}
							&\partial_t \partial_i (\eta-\phi)(x)-\int_{\mathbb{R}^d}\A(\nabla \phi(x),\alpha):{\delta_\alpha \nabla ^2\partial_i}(\eta-\phi)(x)\frac{d\alpha}{|\alpha|^{d+1}}=\partial_i \H[\eta](x)+R_i[\eta](x),\\
							&\partial_i (\eta-\phi)|_{t=0}=\partial_i(\eta_0-\phi),
						\end{aligned}  
					\end{equation}
					where
					\begin{align}\label{defR}
						&R_i[\eta](x)=\int _{\mathbb{R}^d}\partial_{x_i}\left(\A(\nabla \phi(x),\alpha)\right):\delta_\alpha \nabla^2 (\eta-\phi)(x)\frac{d\alpha}{|\alpha|^{d+1}}+\partial_i\left(\int_{\mathbb{R}^d}\A(\nabla \phi(x),\alpha):{\delta_\alpha \nabla ^2}\phi(x)\frac{d\alpha}{|\alpha|^{d+1}}\right).
					\end{align}\vspace{0.1cm}\\
					The term $R[g]$ appears when we differentiate the frozen-coefficient equation
for $g-\phi$. Since the coefficients are frozen at the smooth profile $\phi$,
this term contains only derivatives of $\phi$ and is therefore lower order. The
following lemma shows that its contribution is small on short time intervals.\vspace{0.1cm}\\
					\begin{lemma}\label{lemR}
						Let $\|\cdot\|_{T}$ be as defined in \eqref{defnorgm} with constants $\varkappa,m,\kappa$ fixed in \eqref{consgm}. Then for any $T\in(0,\frac{1}{2})$, there holds
						\begin{align}
							&\int_0^T\|R[g](\tau)\|_{\dot C^{\log^\varkappa}\cap L^2}d\tau+\sup_{t\in[0,T]}t^{\frac{m}{3}+1}\left(|\log t|^\varkappa  \|\nabla^m R[g](t)\|_{L^\infty}+\|\nabla^m R[g](t)\|_{L^2}\right)		\nonumber\\
							&\hspace{5cm}\quad\quad	\lesssim T^\frac{1}{6}(1+\mathfrak{M}_\phi)^{m+3}(\|g-\phi\|_{T}+1),\label{esR}
						\end{align}
                        where $\mathfrak{M}_\phi=\|\phi\|_{H^{m+4}\cap C^{m+4}}$.
					\end{lemma}
					\begin{proof}
						By symmetry,	we have 
						\begin{align*}
							R[g](x)=&	\frac{1}{2}\int_{\mathbb{R}^d} \nabla_x\left( \A(\nabla \phi(x),\alpha)\right):\mathcal{O}_\alpha \nabla^2 (g-\phi)(x)\frac{d\alpha}{|\alpha|^{d}}+\frac{1}{2}\nabla\left(\int_{\mathbb{R}^d}\A(\nabla \phi(x),\alpha):\mathcal{O}_\alpha \nabla ^2\phi(x)\frac{d\alpha}{|\alpha|^{d}}\right).
						\end{align*}
						Then by Lemma \ref{lemoalin}, it is easy to check that 
						\begin{align*}
							\|R[g](t)\|_{\dot C^{\log^\varkappa}}&\lesssim (1+\mathfrak{M}_\phi)^3t^{-\frac{2}{3}}(\|g-\phi\|_{T}+1),\\
							\|\nabla^m R[g](t)\|_{L^\infty\cap L^2}&\lesssim (1+\mathfrak{M}_\phi)^{m+3}|\log t|^{-\varkappa}t^{-\frac{m+2}{3}}(\|g-\phi\|_{T}+1).
						\end{align*}
						Hence, we obtain  \eqref{esR}. 
						The proof is complete. 
					\end{proof}\vspace{0.2cm}\\
					We next estimate the lower-order $L^2$ component in the $X_T$ norm.
					\begin{lemma}\label{mapl2}
						Let  $T\in(0,\frac{1}{2})$,  and $f$ be a solution to the Cauchy problem \eqref{evoeta} satisfying \eqref{condfff}. Then 
						\begin{align}\label{L2}
							\sup_{t\in[0,T]}\|f(t)-\phi\|_{L^2}\leq  \|\eta_0-\phi\|_{L^2}+CT^\frac{1}{10}\|f\|_{X_T}(1+\|f\|_{X_T})^2.
						\end{align}
					\end{lemma}
					\begin{proof}
						We have 
						\begin{align*}
							&\partial_t (f-\phi)(x)= -\frac{1}{\mu_+}e_{d+1}\cdot  (\M(\nabla f)\nabla  _{x,z}v^+_f(x,z)|_{z=0}).
						\end{align*}
						Using $f-\phi$ as a test function to the above equation, we obtain
						\begin{align*}
							\frac{1}{2}\partial_t \big(\|f(t)-\phi\|_{L^2}^2\big)&=-\frac{1}{\mu_+}\int_{\mathbb{R}^d}e_{d+1}\cdot  (\M(\nabla f)\nabla  _{x,z}v^+_f(x,z)|_{z=0})(f-\phi)(x)dx.
						\end{align*}
						By the trace extension theorem, there exists $\bar f:\mathbb{R}^{d+1}_-\to \mathbb{R}$ such that 
						\begin{align*}
							\|\bar f\|_{H^1(\mathbb{R}^{d+1}_+)}\lesssim \|f-\phi\|_{H^\frac{1}{2}(\mathbb{R}^d)}, \ \ \ \ \ \  \bar f(x,0)=f(x)-\phi(x).
						\end{align*}
						Note that $\operatorname{div}_{x,z}  \left(\M(\nabla f)\nabla_{x,z} v^+_f(x,z)\right)=0$. Hence, by H\"{o}lder's inequality and \eqref{vl2},
						\begin{equation}\label{v21}
							\begin{aligned}
								&\left|\int_{\mathbb{R}^d}e_{d+1}\cdot  (\M(\nabla f)\nabla  _{x,z}v^+_f(x,z)|_{z=0})(f-\phi)(x)dx\right|\\&=\frac{1}{\mu_+}\left|\int_{\tilde \Omega_f^+}\operatorname{div}_{x,z}  \left(\M(\nabla f)\nabla  _{x,z}v^+_f(x,z)\bar{f}(x,z)\right)dxdz\right| \\
								&\lesssim \int_{\tilde \Omega_f^+}\left|(\nabla _{x,z}\bar{f})^\top  \M(\nabla f)\nabla  _{x,z}v^+_f(x,z)\right|dxdz\lesssim (1+\|\nabla f\|_{L^\infty})^2\|\nabla_{x,z}\bar f\|_{L^2(\mathbb{R}^{d+1}_+)}\|\nabla_{x,z}v_f^+\|_{L^2(\mathbb{R}^{d+1}_+)}\\
								&\lesssim (1+\|\nabla f\|_{L^\infty})^2\|f-\phi\|_{H^\frac{1}{2}(\mathbb{R}^d)}\|\nabla_{x,z}v_f^+\|_{L^2(\mathbb{R}^{d+1}_+)}\lesssim  t^{-\frac{2}{3}}(1+\|f\|_{X_T})^2\|f-\phi\|_{X_T}\|f\|_{X_T}.
							\end{aligned}
						\end{equation}
						Integrating  in time yields \eqref{L2}.
						This completes the proof of the lemma.
					\end{proof}
					
					\vspace{0.2cm}
					We now combine the estimates above to obtain the a priori bound for
$\|\eta-\phi\|_{X_T}$. The Schauder estimate controls the logarithmic Hölder part
of the norm, while Lemma~\ref{mapl2} gives the lower-order $L^2$ bound.  Applying Lemma \ref{lemlog}, we obtain that there exists $T=T(\mathfrak{M}_\phi)>0$ such that
					\begin{equation}\label{mapes}
						\begin{aligned}
							\|\eta-\phi\|_T\lesssim &\|\eta_0-\phi\|_{\dot  C^{1,\log^\varkappa}  \cap H^1}+\da_T(\H[\eta],R[\eta])+\da_T^2(\H[\eta],R[\eta]),
						\end{aligned}
					\end{equation}
                    where $\da_T,\da_T^p$ are defined in \eqref{defda1} and \eqref{defda2}.
					By Lemma \ref{lemH1}, Lemma \ref{lemh2} and Lemma \ref{lemR}, we have 
					\begin{equation}\label{H11}
						\begin{aligned}
						\da_T(\H[\eta],R[\eta])
							\lesssim &(\|\eta-\phi\|_{X_T}+T^\frac{1}{20}\mathfrak{M}_\phi)^2(1+\|(\eta,\phi)\|_{T})^{2(m+10)}\\
                            &\quad+T^\frac{1}{6}(1+\mathfrak{M}_\phi)^{m+3}(\|\eta-\phi\|_{X_T}+1).
						\end{aligned}
					\end{equation}
			Hence, we obtain 
					\begin{equation}\label{gpht}
						\begin{aligned}
							\|\eta-\phi\|_T
							&\leq C_1\|\eta_0-\phi\|_{\dot  C^{1,\log^\varkappa}  \cap H^1}+C_1(\|\eta-\phi\|_{X_T}+T^\frac{1}{20}\mathfrak{M}_\phi)^2(1+\|(\eta,\phi)\|_{T})^{2(m+10)}\\
                            &\quad+C_1T^\frac{1}{6}(1+\mathfrak{M}_\phi)^{m+3}(\|\eta-\phi\|_{X_T}+1).
						\end{aligned}
					\end{equation}
			On the other hand, the $L^2$ norm is bounded by  Lemma \ref{mapl2}:
					\begin{align}\label{gphl2}
						\|\eta-\phi\|_{L^\infty_TL^2}\leq  \|\eta_0-\phi\|_{L^2}+C_1T^\frac{1}{10}(1+\|\eta\|_{X_T})^2.
					\end{align}
					Then we obtain from \eqref{gpht} and \eqref{gphl2} that 
					\begin{equation}\label{ineeeq1}
						\begin{aligned}
							\|\eta-\phi\|_{X_T}&\leq C_1\|\eta_0-\phi\|_{\dot  C^{1,\log^\varkappa}  \cap H^1}+C_1(\|\eta-\phi\|_{X_T}+T^\frac{1}{20}\mathfrak{M}_\phi)^2(1+\|(\eta,\phi)\|_{T})^{2(m+10)}\\
                            &\quad+C_1T^\frac{1}{6}(1+\mathfrak{M}_\phi+\|\eta-\phi\|_{X_T})^{m+5}.
						\end{aligned}
					\end{equation}
                    We also record the propagation of the separation condition. Since the interface
moves only by the amount $|\eta(t)-\eta_0|$, and we have
					\begin{align*}
						\|\eta(t)-\eta_0\|_{L^\infty_TL^\infty}\leq \|\eta_0-\phi\|_{L^\infty}+\|\eta-\phi\|_{L^\infty_TL^\infty}\leq \|\eta_0-\phi\|_{H^1\cap \dot C^{1,\log^\varkappa  }}+\|\eta-\phi\|_{X_T}.
					\end{align*}
					Then we obtain 
				\begin{equation}\label{dist2}
				\begin{aligned}
				\inf_{t\in[0,T]}\operatorname{dist}(\eta(t),\Gamma^\pm)&\geq 	\operatorname{dist}(\eta_0,\Gamma^\pm)-\|\eta-\eta_0\|_{L^\infty_TL^\infty}\\
                &\geq 2\mathbf{r}-\|\eta_0-\phi\|_{H^1\cap \dot C^{1,\log^\varkappa  }}-\|\eta-\phi\|_{X_T}.
				\end{aligned}
				\end{equation}
					{\bf  Step 2: Existence}\\
                    We construct solutions by smooth approximation and use the a priori estimates
from Step 1 to obtain a lifespan independent of the smoothing parameter.
					For $0<\vartheta\ll 1$, denote 
					\begin{align*}
						\eta_{0\vartheta}=\eta_0\ast \rho_\vartheta.
					\end{align*}
					Applying the well-posedness theory in \cite{1HuyNguyen2020}, there exists $T_\vartheta>0$ and a unique solution $\eta_\vartheta\in C([0,T_\vartheta];H^{d+1}(\mathbb{R}^d))$ to the system \eqref{sysq}-\eqref{evo}. Without loss of generality, let 
                    \begin{align}\label{defT}
                    T_\vartheta=\sup\{T>0:\|\eta_\vartheta-\phi\|_{X_T}\leq 10 C_1 \|\eta_0-\phi\|_{in},\ \inf_{t\in[0,T]}\operatorname{dist}(\eta_\vartheta(t),\Gamma^\pm)\geq  \mathbf{r}\}.
                    \end{align}
                    Note that we can take $\vartheta$ small enough such that $\|\eta_{0\vartheta}-\phi\|_{H^1\cap \dot C^{1,\log^\varkappa}}\leq 2  \|\eta_0-\phi\|_{in}\leq 2\varepsilon_0$.
                    We claim that $T_{\vartheta}$ has a positive lower bound that is independent of $\vartheta$:
                    \begin{align}\label{Tlb}
                    T_\vartheta\geq T':=\left(\frac{\|\eta_0-\phi\|_{in}}{100+C_1+\mathfrak{M}_\phi}\right)^{100m}.
                    \end{align}
                    We prove \eqref{Tlb} by contradiction. If $ T_\vartheta<T'$, by \eqref{ineeeq1} and \eqref{dist2},
                    we know that 
                    \begin{align*}
                    \|\eta_\vartheta-\phi\|_{X_{T_\vartheta}}\leq &C_1  \|\eta_0-\phi\|_{in}+C_1(\|\eta_\vartheta-\phi\|_{X_{T_\vartheta}}+{T_\vartheta}^\frac{1}{20}\mathfrak{M}_\phi)^2(1+\|(\eta_\vartheta,\phi)\|_{T_\vartheta})^{2(m+10)}\\
                            &\quad+C_1{T_\vartheta}^\frac{1}{6}(1+\mathfrak{M}_\phi+\|\eta_\vartheta-\phi\|_{X_{T_\vartheta}})^{m+5}\\
                            \leq & 2C_1  \|\eta_0-\phi\|_{in}+C_1(10C_1  \|\eta_0-\phi\|_{in}+{T_\vartheta}^\frac{1}{20}\mathfrak{M}_\phi)^2(1+10C_1  \|\eta_0-\phi\|_{in}+  {T_\vartheta}^\frac{1}{20}\mathfrak{M}_\phi),
                    \end{align*}
                    and 
                    	\begin{equation*}
				\begin{aligned}
				\inf_{t\in[0,T_\vartheta]}\operatorname{dist}(\eta_\vartheta(t),\Gamma^\pm)
                &\geq 2\mathbf{r}-\|\eta_0-\phi\|_{in}-\|\eta_\vartheta-\phi\|_{X_{T_\vartheta}}.
				\end{aligned}
				\end{equation*}
                    By taking $\|\eta_0-\phi\|_{in}\leq \varepsilon_0:=(100+C_1+\mathbf{r})^{-100}$, we obtain 
                    \begin{align*}
                    &   C_1(10C_1  \|\eta_0-\phi\|_{in}+{T_\vartheta}^\frac{1}{20}\mathfrak{M}_\phi)^2(1+10C_1  \|\eta_0-\phi\|_{in}+  {T_\vartheta}^\frac{1}{20}\mathfrak{M}_\phi)\leq 3\|\eta_0-\phi\|_{in},\\
                 &   \|\eta_0-\phi\|_{in}+\|\eta-\phi\|_{X_{T_\vartheta}}\leq \frac{\mathbf{r}}{2}.
                    \end{align*}
              This implies 
               \begin{align*}
                    \|\eta_\vartheta-\phi\|_{X_{T_\vartheta}}\leq 5C_1 \|\eta_0-\phi\|_{in},\quad\quad\quad \inf_{t\in[0,T_\vartheta]}\operatorname{dist}(\eta(t),\Gamma^\pm)\geq \frac{3\mathbf{r}}{2}.
                    \end{align*}
                    This contradicts the definition of $T_\vartheta$ in \eqref{defT}. Hence, we obtain \eqref{Tlb}.

                    On the other hand, by the equation \eqref{evoeta}, we derive that 
                    \begin{align*}
                   \sup_{t\in[0,T_\vartheta]}t^\frac{2}{3} \|\partial_t \eta_\vartheta\|_{L^\infty}\lesssim \|\eta_\vartheta\|_{X_{T_\vartheta}}(1+\|\eta_\vartheta\|_{X_{T_\vartheta}})^{10}<\infty.
                    \end{align*}
                    Hence by Arzel\`a-Ascoli lemma,  $\{\eta_\vartheta\}$  has a subsequence that converges to a function $\eta$ on $[0,T']$, which is a solution to \eqref{evoeta} with initial data $\eta_0$, and satisfies the  estimate
                   \begin{align}
                    \|\eta-\phi\|_{X_T}\leq 10 C_1  \|\eta_0-\phi\|_{in},\ \quad\quad \operatorname{dist}(\eta(t),\Gamma^\pm)\geq  \mathbf{r}.
                \end{align}
                    \vspace{0.3cm}\\
                 		{\bf  Step 3: Stability}\\
                        We finally prove stability with respect to the initial data, which in particular
implies uniqueness.
					Let $f,g$ be the solutions to the Cauchy problem \eqref{evoeta} associated with initial data $f_0,g_0$, respectively. Then 
                    \begin{align*}
                   & \partial_t (f-g)(x)-\int_{\mathbb{R}^d}\A(\nabla\phi(x),\alpha):\delta_\alpha\nabla^2(f-g)(x)\frac{d\alpha}{|\alpha|^{d+1}}=(\H[f]-\H[g])(x),\\
                    &(f-g)|_{t=0}=f_0-g_0.
                    \end{align*}
                    Applying Lemma \ref{lemlog} to the equation of $\nabla(f-g)$ yields
                    \begin{align*}
                    \|f-g\|_T\lesssim \|f_0-g_0\|_{H^1\cap \dot C^{1,\log^\varkappa}}+\da_T(\H[f]-\H[g],\tilde R[f,g]),
                    \end{align*}
                    where 
                    \begin{align*}
                    \tilde R[f,g]=\int_{\mathbb{R}^d}\nabla (\A(\nabla\phi(x),\alpha)):\delta_\alpha \nabla^2(f-g)(x)\frac{d\alpha}{|\alpha|^{d+1}}.
                    \end{align*}
                    From Lemma \ref{lemR}, we know that 
                    \begin{align*}
                    \da_T(0,\tilde R[f,g])\lesssim T^\frac{1}{6} (1+\mathfrak{M}_\phi)\|f-g\|_{X_T}.
                    \end{align*}
					By Lemma \ref{lemH1con} and Lemma \ref{lemH2con}, we deduce that 
                    \begin{align*}
                    \da_T(\H[f]-\H[g],0)\lesssim \|f-g\|_{X_T}(\|(f-\phi,g-\phi)\|_{X_T}+T^\frac{1}{20}\mathfrak{M}_\phi)(1+\|(f,g)\|_{X_T})^{2(m+10)}.
                    \end{align*}
                    Finally, we estimate the $L^2 $ norm of $f-g$. 
                    Note that 
                    \begin{align*}
                    \partial_t (f-g)=-\frac{1}{\mu_+}e_{d+1}\cdot ((\M(\nabla f)-\tilde {\M}(x,z))\nabla_{x,z}\bar v_g^+|_{z=0})-\frac{1}{\mu_+}e_{d+1}\cdot (\M(\nabla f)\nabla_{x,z}\mathbf{v}^+|_{z=0}),
                    \end{align*}
                    where $\tilde M$, $\bar v_g^\pm$ and $\mathbf{v}^\pm$ are defined in \eqref{S4eq48}, \eqref{ellenew} and \eqref{ellenew1} respectively.
                    Similar to Lemma \ref{mapl2}, we obtain 
                    \begin{align*}
                    \|(f-g)(t)\|_{L^2}^2\lesssim  \|f_0-g_0\|_{L^2}^2+\int_0^t \|f-g\|_{H^\frac{1}{2}}(\|\nabla (f-g)\|_{L^\infty}\|\nabla_{x,z}\bar v^-_g\|_{L^2(\Omega_f^-)}+\|\nabla_{x,z}\mathbf{v}^-\|_{L^2(\Omega_f^-)})d\tau.
                    \end{align*}
                    Combining this with Lemma \ref{wL2} and  Lemma \ref{lemdif} yields
                      \begin{align*}
                    \|(f-g)(t)\|_{L^2}^2\lesssim  \|f_0-g_0\|_{L^2}^2+t^\frac{1}{3}\|f-g\|_{X_T}^2(1+\|(f,g)\|_{X_T})^{10}.
                    \end{align*}
				Hence, we obtain 
					\begin{align*}
						\|f-g\|_{X_T}&\leq \|f-g\|_{T}+\sup_{t\in[0,T]}\|(f-g)(t)\|_{L^2}\\
                        &\leq C_2 \|f_0-g_0\|_{H^1\cap \dot C^{1,\log^\varkappa}}+C_2\|f-g\|_{X_T}(\|(f-\phi,g-\phi)\|_{X_T}+T^\frac{1}{20}\mathfrak{M}_\phi)(1+\|(f,g)\|_{X_T})^{2(m+10)}\\
                        &\leq 2C_2 \|f_0-g_0\|_{H^1\cap \dot C^{1,\log^\varkappa}},
					\end{align*}
                    provided 
                    \begin{align*}
                    (\|(f-\phi,g-\phi)\|_{X_T}+T^\frac{1}{20}\mathfrak{M}_\phi)(1+\|(f,g)\|_{X_T})^{2(m+10)}\leq (100C_2)^{-1}.
                    \end{align*}
This completes the proof of Theorem~\ref{thmgm'}.					
					\section{Appendix}
                    In this appendix we prove the endpoint elliptic estimates used in the paper. We consider second-order divergence-form elliptic equations with coefficients in the log-Dini class $C^{\log^\varkappa}$, $\varkappa>1$. The proof is based on a freezing-coefficient argument together with elementary estimates for the frozen half-space kernels. Related endpoint estimates for elliptic equations can be found in \cite{Don}. \vspace{0.1cm}\\
					We start from the  Laplace's equation, for which the fundamental solution satisfies 
					\begin{align*}
						-\Delta \Phi(x)=\delta(x)~~\text{in}~~\mathbb{R}^d. 
					\end{align*}
                    where $\delta_0$ is the Dirac mass at the origin.\\
					The fundamental solution has an explicit formula
					\begin{equation}\label{defPhi}
						\Phi(x):=\begin{cases}
							&-\frac{1}{2\pi}\log|x|,\ \ \ d=2\\
							&\frac{1}{d(d-2)V(d)}\frac{1}{|x|^{d-2}},\ \ \ d\geq 3.
						\end{cases}\ \ \ \forall \ x\in\mathbb{R}^d\backslash\{0\}.
					\end{equation}
					Here $V(d)$ is the volume of the unit ball in $\mathbb{R}^d$. \vspace{0.2cm}\\
					We will sometimes slightly abuse notation and write $\Phi(x)=\Phi(|x|)$ to emphasize that the fundamental solution is radial.  We have the estimates
					\begin{align}\label{esfdphi}
						&	|\nabla^k \Phi(x)|\lesssim \frac{1}{|x|^{d-2+k}},\ \ \ \forall k\in\mathbb{N}^+,\\
						&	|\delta_\alpha \nabla^k \Phi(x)|\lesssim \left(\frac{1}{|x|^{d-2+k}}+\frac{1}{|x-\alpha|^{d-2+k}}\right)\min\left\{1,\frac{|\alpha|}{|x|}\right\},\quad\quad\quad \forall \alpha\in\mathbb{R}^d.\label{delphi}
					\end{align}
					Moreover, the fundamental solution has the following  cancellation property 
					\begin{align}\label{S6eq1}
						\int_{|x|=r}\partial_i\partial_j\Phi(x)dS(x)=0,\ \ \ \forall \ r>0,\ \ \forall \ i,j=1,\cdots,d.
					\end{align}
					For the Poisson equation 
					\begin{align}\label{poisson}
						\Delta u=\operatorname{div}f.
					\end{align}
					The solution reads $u(x)=-\int_{\mathbb{R}^d}\Phi(x-y)\operatorname{div}f(y)dy$. 
					By the classical Schauder theory, for $0<\alpha<1$,
					\begin{align*}
						\|\nabla u\|_{\dot C^\alpha}\lesssim\frac{1}{\alpha(1-\alpha)}\|f\|_{\dot C^\alpha}.
					\end{align*} 
					However, the $L^\infty\to L^\infty$ estimate does not hold. To control the $L^\infty$ norm of $\nabla u$, we need the $\dot C^{\log^\varkappa  }$ norm of $f$. In fact, by \eqref{S6eq1} and integrating by parts, we can obtain 
					\begin{align*}
						\nabla u(x)&=\int_{\mathbb{R}^d}\nabla ^2\Phi(x-y)\cdot f(y)dy\\
						&=\int_{|x-y|\leq 1}\nabla ^2\Phi(x-y)\cdot(f(y)-f(x))dy+\int_{|x-y|\geq 1}\nabla ^2\Phi(x-y)\cdot f(y)dy\\
						:&=I_1+I_2.
					\end{align*} 
					It is easy to check that 
					\begin{align*}
						|I_1|\lesssim \int_{|x|\leq 1} \frac{1}{|x|^d\log^\varkappa  (|x|^{-1}+10)}dx\|f\|_{\dot C^{\log ^\varkappa}}\lesssim \|f\|_{\dot C^{\log ^\varkappa}},
					\end{align*}
					provided $\varkappa>1$. For lower frequency, we have 
					\begin{align*}
						|I_2|\lesssim \int_{|x-y|\geq 1}\frac{1}{|x-y|^d}|f(y)|dy\lesssim \|f\|_{L^2}.
					\end{align*}
					Hence we obtain the Lipschitz estimate to the Poisson equation \eqref{poisson}
					\begin{align}\label{gaes}
						\|\nabla u\|_{L^\infty}\lesssim  \|f\|_{\dot C^{\log ^\varkappa}}+\|f\|_{L^2}.
					\end{align}
					We remark that,  when considering elliptic equation with variable coefficients, the classical method to estimate equations with non-constant coefficient is to combine the standard freezing coefficient method together with  the estimate of constant coefficient equations. But this fails when we estimate the endpoint $L^\infty$ norm. Because the norm $\|\cdot\|_{\dot C^{\log^\varkappa  }}$ on the right hand side and the $\|\cdot\|_{L^\infty}$ norm on the left hand side of \eqref{gaes} do not match. Hence the remainder term for freezing coefficient cannot be absorbed by the main term on the left hand side. To obtain the $L^\infty$ estimate of solution to \eqref{poisson}, we need to find a suitable way to freeze the coefficient and estimate carefully. 
					\vspace{0.1cm}\\
                    We next derive a representation formula for solutions to an elliptic
transmission problem with variable coefficients. For clarity, we first introduce
the notation
					\begin{align*}
                &\mathbb{R}^{d+1}_\tau=\mathbb{R}^{d}\times (0,\tau)\ \text{if}\ \tau>0,\ \ \mathbb{R}^{d+1}_\tau=\mathbb{R}^{d}\times (\tau,0)\ \text{if}\ \tau<0.
					\end{align*}
					Consider the system
					\begin{equation}\label{ellinte}
						\begin{aligned}
							&\operatorname{div}_{x,z}(A(x)\nabla_{x,z} u(x,z))=\operatorname{div}_{x,z}g(x,z),\ \ \text{in}\ \mathbb{R}^{d+1}_{2r}\cup \mathbb{R}^{d+1}_{-2r},\\
							&u(x,0^-)-u(x,0^+)=h(x),\ \ \text{on} \ \{z=0\},\\
							&e_{d+1}\cdot(\mu_+^{-1}(A(x)\nabla_{x,z} u(x,0^+)-g(x,0^+))-\mu_-^{-1}(A(x)\nabla_{x,z} u(x,0^-)-g(x,0^-)))=0,\ \ \text{on} \ \{z=0\}.
						\end{aligned}
					\end{equation}
					Suppose the variable coefficient $A(x)\in  C^{\log^\varkappa}(\mathbb{R}^d,\mathbb{R}^{(d+1)\times(d+1)}_{sym})$ is uniformly elliptic:
					$$
					c_0	\mathrm{Id}\leq A(x)\leq c_0^{-1} \mathrm{Id},\ \ \forall x\in\mathbb{R}^d.
					$$
Fix \(x_0\in \mathbb R^d\). Choose an invertible matrix
\(Q_{x_0}\in \mathbb R^{(d+1)\times (d+1)}\) such that
\begin{align}\label{defQx0}
Q_{x_0}Q_{x_0}^{\top}=A(x_0),
\qquad
(Q_{x_0})_{d+1,i}=0,\quad i=1,\dots,d.
\end{align}
Such a matrix $Q_{x_0}$ can be chosen as follows. Starting from any square root
of $A(x_0)$, we multiply it on the left by a suitable orthogonal matrix so that
the last row is parallel to $e_{d+1}$. Set
\[
R:=\operatorname{diag}(I_d,-1).
\]
For \(X=(x,z)\in \mathbb R^d\times \mathbb R\) and
\(Y=(y,w)\in \mathbb R^d\times \mathbb R\), define the reflected point
\[
Y^{*,x_0}:=Q_{x_0}R Q_{x_0}^{-1}Y.
\]
Equivalently, if
\[
Q_{x_0}=
\begin{pmatrix}
	B_{x_0} & b_{x_0}\\
	0 & q_{x_0}
\end{pmatrix},
\]
then
\begin{align}\label{refle}
    Y^{*,x_0}
=
\left(y-2\frac{b_{x_0}}{q_{x_0}}w,\,-w\right).
\end{align}
We define the frozen Dirichlet and Neumann Green functions associated with
the constant matrix $A(x_0)$ by
\begin{equation}
\label{defGtiG}
\begin{aligned}
	\G_{x_0}^D(X,Y)
	&:=
	|\det Q_{x_0}|^{-1}
	\left[
	\Phi \left(Q_{x_0}^{-1}(X-Y)\right)
	-
	\Phi\left(Q_{x_0}^{-1}(X-Y^{*,x_0})\right)
	\right],
	\\
	\G_{x_0}^N(X,Y)
	&:=
	|\det Q_{x_0}|^{-1}
	\left[
	\Phi\left(Q_{x_0}^{-1}(X-Y)\right)
	+
	\Phi\left(Q_{x_0}^{-1}(X-Y^{*,x_0})\right)
	\right].
\end{aligned}
\end{equation}
Since $Q_{x_0}^{-1}Y^{*,x_0}=RQ_{x_0}^{-1}Y$, these kernels can equivalently be written as
\begin{align*}
	\G_{x_0}^D(X,Y)
	&=
	|\det Q_{x_0}|^{-1}
	\left[
	\Phi\left(Q_{x_0}^{-1}X-Q_{x_0}^{-1}Y\right)
	-
	\Phi\left(Q_{x_0}^{-1}X-RQ_{x_0}^{-1}Y\right)
	\right],
	\\
	\G_{x_0}^N(X,Y)
	&=
	|\det Q_{x_0}|^{-1}
	\left[
	\Phi\left(Q_{x_0}^{-1}X-Q_{x_0}^{-1}Y\right)
	+
	\Phi\left(Q_{x_0}^{-1}X-RQ_{x_0}^{-1}Y\right)
	\right].
\end{align*}
By construction, these kernels satisfy the boundary conditions on $\{w=0\}$:
\begin{equation}\label{Gbdy}
\begin{aligned}
&\G_{x_0}^D(X,Y)\big|_{w=0}=0,\\
&e_{d+1}\cdot A(x_0)\nabla_Y \G_{x_0}^N(X,Y)\big|_{w=0}=0.
\end{aligned}
\end{equation}
Moreover, we have 
\begin{align}
\label{eqtilG}
\operatorname{div}_{y,w}\bigl(A(x_0)\nabla_{y,w} \G_{x_0}^D(x,z,y,w)\bigr)
=
-\delta_{(x,z)}(y,w)
+
\delta_{(x,z)^{*,x_0}}(y,w),
\end{align}
and
\begin{align*}
    \operatorname{div}_{y,w}\bigl(A(x_0)\nabla_{y,w} \G_{x_0}^N(x,z,y,w)\bigr)
=
-\delta_{(x,z)}(y,w)
-
\delta_{(x,z)^{*,x_0}}(y,w),
\end{align*}
where the identities are understood in the sense of distributions.\vspace{0.2cm}\\

			We now state the representation formula obtained by freezing the coefficient
matrix at a fixed tangential point $x_0$. The formula separates three
contributions: the bulk forcing, the jump data on the interface, and an
artificial-boundary term coming from the localization in the normal variable.
The latter is lower order and is controlled by energy norms.
				\begin{lemma}\label{leminterf}
	Fix \(x_0\in\mathbb R^d\).
	Let \(\G_{x_0}^D\) and \(\G_{x_0}^N\) be the frozen Dirichlet and
	Neumann kernels defined in \eqref{defGtiG}, and let \(X^{*,x_0}\) be the corresponding
	reflected point.
	Assume that \(u,g,h\) are smooth and decay sufficiently fast in the
	tangential variables, and that \(u\) solves \eqref{ellinte}. Then
	for every \(X=(x,z)\in\mathbb{R}_r^{d+1}\cup\mathbb{R}_{-r}^{d+1}\), one has
	\begin{equation}\label{foruintf-correct-full}
		\begin{aligned}
			u(X)
			=&
			\frac{\mu_+\mathbf 1_{\{z>0\}}+\mu_-\mathbf 1_{\{z<0\}}}
			{\mu_++\mu_-}
			\left[
			\frac1r\int_r^{2r}
			\int_{\mathbb{R}^{d+1}_\tau\cup\mathbb{R}^{d+1}_{-\tau}}
			\mathcal H_{x_0}(Y)\cdot
			\nabla_Y\G^D_{x_0}(X,Y)\,dY\,d\tau
			\right.
			\\
			&\hspace{5.5cm}
			\left.
			-
			\int_{\mathbb R^d}
			e_{d+1}\cdot
			A(x_0)\nabla_Y\G^D_{x_0}(X,(y,0))h(y)\,dy
			\right]
			\\
			&+
			\frac{\mu_+\mu_-}{\mu_++\mu_-}
			\sum_{\pm}
			\frac{1}{r\mu_\pm}
			\int_r^{2r}
			\int_{\mathbb{R}_{\pm\tau}^{d+1}}
			\mathcal H_{x_0}(Y)\cdot
			\nabla_Y\G^N_{x_0}(X,Y)\,dY\,d\tau
			\\
			&+\mathrm{LO}(u,g)(X),
		\end{aligned}
	\end{equation}
	where
	\[
	\mathcal H_{x_0}(Y):=
	g(Y)+\bigl(A(x_0)-A(y)\bigr)\nabla_Yu(Y).
	\]
	The term \(\mathrm{LO}(u,g)\) is the averaged artificial-boundary
	contribution which satisfies
	\begin{equation}\label{LO-estimate-interf}
		\left\|
		\nabla_X^n\mathrm{LO}(u,g)
		\right\|_{L^\infty(\mathbb{R}_{r/2}^{d+1}\cup\mathbb{R}_{-r/2}^{d+1})}
		\le
		C(r,n,c_0,d)
		\left(
		\|\nabla u\|_{L^2(\mathbb{R}_{2r}^{d+1}\cup\mathbb{R}_{-2r}^{d+1})}
		+
		\|g\|_{L^2(\mathbb{R}_{2r}^{d+1}\cup\mathbb{R}_{-2r}^{d+1})}
		\right),\quad\quad \forall \ n\in\mathbb N_+.
	\end{equation}
\end{lemma}
\begin{proof}
	We write $
	A_0:=A(x_0)$ and 
	\[
	\mathcal E_{x_0}(Y):=\bigl(A_0-A(y)\bigr)\nabla_Yu(Y).
	\]
	Then
	\[
	A(y)\nabla_Yu
	=
	A_0\nabla_Yu-\mathcal E_{x_0},
	\qquad
	\mathcal H_{x_0}
	=
	g+\mathcal E_{x_0}.
	\]
	All integrations by parts below are first justified with a tangential
	cutoff and by excising small balls around the poles of the frozen
	kernels. The cutoff is then removed and the radii of the excised balls are
	sent to zero. We suppress these standard limiting steps in the notation.
	
	Fix \(\tau\in(r,2r)\). We first record the artificial-boundary terms.
	 For
	\(\mathcal S\in\{D,N\}\), define
	\begin{equation}\label{def-LO-tau-pm}
		\begin{aligned}
			\mathrm{LO}_{\tau,\pm}^{\mathcal S}(X)
			:=
			&\pm\int_{\mathbb R^d}
		e_{d+1}\cdot
			\bigl(A(y)\nabla_Yu-g\bigr)(y,\pm\tau)\,
			\G_{x_0}^{\mathcal S}(X,(y,\pm\tau))\,dy
			\\
			&\mp
			\int_{\mathbb R^d}
			u(y,\pm\tau)\,
			e_{d+1}\cdot A_0\nabla_Y
			\G_{x_0}^{\mathcal S}(X,(y,\pm\tau))\,dy .
		\end{aligned}
	\end{equation}
		Using the Neumann kernel in \(\mathbb{R}^{d+1}_\tau\), we obtain
	\begin{equation}\label{N-plus}
		\begin{aligned}
			&\mathbf 1_{\mathbb{R}^{d+1}_\tau}(X)u(X)
			+
			\mathbf 1_{\mathbb{R}^{d+1}_\tau}(X^{*,x_0})u(X^{*,x_0})
			\\
			&\quad =
			\int_{\mathbb{R}^{d+1}_\tau}
			\mathcal H_{x_0}(Y)\cdot\nabla_Y\G^N_{x_0}(X,Y)\,dY
			-
			\int_{\mathbb R^d}
			e_{d+1}\cdot(A\nabla_Y u-g)(y,0^+)
			\G^N_{x_0}(X,(y,0))\,dy
			+
			\mathrm{LO}_{\tau,+}^N(X).
		\end{aligned}
	\end{equation}
	The minus sign in the interface term comes from the fact that the outward
	normal of \(\mathbb{R}^{d+1}_\tau\) on \(\{w=0\}\) is \(-e_{d+1}\).
	
	Similarly, using the Neumann kernel in \(\mathbb{R}^{d+1}_{-\tau}\), whose outward
	normal on \(\{w=0\}\) is \(e_{d+1}\), we get
	\begin{equation}\label{N-minus}
		\begin{aligned}
			&\mathbf 1_{\mathbb{R}^{d+1}_{-\tau}}(X)u(X)
			+
			\mathbf 1_{\mathbb{R}^{d+1}_{-\tau}}(X^{*,x_0})u(X^{*,x_0})
			\\
			&\quad =
			\int_{\mathbb{R}^{d+1}_{-\tau}}
			\mathcal H_{x_0}(Y)\cdot\nabla_Y\G^N_{x_0}(X,Y)\,dY
			+
			\int_{\mathbb R^d}
			e_{d+1}\cdot(A\nabla_Y u-g)(y,0^-)
			\G^N_{x_0}(X,(y,0))\,dy
			+
			\mathrm{LO}_{\tau,-}^N(X).
		\end{aligned}
	\end{equation}
	Multiplying \eqref{N-plus} by \(\mu_+^{-1}\), multiplying
	\eqref{N-minus} by \(\mu_-^{-1}\), and adding the two formulas, the
	interface flux terms cancel by the transmission condition in
	\eqref{ellinte}. Therefore
	\begin{equation}\label{B1-eq}
		\begin{aligned}
			&\mu_+^{-1}
			\left[
			\mathbf 1_{\mathbb{R}^{d+1}_\tau}(X)u(X)
			+
			\mathbf 1_{\mathbb{R}^{d+1}_\tau}(X^{*,x_0})u(X^{*,x_0})
			\right]
			\\
			&\quad+
			\mu_-^{-1}
			\left[
			\mathbf 1_{\mathbb{R}^{d+1}_{-\tau}}(X)u(X)
			+
			\mathbf 1_{\mathbb{R}^{d+1}_{-\tau}}(X^{*,x_0})u(X^{*,x_0})
			\right]
			\\
			&=
			\sum_{\pm}\mu_\pm^{-1}
			\int_{\mathbb{R}_{\pm\tau}^{d+1}}
			\mathcal H_{x_0}(Y)\cdot\nabla_Y\G^N_{x_0}(X,Y)\,dY
			+
			\mathrm{LO}_\tau^N(X),
		\end{aligned}
	\end{equation}
	where
	\begin{equation*}
		\mathrm{LO}_\tau^N(X)
		:=
		\mu_+^{-1}\mathrm{LO}_{\tau,+}^N(X)
		+
		\mu_-^{-1}\mathrm{LO}_{\tau,-}^N(X).
	\end{equation*}
	
	Next we use the Dirichlet kernel. Since
	\[
	\G^D_{x_0}(X,(y,0))=0,
	\]
	the interface terms containing the conormal flux vanish. In
	\(\mathbb{R}^{d+1}_\tau\), integration by parts gives
	\begin{equation}\label{D-plus}
		\begin{aligned}
			&\mathbf 1_{\mathbb{R}^{d+1}_\tau}(X)u(X)
			-
			\mathbf 1_{\mathbb{R}^{d+1}_\tau}(X^{*,x_0})u(X^{*,x_0})
			\\
			&\quad =
			\int_{\mathbb{R}^{d+1}_\tau}
			\mathcal H_{x_0}(Y)\cdot\nabla_Y\G^D_{x_0}(X,Y)\,dY
			+
			\int_{\mathbb R^d}
			u(y,0^+)
			e_{d+1}\cdot A_0\nabla_Y\G^D_{x_0}(X,(y,0))\,dy
			+
			\mathrm{LO}_{\tau,+}^D(X).
		\end{aligned}
	\end{equation}
	Similarly, in \(\mathbb{R}^{d+1}_{-\tau}\),
	\begin{equation}\label{D-minus}
		\begin{aligned}
			&\mathbf 1_{\mathbb{R}^{d+1}_{-\tau}}(X)u(X)
			-
			\mathbf 1_{\mathbb{R}^{d+1}_{-\tau}}(X^{*,x_0})u(X^{*,x_0})
			\\
			&\quad =
			\int_{\mathbb{R}^{d+1}_{-\tau}}
			\mathcal H_{x_0}(Y)\cdot\nabla_Y\G^D_{x_0}(X,Y)\,dY
			-
			\int_{\mathbb R^d}
			u(y,0^-)
			e_{d+1}\cdot A_0\nabla_Y\G^D_{x_0}(X,(y,0))\,dy
			+
			\mathrm{LO}_{\tau,-}^D(X).
		\end{aligned}
	\end{equation}
	Adding \eqref{D-plus} and \eqref{D-minus}, and using
	\[
	u(y,0^-)-u(y,0^+)=h(y),
	\]
	we deduce that
	\begin{equation}\label{B2-eq}
		\begin{aligned}
			u(X)-u(X^{*,x_0})
			&=
			\int_{\mathbb{R}^{d+1}_\tau\cup\mathbb{R}^{d+1}_{-\tau}}
			\mathcal H_{x_0}(Y)\cdot\nabla_Y\G^D_{x_0}(X,Y)\,dY
			\\
			&\quad
			-
			\int_{\mathbb R^d}
			e_{d+1}\cdot A_0\nabla_Y\G^D_{x_0}(X,(y,0))h(y)\,dy
			+
			\mathrm{LO}_\tau^D(X),
		\end{aligned}
	\end{equation}
	where
	\begin{equation}\label{def-LO-tau-D}
		\mathrm{LO}_\tau^D(X)
		:=
		\mathrm{LO}_{\tau,+}^D(X)
		+
		\mathrm{LO}_{\tau,-}^D(X).
	\end{equation}
	Here, for \(X\in\mathbb{R}^{d+1}_\tau\cup\mathbb{R}^{d+1}_{-\tau}\), exactly one of \(X\)
	and \(X^{*,x_0}\) lies in the upper half-domain and the other lies in the
	lower half-domain.
	
	If \(X\in\mathbb{R}_r^{d+1}\), then \eqref{B1-eq} and \eqref{B2-eq} give
	\[
	\mu_+^{-1}u(X)+\mu_-^{-1}u(X^{*,x_0})=\mathcal B_1(X,\tau),
	\qquad
	u(X)-u(X^{*,x_0})=\mathcal B_2(X,\tau),
	\]
	where
	\[
	\mathcal B_1(X,\tau)
	:=
	\sum_{\pm}\mu_\pm^{-1}
	\int_{\mathbb{R}_{\pm\tau}^{d+1}}
	\mathcal H_{x_0}(Y)\cdot\nabla_Y\G^N_{x_0}(X,Y)\,dY
	+
	\mathrm{LO}_\tau^N(X)
	\]
	and
	\[
	\mathcal B_2(X,\tau)
	:=
	\int_{\mathbb{R}^{d+1}_\tau\cup\mathbb{R}^{d+1}_{-\tau}}
	\mathcal H_{x_0}(Y)\cdot\nabla_Y\G^D_{x_0}(X,Y)\,dY
	-
	\int_{\mathbb R^d}
	e_{d+1}\cdot A_0\nabla_Y\G^D_{x_0}(X,(y,0))h(y)\,dy
	+
	\mathrm{LO}_\tau^D(X).
	\]
	Solving the system yields
	\[
	u(X)
	=
	\frac{\mu_+\mu_-}{\mu_++\mu_-}\mathcal B_1(X,\tau)
	+
	\frac{\mu_+}{\mu_++\mu_-}\mathcal B_2(X,\tau).
	\]
	If \(X\in\mathbb{R}_{-r}^{d+1}\), the same identities give directly
	\[
	\mu_-^{-1}u(X)+\mu_+^{-1}u(X^{*,x_0})=\mathcal B_1(X,\tau),
	\qquad
	u(X)-u(X^{*,x_0})=\mathcal B_2(X,\tau),
	\]
	and hence
	\[
	u(X)
	=
	\frac{\mu_+\mu_-}{\mu_++\mu_-}\mathcal B_1(X,\tau)
	+
	\frac{\mu_-}{\mu_++\mu_-}\mathcal B_2(X,\tau).
	\]
	Thus, for \(X=(x,z)\in\mathbb{R}_r^{d+1}\cup\mathbb{R}_{-r}^{d+1}\),
	\[
	u(X)
	=
	\frac{\mu_+\mu_-}{\mu_++\mu_-}\mathcal B_1(X,\tau)
	+
	\frac{
		\mu_+\mathbf 1_{\{z>0\}}
		+
		\mu_-\mathbf 1_{\{z<0\}}
	}
	{\mu_++\mu_-}
	\mathcal B_2(X,\tau).
	\]
	Averaging this identity in \(\tau\in[r,2r]\) gives
	\eqref{foruintf-correct-full}, with
	\begin{equation}\label{def-LO-final}
		\begin{aligned}
			\mathrm{LO}(u,g)(X)
			:=
			&\frac{\mu_+\mu_-}{\mu_++\mu_-}
			\frac1r\int_r^{2r}\mathrm{LO}_\tau^N(X)\,d\tau
			\\
			&+
			\frac{
				\mu_+\mathbf 1_{\{z>0\}}
				+
				\mu_-\mathbf 1_{\{z<0\}}
			}
			{\mu_++\mu_-}
			\frac1r\int_r^{2r}\mathrm{LO}_\tau^D(X)\,d\tau .
		\end{aligned}
	\end{equation}
	
	It remains to prove \eqref{LO-estimate-interf}. For weak solutions, the
	averaged artificial-boundary contributions are understood as integrals
	over the slabs
	\[
	\mathbb R^d\times(r,2r),
	\qquad
	\mathbb R^d\times(-2r,-r).
	\]
	
Let $\alpha$ be a multi-index with $|\alpha|=n\ge1$. By \eqref{def-LO-tau-pm}, after applying
	\(\partial_X^\alpha\), the artificial-boundary terms containing
	\(\nabla u\) or \(g\) have the form
	\[
	\frac1r\int_r^{2r}\int_{\mathbb R^d}
	K_\alpha(X,y,\pm\tau)
	\bigl(\nabla u(y,\pm\tau),g(y,\pm\tau)\bigr)\,dy\,d\tau.
	\]
    Here $K_\alpha$ denotes a generic kernel obtained by applying $\partial_X^\alpha$ to the frozen Green kernels appearing in the artificial-boundary representation.
	Since \(X\in\mathbb{R}_{r/2}^{d+1}\cup\mathbb{R}_{-r/2}^{d+1}\) and
	\(\tau\in[r,2r]\), we have $|X-(y,\pm\tau)|\geq \frac{1}{2}r$. Moreover,
	the standard decay of the frozen Green kernels and their derivatives gives
	\[
	\sup_{X\in\mathbb{R}_{r/2}^{d+1}\cup\mathbb{R}_{-r/2}^{d+1}}
	\left(
	\frac1r\int_r^{2r}
	\|K_\alpha(X,\cdot,\pm\tau)\|_{L^2_y}^2\,d\tau
	\right)^{1/2}
	\le C(\alpha,c_0,d)r^{-\alpha-\frac{d}{2}+1}.
	\]
	By Cauchy's inequality in \((y,\tau)\), these terms are bounded by
	\[
	C(r,\alpha,c_0,d)
	\left(
	\|\nabla u\|_{L^2(\mathbb{R}_{2r}^{d+1}\cup\mathbb{R}_{-2r}^{d+1})}
	+
	\|g\|_{L^2(\mathbb{R}_{2r}^{d+1}\cup\mathbb{R}_{-2r}^{d+1})}
	\right).
	\]
	It remains to estimate the artificial-boundary terms containing \(u\).
	They have the form
	\[
	I_\alpha(X)
	=
	\frac1r\int_r^{2r}\int_{\mathbb R^d}
	\partial_X^\alpha(e_{d+1}\cdot A_0\nabla_Y
			\G_{x_0}^{\mathcal S}(X,(y,\pm\tau)) )\,
	u(y,\pm\tau)\,dy\,d\tau,
	\qquad |\alpha|\ge1,\quad \mathcal{S}\in\{D,N\}.
	\]
	We claim that
	\begin{equation}\label{tangential-divergence-kernel}
		\partial_X^\alpha (e_{d+1}\cdot A_0\nabla_Y
			\G_{x_0}^{\mathcal S}(X,(y,\pm\tau)) )
		=
		\sum_{i=1}^d\partial_{y_i}L_{\alpha,i}(X,y,\pm\tau),
	\end{equation}
    where $L_{\alpha,i}$ satisfies 
 	\begin{equation}\label{S6eq29}
	\sup_{X\in\mathbb{R}_{r/2}^{d+1}\cup\mathbb{R}_{-r/2}^{d+1}}
	\left(
	\frac1r\int_r^{2r}
	\|L_{\alpha,i}(X,\cdot,\pm\tau)\|_{L^2_y}^2\,d\tau
	\right)^{1/2}
	\le C(\alpha,c_0,d)r^{-\alpha-\frac{d}{2}+1}.
	\end{equation}
	Indeed, for the direct part of the frozen kernels, tangential
	\(X\)-derivatives are equal to minus tangential \(y\)-derivatives. For
	the reflected part, the same statement holds after passing to the
	\(Q_{x_0}\)-coordinates; equivalently, in the original variables,
	tangential \(X\)-derivatives are constant linear combinations of
	tangential \(y\)-derivatives. Thus tangential \(X\)-derivatives already
	give the required tangential divergence structure.
	
	It remains to consider normal \(X\)-derivatives. The only terms
	which are not already tangential divergences are the pure normal second
	derivatives. Since \(X\in\mathbb{R}_{r/2}^{d+1}\cup\mathbb{R}_{-r/2}^{d+1}\) and
	\(\tau\in[r,2r]\), the points \((y,\pm\tau)\) remain away from both
	poles \(X\) and \(X^{*,x_0}\). Therefore the frozen kernels are
	\(A_0\)-harmonic in a neighborhood of the artificial boundaries:
	\[
	\operatorname{div}_Y(A_0\nabla_YG)=0.
	\]
	Writing \(A_0=(a_{ij})_{1\le i,j\le d+1}\), this gives
	\[
	a_{d+1,d+1}\partial_{ww}G
	+
	2\sum_{i=1}^d a_{i,d+1}\partial_{y_iw}G
	+
	\sum_{i,j=1}^d a_{ij}\partial_{y_iy_j}G
	=
	0.
	\]
	Since \(a_{d+1,d+1}\ge c_0>0\), every occurrence of
	\(\partial_{ww}G\) can be rewritten as a tangential divergence. Since the pole is excluded, the kernel is smooth and \eqref{S6eq29} holds. Commuting
	this identity with further \(Y\)-derivatives yields the same conclusion
	for higher-order differentiated kernels. Therefore
	\eqref{tangential-divergence-kernel} holds for all \(|\alpha|\ge1\).

With the expression \eqref{tangential-divergence-kernel}, we use integration by parts in \(y\)
to obtain
	\[
	I_\alpha(X)
	=
	-\frac1r\int_r^{2r}\int_{\mathbb R^d}
	L_{\alpha}(X,y,\pm\tau)\cdot
	\nabla_yu(y,\pm\tau)\,dy\,d\tau .
	\]
 Consequently,
	\[
	|I_\alpha(X)|
	\le
	C(r,\alpha,c_0,d)
	\|\nabla u\|_{L^2(\mathbb{R}_{2r}^{d+1}\cup\mathbb{R}_{-2r}^{d+1})}.
	\]
	Combining the estimates for the \(u\)-terms, the \(\nabla u\)-terms, and
	the \(g\)-terms proves \eqref{LO-estimate-interf} for every
	\(n\in\mathbb N_+\).
\end{proof}
				
                    For later use in the estimate of the boundary term generated by the jump data, we record that the same Poisson-kernel structure remains valid for any uniformly elliptic constant matrix.
\begin{lemma}\label{remF2}
Let $A_*$ be a constant symmetric matrix satisfying$$
c_0 \mathrm{Id}\le A_*\le c_0^{-1}\mathrm{Id} .
$$
Let $\G^D_{A_*}$ denote the Dirichlet Green function in the upper or lower half-space associated with the constant-coefficient operator
$
\operatorname{div}_{X}(A_* \nabla_X).
$
For $X=(x,z),Y=(y,w)\in\mathbb R^{d+1}_\pm$, define the corresponding boundary kernel
$$
\mathcal P_{A_*}(X,y)
:=
e_{d+1}\cdot (A_*\nabla_Y\G^D_{A_*}(X,Y))\big|_{w=0}.
$$ 
For $h\in C_c^\infty(\mathbb R^d)$, set
$$
T_{A_*}[h](X)
:=
\int_{\mathbb R^d}
\mathcal P_{A_*}(X,y)h(y) dy .
$$
Then for $\varkappa>1$,
\begin{align*}
	\|\nabla_XT_{A_*}[h]\|_{L^\infty}\lesssim \|h\|_{C^{1,\log^\varkappa}}.
\end{align*}
					\end{lemma}
                    \begin{proof}
                        Note that $\mathcal P_{A_*}$ is an anisotropic Poisson kernel. 
          Indeed, 
          let \[
Q_{*}=
\begin{pmatrix}
	B_{*} & b_{*}\\
	0 & q_{*}
\end{pmatrix},
\] be the matrix defined in \eqref{defQx0} associated to the matrix $A_*$. 
          By the definition of the Dirichlet Green function in \eqref{defGtiG}, it is easy to check that 
          \begin{align*}
            \mathcal{P}_{A_*}(X,y)=|\det B_*|^{-1}\mathcal{P}_{\mathrm{Id}}(Q_*^{-1}X,B_*^{-1}y).
          \end{align*}
          As a result, we obtain 
          \begin{align*}
            \nabla_XT_{A_*}[h](X)= Q_*^{-\top}(\nabla_XT_{\mathrm{Id}}[h\circ B_*])(Q_*^{-1}X).
          \end{align*}
      On the other hand, we have  $$\mathcal{P}_{\mathrm{Id}}(x,z,y)=\operatorname{sgn}(z)K_\b(x-y,z)|_{\b=0},$$ where $K_\b$ is defined in \eqref{defK}.  From Lemma \ref{lemKf}, we obtain that 
      \begin{align*}
        \|\nabla_XT_{\mathrm{Id}}[h]\|_{L^\infty}\lesssim  \|\nabla_{x,z}(K_\b(\cdot,z)\ast h)(x)|_{\b=0}\|_{L^\infty}\lesssim \|h\|_{C^{1,\log^{\varkappa}}},
      \end{align*}
      with $\varkappa>1$. Then it follows that 
      \begin{align*}
        \|  \nabla_XT_{A_*}[h]\|_{L^\infty}\lesssim\|h\circ B_*\|_{C^{1,\log^{\varkappa}}}\lesssim \|h\|_{C^{1,\log^{\varkappa}}},
      \end{align*}
where  the implicit constant depends only on \(d,\varkappa\) and the ellipticity constant \(c_0\).    The proof is complete.
                
                    \end{proof}

					We have the following lemma considering the Lipschitz estimate of the solution in Lemma \ref{leminterf}.
\begin{lemma}\label{lemLip}
	Let  \(\varkappa>1\).
	Assume that
	\[
	A\in C^{\log^\varkappa}(\mathbb R^d;\mathbb R^{(d+1)\times(d+1)}_{\rm sym}),
	\qquad
	c_0\mathrm{Id}\le A(y)\le c_0^{-1}\mathrm{Id}.
	\]
There exists \(r=r(d,c_0,\mu^\pm,\varkappa,\|A\|_{\dot C^{\log^\varkappa}})\in(0,1)\) such that, if    \(u\in \dot W^{1,\infty}(\mathbb{R}_{3r}^{d+1}\cup\mathbb{R}_{-3r}^{d+1})\) is a weak solution of \eqref{ellinte} in
	\(\mathbb{R}_{3r}^{d+1}\cup\mathbb{R}_{-3r}^{d+1}\), then
	\begin{equation}\label{Lip-estimate-interface}
		\begin{aligned}
			\|\nabla_{x,z}u\|_{L^\infty(\mathbb{R}_{r/2}^{d+1}\cup\mathbb{R}_{-r/2}^{d+1})}
			\le\;&
			C\|h\|_{C^{1,\log^\varkappa}(\mathbb R^d)}
			+
			C\|g\|_{C^{\log^\varkappa} (\mathbb{R}_{3r}^{d+1}\cup\mathbb{R}_{-3r}^{d+1})}
			\\
			&+
			C(r)
			\left(
			\|\nabla_{x,z}u\|_{L^2(\mathbb{R}_{3r}^{d+1}\cup\mathbb{R}_{-3r}^{d+1})}
			+
			\|g\|_{L^2(\mathbb{R}_{3r}^{d+1}\cup\mathbb{R}_{-3r}^{d+1})}
			\right),
		\end{aligned}
	\end{equation}
	where
	$
	C=C(d,c_0,\mu_\pm,\varkappa),
	C(r)=C(d,c_0,\mu_\pm,\varkappa,r),
	$
	and
	\[
	\|g\|_{C^{\log^\varkappa} (\mathbb{R}_{3r}^{d+1}\cup\mathbb{R}_{-3r}^{d+1})}
	:=
	\|g\|_{C^{\log^\varkappa}(\mathbb{R}_{3 r}^{d+1})}+	\|g\|_{C^{\log^\varkappa}(\mathbb{R}_{-3 r}^{d+1})}
	\]
	is understood phase by phase. 
\end{lemma}
\begin{proof}
	It suffices to prove the estimate for smooth solutions with sufficient
	decay in the tangential variables. The estimate for weak solutions follows
	by the approximation argument in Lemma \ref{leminterf}, together with the
	differentiated lower-order estimate \eqref{LO-estimate-interf}.
	
	For \(\tau>0\), recall that
	\begin{equation*}
		\mathbb R^{d+1}_{\tau}:=\mathbb R^d\times(0,\tau),
		\qquad
		\mathbb R^{d+1}_{-\tau}:=\mathbb R^d\times(-\tau,0).
	\end{equation*}
	For \(\diamond \in\{+,-\}\), we also use the shorthand notation
	\begin{equation*}
		\mathbb R^{d+1}_{\diamond \tau}
		:=
		\begin{cases}
			\mathbb R^{d+1}_{\tau}, & \diamond =+,\\
			\mathbb R^{d+1}_{-\tau}, & \diamond =-.
		\end{cases}
	\end{equation*}
	All \(C^{\log^\varkappa}\)-norms on
	\(\mathbb R^{d+1}_{\tau}\cup\mathbb R^{d+1}_{-\tau}\) are understood
	phase by phase.
	
	We first prove the estimate in the upper phase. The lower-phase estimate is
	obtained by repeating the same argument with the roles of \(+\) and \(-\)
	interchanged. Fix
	\[
	X_*=(x_*,z_*)\in \mathbb R^{d+1}_{r/2}.
	\]
	In the representation formula \eqref{foruintf-correct-full}, choose
	\[
	x_0=x_*,
	\qquad
	A_0=A(x_*),
	\]
	and keep this frozen point fixed when differentiating with respect to
	\(X=(x,z)\). We then evaluate the differentiated identity at \(X=X_*\).
	
	Set
	\begin{equation*}
		\theta_+:=\frac{\mu_+}{\mu_++\mu_-},
		\qquad
		\theta_-:=\frac{\mu_-}{\mu_++\mu_-}.
	\end{equation*}
	We define the frozen transmission kernel by
	\begin{equation*}
		\mathfrak K_{x_*}^{\diamond}(X,Y)
		:=
		\theta_+ \G_{x_*}^D(X,Y)
		+
		\theta_{-\diamond}\G_{x_*}^N(X,Y),\quad\quad \diamond\in\{+,-\}.
	\end{equation*}
	Then the body terms in \eqref{foruintf-correct-full} can be written as
	\begin{equation*}
		\sum_{\sigma=\pm}
		\frac1r\int_r^{2r}
		\int_{\mathbb R^{d+1}_{\diamond\tau}}
		\mathcal H_{x_*}(Y)\cdot
		\nabla_Y\mathfrak K_{x_*}^{\diamond}(X,Y)\,dY\,d\tau,
	\end{equation*}
	where, for \(Y=(y,w)\),
	\begin{equation*}
		\mathcal H_{x_*}(Y)
		:=
		g(Y)+\bigl(A(x_*)-A(y)\bigr)\nabla_Yu(Y).
	\end{equation*}
	The boundary jump term is
	\begin{equation*}
		-\theta_+
		\int_{\mathbb R^d}
		e_{d+1}\cdot A(x_*)\nabla_Y\G_{x_*}^D(X,(y,0))h(y)\,dy,
	\end{equation*}
	and the remaining term is denoted by \(\mathrm{LO}(u,g)(X)\).
	
	We next record the kernel structure used below. Write
	\begin{equation*}
		\G_{x_*}^D=\mathcal K_{x_*}-\mathcal K_{x_*}^{\rm ref},
		\qquad
		\G_{x_*}^N=\mathcal K_{x_*}+\mathcal K_{x_*}^{\rm ref},
	\end{equation*}
	where \(\mathcal K_{x_*}\) is the direct frozen fundamental solution and
	\(\mathcal K_{x_*}^{\rm ref}\) is the reflected part. Hence
	\begin{equation*}
		\mathfrak K_{x_*}^{\diamond}
		=
		(\theta_++\theta_{-\diamond})\mathcal K_{x_*}
		+
		(-\theta_++\theta_{-\diamond})\mathcal K_{x_*}^{\rm ref}.
	\end{equation*}
	Since \(X_*\in \mathbb R^{d+1}_{r/2}\), the reflected pole is absent from
	the opposite phase. More precisely,
	\begin{equation*}
		\mathfrak K_{x_*}^{-}
		=
		2\theta_+\,\mathcal K_{x_*}.
	\end{equation*}
	The remaining direct pole may still approach the interface as \(z_*\to0\),
	so the opposite-phase contribution is treated as a near-boundary singular
	term, not as a uniformly smooth error.
	
	For each \(\diamond \in\{+,-\}\) and \(\tau\in[r,2r]\), let
	\(\mathcal P_{\diamond ,\tau}\) denote the finite set of relevant poles of
	\[
	\nabla_X\nabla_Y\mathfrak K_{x_*}^{\diamond }(X_*,Y)
	\]
	with respect to the phase \(\mathbb R^{d+1}_{\diamond \tau}\). These are the
	poles which either lie in \(\mathbb R^{d+1}_{\diamond \tau}\), or lie outside
	\(\mathbb R^{d+1}_{\diamond \tau}\) but have distance \(\lesssim r\) from this
	phase. All other poles are uniformly separated and will be included in the
	regular part. For \(P=(p,s)\in\mathcal P_{\diamond ,\tau}\), set
	\begin{equation*}
		a_P:=
		\operatorname{dist}\bigl(P,\mathbb R^{d+1}_{\diamond \tau}\bigr).
	\end{equation*}
	Thus \(a_P=0\) if \(P\in\mathbb R^{d+1}_{\diamond \tau}\), while \(a_P>0\) if
	\(P\) is an exterior pole. For every such pole,
	\begin{equation}\label{pole-tangential-distance}
		|p-x_*|\le C a_P .
	\end{equation}
	For the direct pole this is immediate. For the reflected pole it follows
	from the explicit frozen reflection formula and from the fact that its
	normal distance to the relevant phase is comparable to \(|z_*|\).
	
	For \(P\in\mathcal P_{\diamond ,\tau}\), define
	\begin{equation*}
		\mathcal N_P
		:=
		\left\{
		Y\in\mathbb R^{d+1}_{\diamond \tau}:\ |Y-P|\le cr
		\right\}.
	\end{equation*}
	On \(\mathcal N_P\), one has \(|Y-P|\ge a_P\), and the frozen kernel satisfies
	\begin{equation}\label{kernel-singular-bound}
		\left|
		\nabla_X\nabla_Y\mathfrak K_{x_*}^{\diamond }(X_*,Y)
		\right|
		\le C|Y-P|^{-(d+1)}.
	\end{equation}
	Here the distance may be taken either in the Euclidean metric or in the
	equivalent frozen metric. Away from the union of the neighborhoods
	\(\mathcal N_P\), the differentiated kernels have \(L^2\)-norm bounded by
	\(C(r)\), uniformly in
	\(X_*\in\mathbb R^{d+1}_{r/2}\cup\mathbb R^{d+1}_{-r/2}\) and
	\(\tau\in[r,2r]\).

On $\mathcal{N}_P$, the singular part of $\nabla_X\nabla_Y\mathfrak{K}^\diamond_{x_\ast}(X_\ast,Y)$ generated by the pole $P$ will be denoted by $\mathcal K_P(Y)$. Thus, near the relevant poles,
\begin{equation*}
	\nabla_X\nabla_Y\mathfrak{K}^\diamond_{x_\ast}(X_\ast,Y)
=
	\sum_{P\in\mathcal P_{\diamond,\tau}}\mathcal K_P(Y)
	+
	\mathcal R_{\diamond,\tau}(Y),
\end{equation*}
where $\mathcal R_{\diamond,\tau}$ contains the pole contributions separated from the phase and the regular remainder.

We record the cancellation of $\mathcal K_P$. Define the frozen distance from $P$ by
\begin{equation*}
	\rho_P(Y):=\left|Q_{x_\ast}^{-1}(Y-P)\right|.
\end{equation*}
After the affine change $Z=Q_{x_\ast}^{-1}(Y-P)$, the kernel $\mathcal K_P$ is a constant matrix multiple of $\nabla^2\Phi(Z)$. Since $\nabla^2\Phi$ is homogeneous of degree $-(d+1)$, we may write
\begin{equation*}
	\mathcal K_P(Y)
=
	\rho_P(Y)^{-(d+1)}
	\Omega_P\left(
	\frac{Q_{x_\ast}^{-1}(Y-P)}{\rho_P(Y)}
	\right),
\end{equation*}
where $\Omega_P$ is smooth on $\mathbb S^d$. From the cancellation property of $\Phi$ in \eqref{S6eq1},   for every $0<\varepsilon<R$,
\begin{equation}\label{cancelP}
	\int_{\varepsilon<\rho_P(Y)<R}\mathcal K_P(Y) dY=0.
\end{equation}
Thus the cancellation holds on the $A_\ast(x_*)$-deformed annuli centered at $P$.
    
	We now estimate the four terms in the differentiated representation formula.
	
	First, consider the jump term
	\begin{equation*}
		T_h(X)
		:=
		\int_{\mathbb R^d}
		e_{d+1}\cdot A(x_*)\nabla_Y\G_{x_*}^D(X,(y,0))h(y)\,dy.
	\end{equation*}
	By Lemma \ref{remF2}, we have
	\begin{equation}\label{Th-est}
		|\nabla_XT_h(X_*)|
		\lesssim
		\|h\|_{C^{1,\log^\varkappa}(\mathbb R^d)}.
	\end{equation}
	
	Next consider the \(g\)-part:
	\begin{equation*}
		T_g(X)
		:=
		\sum_{\diamond =\pm}
		\frac1r\int_r^{2r}
		\int_{\mathbb R^{d+1}_{\diamond \tau}}
		g(Y)\cdot
		\nabla_Y\mathfrak K_{x_*}^{\diamond }(X,Y)\,dY\,d\tau .
	\end{equation*}
	We decompose the kernel into singular pieces near the poles
	\(P\in\mathcal P_{\diamond ,\tau}\) and a regular piece. On each singular
	neighborhood \(\mathcal N_P\), we subtract a phasewise constant \(g_P\).
	If \(P\in\mathbb R^{d+1}_{\diamond \tau}\), we take \(g_P=g(P)\). If
	\(P\notin\mathbb R^{d+1}_{\diamond \tau}\), we choose
	\(P^\sharp\in\mathbb R^{d+1}_{\diamond \tau}\) such that
	\begin{equation*}
		|P^\sharp-P|\le 2a_P,\qquad |Y-P^\sharp|\le |P-P^\sharp|+|Y-P|\le 3|Y-P|,
	\end{equation*}
	and set \(g_P=g(P^\sharp)\).
	
	The constant part is treated by inserting and subtracting the corresponding
	principal value integral over the full frozen phase. This full singular
	integral vanishes by the cancellation of the frozen Calderon--Zygmund
	kernel, see \eqref{cancelP}. The error produced by the localization is supported in an annular
	region separated from the pole and is therefore included in the regular
	part. Hence the singular contribution is reduced to the integral with
	\(g-g_P\). By the phasewise logarithmic modulus of \(g\),
	\begin{equation}\label{g-minus-gP}
		|g(Y)-g_P|
		\le
		C\|g\|_{C^{\log^\varkappa}
			(\mathbb R^{d+1}_{2r}\cup\mathbb R^{d+1}_{-2r})}
		\log^{-\varkappa}(2+|Y-P|^{-1}).
	\end{equation}
	Combining \eqref{g-minus-gP} with \eqref{kernel-singular-bound}, we obtain
	\begin{equation*}
		\int_{\mathcal N_P}
		\left|
		\nabla_X\nabla_Y\mathfrak K_{x_*}^{\sigma}(X_*,Y)
		\right|
		|g(Y)-g_P|\,dY
		\lesssim
		\|g\|_{C^{\log^\varkappa}}
		\int_0^{Cr}
		\frac{d\ell}{\ell\log^\varkappa(2+\ell^{-1})}.
	\end{equation*}
	Since \(\varkappa>1\), the last integral is finite (note that $r< 1$). The regular pieces are
	estimated by Cauchy's inequality and the \(L^2\)-bounds of the differentiated
	kernels. Therefore
	\begin{equation}\label{Tg-est}
		|\nabla_XT_g(X_*)|
		\le
		C\|g\|_{C^{\log^\varkappa}
			(\mathbb R^{d+1}_{2r}\cup\mathbb R^{d+1}_{-2r})}
		+
		C(r)\|g\|_{L^2
			(\mathbb R^{d+1}_{2r}\cup\mathbb R^{d+1}_{-2r})}.
	\end{equation}
		We now estimate the coefficient-freezing error
	\begin{equation*}
		T_A(X)
		:=
		\sum_{\diamond=\pm}
		\frac1r\int_r^{2r}
		\int_{\mathbb R^{d+1}_{\diamond \tau}}
		\bigl(A(x_*)-A(y)\bigr)\nabla_Yu(Y)\cdot
		\nabla_Y\mathfrak K_{x_*}^{\diamond }(X,Y)\,dY\,d\tau .
	\end{equation*}
	For each \(\diamond \) and \(\tau\), choose
	a smooth cutoff \(\chi_{\tau}\) such that
	\begin{equation*}
		\begin{aligned}
		    &0\le \chi_{\diamond,\tau}\le1,\qquad
		\operatorname{supp}\chi_{\diamond,\tau}\subset\bigcup_{P\in\mathcal P_{\diamond ,\tau}}\mathcal N_P,\\
        &\chi_{\diamond,\tau}\equiv1
		\quad\text{on}\quad
		\left\{
		Y\in\mathbb R^{d+1}_{\diamond \tau}:\ \operatorname{dist}(Y,\cup_{P\in\mathcal P_{\diamond ,\tau}}\mathcal N_P)\le cr/2
		\right\}.
		\end{aligned}
	\end{equation*}
	Set
	\begin{equation*}
		\chi_{0,\diamond ,\tau}
		:=
		1-\chi_{\diamond,\tau}.
	\end{equation*}
	We decompose
	\begin{equation*}
		T_A=T_A^{\rm sing}+T_A^{\rm reg},
	\end{equation*}
	where
	\begin{align}
		T_A^{\rm sing}(X)
		:=
		\sum_{\diamond =\pm}
		\frac1r\int_r^{2r}
		\int_{\mathbb R^{d+1}_{\diamond \tau}}
		&\chi_{\diamond,\tau}(Y)
		\bigl(A(x_*)-A(y)\bigr)\nabla_Yu(Y)
		\nonumber\\
		&\cdot
		\nabla_Y\mathfrak K_{x_*}^{\diamond }(X,Y)\,dY\,d\tau ,
		\nonumber
	\end{align}
	and
	\begin{align}
		T_A^{\rm reg}(X)
		:=
		\sum_{\diamond =\pm}
		\frac1r\int_r^{2r}
		\int_{\mathbb R^{d+1}_{\diamond \tau}}
		&\chi_{0,\diamond ,\tau}(Y)
		\bigl(A(x_*)-A(y)\bigr)\nabla_Yu(Y)
		\nonumber\\
		&\cdot
		\nabla_Y\mathfrak K_{x_*}^{\diamond }(X,Y)\,dY\,d\tau .
		\nonumber
	\end{align}
	We first estimate the singular part. Let
	\(P=(p,s)\in\mathcal P_{\diamond ,\tau}\). For
	\(Y=(y,w)\in\mathcal N_P\), using \eqref{pole-tangential-distance} and
	\(|Y-P|\ge a_P\), we get
	\begin{equation*}
		|y-x_*|
		\le |y-p|+|p-x_*|
		\le C(|Y-P|+a_P)
		\le C|Y-P|.
	\end{equation*}
	Therefore, by the logarithmic modulus of \(A\),
	\begin{equation}\label{A-log-modulus}
		|A(x_*)-A(y)|
		\le
		C\|A\|_{\dot C^{\log^\varkappa}(\mathbb R^d)}
		\log^{-\varkappa}(2+|Y-P|^{-1}),
	\end{equation}
	where
	\[
	\|A\|_{\dot C^{\log^\varkappa}(\mathbb R^d)}
	:=
	\sup_{x\ne y}
	\frac{|A(x)-A(y)|}
	{\log^{-\varkappa}(2+|x-y|^{-1})}.
	\]
	Combining \eqref{A-log-modulus} with \eqref{kernel-singular-bound}, we obtain
	\begin{align}
		|\nabla_XT_A^{\rm sing}(X_*)|
		&\le
		C\|A\|_{\dot C^{\log^\varkappa}(\mathbb R^d)}
		\|\nabla u\|_{L^\infty
			(\mathbb R^{d+1}_{2r}\cup\mathbb R^{d+1}_{-2r})}
		\nonumber\\
		&\qquad\qquad\times
		\frac1r\int_r^{2r}
		\sum_{\diamond =\pm}
		\sum_{P\in\mathcal P_{\diamond ,\tau}}
		\int_0^{Cr}
		\frac{d\ell}{\ell\log^\varkappa(2+\ell^{-1})}
		\,d\tau
		\nonumber\\
		&\le
		c(r)\|A\|_{\dot C^{\log^\varkappa}(\mathbb R^d)}
		\|\nabla u\|_{L^\infty
			(\mathbb R^{d+1}_{2r}\cup\mathbb R^{d+1}_{-2r})},
		\nonumber
	\end{align}
	where
	\begin{equation*}
		c(r)
		:=
		C\int_0^{Cr}
		\frac{d\ell}{\ell\log^\varkappa(2+\ell^{-1})}
		\longrightarrow0
		\qquad\text{as } r\to0.
	\end{equation*}
	This is the only place where the log-Dini condition \(\varkappa>1\) is used
	in an essential way.
	
	We next estimate the regular part. On the support of
	\(\chi_{0,\diamond,\tau}\), the kernel is separated from all relevant interior
	and exterior poles. Hence
	\begin{equation*}
		\sup_{X_*\in\mathbb R^{d+1}_{r/2}}
		\left(
		\frac1r\int_r^{2r}
		\left\|
		\chi_{0,\diamond,\tau}
		\nabla_X\nabla_Y\mathfrak K_{x_*}^{\diamond }
		(X_*,\cdot)
		\right\|_{L^2(\mathbb R^{d+1}_{\diamond \tau})}^2
		\,d\tau
		\right)^{1/2}
		\le C(r).
	\end{equation*}
	Since \(A\) is bounded, Cauchy's inequality gives
	\begin{equation*}
		|\nabla_XT_A^{\rm reg}(X_*)|
		\le
		C(r)\|\nabla u\|_{L^2
			(\mathbb R^{d+1}_{2r}\cup\mathbb R^{d+1}_{-2r})}.
	\end{equation*}
	Consequently,
	\begin{equation}\label{TA-est}
		|\nabla_XT_A(X_*)|
		\le
		c(r)\|A\|_{\dot C^{\log^\varkappa}(\mathbb R^d)}
		\|\nabla u\|_{L^\infty
			(\mathbb R^{d+1}_{2r}\cup\mathbb R^{d+1}_{-2r})}
		+
		C(r)\|\nabla u\|_{L^2
			(\mathbb R^{d+1}_{2r}\cup\mathbb R^{d+1}_{-2r})}.
	\end{equation}
	Finally, applying the lower-order estimate in Lemma \ref{leminterf} with
	\(n=1\), we obtain
	\begin{align}
		&\|\nabla_X\mathrm{LO}(u,g)\|_{L^\infty
			(\mathbb R^{d+1}_{r/2}\cup\mathbb R^{d+1}_{-r/2})}
	\le
		C(r)
		\left(
		\|\nabla u\|_{L^2
			(\mathbb R^{d+1}_{2r}\cup\mathbb R^{d+1}_{-2r})}
		+
		\|g\|_{L^2
			(\mathbb R^{d+1}_{2r}\cup\mathbb R^{d+1}_{-2r})}
		\right).
		\label{LO-Lip-est}
	\end{align}
	Combining \eqref{Th-est}, \eqref{Tg-est}, \eqref{TA-est}, and
	\eqref{LO-Lip-est}, and then repeating the same argument in the lower phase $\mathbb{R}^{d+1}_{-r/2}$,
	we get
	\begin{align}
		\|\nabla_{x,z}u\|_{L^\infty
			(\mathbb R^{d+1}_{r/2}\cup\mathbb R^{d+1}_{-r/2})}
		&\le
		C\|h\|_{C^{1,\log^\varkappa}(\mathbb R^d)}
		+
		C\|g\|_{C^{\log^\varkappa}
			(\mathbb R^{d+1}_{2r}\cup\mathbb R^{d+1}_{-2r})}
		\nonumber\\
		&\qquad\quad+
		c(r)\|A\|_{\dot C^{\log^\varkappa}(\mathbb R^d)}
		\|\nabla_{x,z}u\|_{L^\infty
			(\mathbb R^{d+1}_{2r}\cup\mathbb R^{d+1}_{-2r})}
		\nonumber\\
		&\qquad\quad+
		C(r)
		\left(
		\|\nabla_{x,z}u\|_{L^2
			(\mathbb R^{d+1}_{2r}\cup\mathbb R^{d+1}_{-2r})}
		+
		\|g\|_{L^2
			(\mathbb R^{d+1}_{2r}\cup\mathbb R^{d+1}_{-2r})}
		\right).
		\label{lipofu}
	\end{align}
	It remains to remove the larger \(L^\infty\)-norm on the right-hand side.
	Let
	\begin{equation*}
		\mathcal A_r
		:=
		\mathbb R^d\times
		\bigl([r/2,2r]\cup[-2r,-r/2]\bigr).
	\end{equation*}
	On \(\mathcal A_r\), the equation is separated from the interface. Applying
	the standard interior Lipschitz estimate for divergence-form elliptic
	equations with Dini-continuous coefficients separately in
	\(\mathbb R^d\times[r/2,2r]\) and
	\(\mathbb R^d\times[-2r,-r/2]\), we obtain
	\begin{align}
		\|\nabla_{x,z}u\|_{L^\infty(\mathcal A_r)}
		\le
		C(r)
		\left(
		\|\nabla_{x,z}u\|_{L^2
			(\mathbb R^{d+1}_{3r}\cup\mathbb R^{d+1}_{-3r})}
		+
		\|g\|_{C^{\log^\varkappa}
			(\mathbb R^{d+1}_{3r}\cup\mathbb R^{d+1}_{-3r})}
		\right).
		\label{interior-annulus-est}
	\end{align}
	Since
	\begin{equation*}
		\mathbb R^{d+1}_{2r}\cup\mathbb R^{d+1}_{-2r}
		=
		\left(
		\mathbb R^{d+1}_{r/2}\cup\mathbb R^{d+1}_{-r/2}
		\right)
		\cup
		\mathcal A_r,
	\end{equation*}
	we have
	\begin{align}
		&\|\nabla_{x,z} u\|_{L^\infty
			(\mathbb R^{d+1}_{2r}\cup\mathbb R^{d+1}_{-2r})}
\le
		\|\nabla_{x,z} u\|_{L^\infty
			(\mathbb R^{d+1}_{r/2}\cup\mathbb R^{d+1}_{-r/2})}
		+
		\|\nabla_{x,z} u\|_{L^\infty(\mathcal A_r)}.
		\label{large-strip-split}
	\end{align}
	Substituting \eqref{interior-annulus-est} and \eqref{large-strip-split}
	into \eqref{lipofu}, 
	we obtain
	\begin{align}
		\|\nabla_{x,z}u\|_{L^\infty
			(\mathbb R^{d+1}_{r/2}\cup\mathbb R^{d+1}_{-r/2})}
		&\le
		C\|h\|_{C^{1,\log^\varkappa}(\mathbb R^d)}
		+
		C\|g\|_{C^{\log^\varkappa}
			(\mathbb R^{d+1}_{3r}\cup\mathbb R^{d+1}_{-3r})}
		\nonumber\\
		&\qquad\quad+
		c(r)\|A\|_{\dot C^{\log^\varkappa}(\mathbb R^d)}
		\|\nabla_{x,z}u\|_{L^\infty
			(\mathbb R^{d+1}_{r/2}\cup\mathbb R^{d+1}_{-r/2})}
		\nonumber\\
		&\qquad\quad+
		C(r)
		\left(
		\|\nabla_{x,z}u\|_{L^2
			(\mathbb R^{d+1}_{3r}\cup\mathbb R^{d+1}_{-3r})}
		+
		\|g\|_{L^2
			(\mathbb R^{d+1}_{3r}\cup\mathbb R^{d+1}_{-3r})}
		\right).
		\label{pre-absorption}
	\end{align}
	Choose \(r\in(0,1)\) sufficiently small, depending on
	\(\|A\|_{\dot C^{\log^\varkappa}(\mathbb R^d)}\), so that
	\begin{equation*}
		c(r)\|A\|_{\dot C^{\log^\varkappa}(\mathbb R^d)}
		\le \frac12.
	\end{equation*}
	Absorbing the corresponding term into the left-hand side of
	\eqref{pre-absorption}, we obtain \eqref{Lip-estimate-interface}. The proof is complete. 
\end{proof}
             ~~~\vspace{0.2cm}\\       
			We conclude the appendix with two elementary singular-integral estimates used in
the proof. The first one is a Hölder estimate for a Calderón--Zygmund type
operator whose kernel depends on the base point.
					Let $\mathcal{K}:\mathbb{R}^d\times\mathbb{R}^d\to \mathbb{R}$ be such that 
					\begin{equation}\label{defczop}
						\begin{cases}
							&|\mathcal{K}(x,y)|\leq C|x-y|^{-d},\ \ \ \forall x,y\in\mathbb{R}^d, x\neq y,\\
							&\int_{|x-y|\geq 10 |\alpha|}|\delta_\alpha^x \mathcal{K}(x,y)||x-y|^a dy\lesssim |\alpha|^a,\ \ \ \forall \alpha\in \mathbb{R}^d, a\in(0,1).
						\end{cases}
					\end{equation}
					For any $h:\mathbb{R}^d\times \mathbb{R}^d\to \mathbb{R}$, define the integral operator $\T $ as following.
					\begin{align}
						\T h(x,x'):=\mathrm{P.V.}\int_{\mathbb{R}^d}\mathcal{K}(x,y) h(y,x') dy.
					\end{align}
					We have the following results.
					\begin{lemma}\label{lemsgit}
						Let $K$ satisfy \eqref{defczop} and 
						\begin{align}\label{kcanc}
							\sup_{0<r_1<r_2<\infty} \left|\int_{r_1<|x-y|<r_2}\mathcal{K}(x,y)dy\right|\leq c_1,
						\end{align}
						for a constant $c_1\geq 0$.
						Then for $h:\mathbb{R}^d\times \mathbb{R}^d\to \mathbb{R}$, $a\in(0,1)$, 
						\begin{align}
							&   \sup_{x,\alpha} \frac{|\T h(x,x)-\T h(x-\alpha,x)|}{|\alpha|^a}\lesssim \frac{1}{a(1-a)}\lceil h\rfloor_{a},\label{holdTf}
						\end{align}
						where 
						$$
						\lceil h\rfloor_{a}:=    \sup_{x,\alpha}\sup_{|z|\leq |\alpha|}\frac{| h(x,x+z)-h(x-\alpha,x+z)|}{|\alpha|^a}.
						$$
					More generally, for $\mathbf{p}=(p_i)_{i=1}^d\in[1,\infty]^d$, define the anisotropic Lebesgue space $L^{\mathbf{p}}_x=L^{p_1}_{x_1}\cdots L^{p_d}_{x_d}$, it holds
						\begin{align}\label{L2Tf}
							\sup_{\alpha}\frac{\|\T h(x,x)-\T h(x-\alpha,x)\|_{L^{\mathbf{p}}_x}}{|\alpha|^a}\lesssim \sup_{\alpha}\sup_{|z|\leq |\alpha|}\frac{\|h(x,x+z)-h(x-\alpha,x+z)\|_{L^{\mathbf{p}}_x}}{|\alpha|^a}.
						\end{align}
					\end{lemma}
					\begin{proof}
                    Let $x_1,x_2\in\mathbb R^d$ and set $\delta=|x_1-x_2|$. By the cancellation
assumption \eqref{kcanc}, we may subtract constants in the principal value
integrals. Hence 
						\begin{align*}
							&	|\T h(x_1,x_1)-\T h(x_2,x_1)|\\
							&\leq \sum_{j=1,2}\left|\int_{|x_j-y|\leq 10\delta}\mathcal{K}(x_j,y)(h(y,x_1)-h(x_j,x_1))dy\right|+ 2c_1|h(x_1,x_1)-h(x_2,x_1)|\\
							&+\int_{\mathbb{R}^d} \left|\mathbf{1}_{|x_1-y|>10\delta}\mathcal{K}(x_1,y)-\mathbf{1}_{|x_2-y|>10\delta}\mathcal{K}(x_2,y)\right||h(y,x_1)-h(x_1,x_1)|dy\\
							&\lesssim \left(\int_{|y|\leq 10\delta}\frac{|y|^a}{|y|^d}dy+\delta^a+\int_{|x_1-y|>10\delta}|\mathcal{K}(x_1,y)-\mathcal{K}(x_2,y)||y-x_1|^ady\right.\\
							&\quad\left.\quad\quad+\int_{9\delta\leq |x_1-y|,|x_2-y|\leq 10\delta} (|\mathcal{K}(x_1,y)|+|\mathcal{K}(x_2,y)|)|y-x_1|^ady\right)\lceil h\rfloor_{a}\\
							&\lesssim \frac{\delta^a\lceil h\rfloor_{a}}{a(1-a)},
						\end{align*}
						which leads to \eqref{holdTf}.  The estimate \eqref{L2Tf} follows similarly by setting $x_1=x, x_2=x-\alpha$ and taking the $L^{\mathbf{p}}$ norm in $x$. This completes the proof of the lemma.
					\end{proof}\vspace{0.2cm}\\

The next lemma will be used to pass from estimates with a frozen parameter
to estimates in which the frozen parameter is chosen depending on the
spatial point.  In our application the frozen parameter is
\[
\b=\nabla f(x).
\]
Thus, even if a quantity \(g(\b,\cdot)\) is controlled in \(L^2_x\)
uniformly for each fixed \(\b\), we still need a mechanism to control
\[
g(\nabla f(x),x)
\]
as a function of \(x\).  The following parameter-selection estimate gives
such a mechanism.  It only uses the boundedness of the range of
\(\nabla f\); no regularity of the selecting map \(x\mapsto\nabla f(x)\) is
needed.  The price is Sobolev regularity of \(g\) in the parameter variable,
above the threshold \(d/2\).
\begin{lemma}\label{lem:param-selection}
Let \(k,d\ge1\), let \(U\subset\mathbb R^k\) be open, and let
\(K\Subset U\). Let \(h:\mathbb R^d\to K\) be measurable. Let
\(l\in\mathbb N\) satisfy
$
l>\frac{k}{2}.
$
Then there exists an open set \(U_0\Subset U\), with \(K\Subset U_0\), such
that for every function \(g(\beta,x):U\times\mathbb R^d\to\mathbb R^N\) with
\[
\nabla_\beta^\alpha g\in L^2_\beta(U_0;L^2_x(\mathbb R^d)),
\qquad |\alpha|\le l,
\]
one has
\begin{equation}\label{param-selection-L2}
\|g(h(\cdot),\cdot)\|_{L^2_x(\mathbb R^d)}
\lesssim_{K,U_0,l}
\sum_{|\alpha|\le l}
\|\nabla_\beta^\alpha g\|_{L^2_\beta(U_0;L^2_x(\mathbb R^d))}.
\end{equation}
In particular,
\begin{equation}\label{param-selection-L2-sup}
\|g(h(\cdot),\cdot)\|_{L^2_x(\mathbb R^d)}
\lesssim_{K,U_0,l}
\sum_{|\alpha|\le l}
\sup_{\beta\in U_0}
\|\nabla_\beta^\alpha g(\beta,\cdot)\|_{L^2_x(\mathbb R^d)}.
\end{equation}
The constant is independent of the measurable selector \(h\) and of \(g\).
\end{lemma}
\begin{proof}
Choose \(U_0\Subset U\) such that \(K\Subset U_0\), and choose
\(\chi\in C_c^\infty(U_0)\) with \(\chi\equiv1\) on a neighborhood of
\(K\). For a.e. \(x\in\mathbb R^d\), apply the Sobolev embedding
\(H^l(\mathbb R^k)\hookrightarrow L^\infty(\mathbb R^k)\) to
\[
\beta\mapsto \chi(\beta)g(\beta,x).
\]
Since \(h(x)\in K\) and \(\chi=1\) on \(K\), we get
\[
|g(h(x),x)|
\le
\|\chi g(\cdot,x)\|_{L^\infty_\beta}
\lesssim\|\chi g(\cdot,x)\|_{H^l_\beta(\mathbb R^k)}.
\]
By Leibniz' rule and the fact that \(\chi\) is fixed,
\[
\|\chi g(\cdot,x)\|_{H^l_\beta(\mathbb R^k)}
\lesssim_{K,U_0,l}
\sum_{|\alpha|\le l}
\|\nabla_\beta^\alpha g(\cdot,x)\|_{L^2_\beta(U_0)}.
\]
Therefore
\[
|g(h(x),x)|
\lesssim_{K,U_0,l}
\sum_{|\alpha|\le l}
\|\nabla_\beta^\alpha g(\cdot,x)\|_{L^2_\beta(U_0)}.
\]
Squaring, integrating in \(x\), and using that the number of multiindices
\(|\alpha|\le l\) is finite, we obtain
\[
\begin{aligned}
\|g(h(\cdot),\cdot)\|_{L^2_x}^2
&\lesssim_{K,U_0,l}
\sum_{|\alpha|\le l}
\int_{\mathbb R^d}
\|\nabla_\beta^\alpha g(\cdot,x)\|_{L^2_\beta(U_0)}^2\,dx
\\
&\lesssim_{K,U_0,l}
\sum_{|\alpha|\le l}
\|\nabla_\beta^\alpha g\|_{L^2_\beta(U_0;L^2_x)}^2 .
\end{aligned}
\]
Taking square roots gives \eqref{param-selection-L2}. Since \(U_0\) has
finite measure,
\[
\|\nabla_\beta^\alpha g\|_{L^2_\beta(U_0;L^2_x)}
\le
|U_0|^{1/2}
\sup_{\beta\in U_0}
\|\nabla_\beta^\alpha g(\beta,\cdot)\|_{L^2_x},
\]
which gives \eqref{param-selection-L2-sup}. The proof is complete.
\end{proof}

				\end{document}